\documentclass[10pt]{siamltex}
\usepackage{amsmath,amsfonts,amscd,amssymb,bm,cite}
\usepackage{mathrsfs,listings,url}
\usepackage{graphics,color,ulem}
\usepackage{float}
\usepackage[titletoc,page]{appendix}
\usepackage{subfigure}
\usepackage[colorlinks=true, bookmarksopen,
pdfauthor={CST},
pdfcreator={pdftex},
pdfsubject={algorithms},
linkcolor={blue},
anchorcolor={black},
citecolor={firebrick},
filecolor={magenta},
menucolor={black},
pagecolor={red},
plainpages=false,pdfpagelabels,
urlcolor={db} ]{hyperref}
\usepackage[capitalise]{cleveref}
\usepackage{comment}
\usepackage{verbatim}
\usepackage{marginnote}
\usepackage{float}
\usepackage[makeroom]{cancel}
\usepackage[pdftex]{graphicx}
\usepackage[section]{placeins}
\usepackage{mathptmx}
\usepackage{cases}
\usepackage{subeqnarray}
\usepackage{hhline}
\usepackage{xcolor}
\usepackage{tikz}
\usepackage{caption} 
\usepackage{lmodern}
\usepackage{algorithm}
\usepackage{algpseudocode}

\usetikzlibrary{arrows}
\usetikzlibrary{positioning,shapes,snakes,calc,decorations,decorations.markings}

\newtheorem{sche}{Scheme}

\includecomment{confidential}
\excludecomment{confidential}

\restylefloat{table}
\allowdisplaybreaks
\definecolor{db}{rgb}{0.0470,0,0.5294}
\definecolor{dg}{rgb}{0.0,0.392,0.0}
\definecolor{firebrick}{rgb}{0.698,0.133,0.133}
\definecolor{bl}{rgb}{0.0,0.0,0.0}
\definecolor{linen}{rgb}{0.980,0.941,0.902}
\definecolor{ivory}{rgb}{1.0,1.0,0.941}
\definecolor{aliceblue}{rgb}{0.941,0.973,1.0}
\definecolor{beige}{rgb}{0.961,0.961,0.863}
\definecolor{tan}{rgb}{0.824,0.706,0.549}
\definecolor{lightsteelblue}{rgb}{0.690,0.769,0.871}
\definecolor{paleturquoise}{rgb}{0.686,0.933,0.933}
\definecolor{lightblue}{rgb}{0.678,0.847,0.902}
\definecolor{skyblue}{rgb}{0.529,0.808,0.922}
\definecolor{palegoldenrod}{rgb}{0.933,0.910,0.667}
\definecolor{lightgoldenrod}{rgb}{0.933,0.867,0.510}
\definecolor{lightyellow}{rgb}{1.0,1.0,0.878}
\definecolor{yellow}{rgb}{1.0,1.0,0.0}
\definecolor{lightyellow1}{rgb}{1.0,1.0,0.878}
\definecolor{lemonchiffon}{rgb}{1.0,0.980,0.804}
\definecolor{myyellow}{rgb}{1,1,.9}
\definecolor{darkgreen}{rgb}{0.0,0.392,0.0}
\definecolor{darkviolet}{rgb}{0.580,0.0,0.827}
\definecolor{lightsalmon}{rgb}{1.0,0.627,0.478}
\definecolor{orange}{rgb}{1.0,0.647,0.0}
\definecolor{darkblue}{rgb}{0.00,0.00,0.55}
\definecolor{mygrey}{rgb}{.412,.428,.444}
\numberwithin{equation}{section}
\Crefname{table}{Table}{Tables}
\Crefname{figure}{Figure}{Figures}

\begin{document}

	\title{\large{L\lowercase{ong-time} $L^2 \& H^1$\lowercase{-stability of the} f\lowercase{amily of} DLN m\lowercase{ethods for the} T\lowercase{wo-dimensional} i\lowercase{ncompressible} N\lowercase{avier}-S\lowercase{tokes} e\lowercase{quations}}}
	\author{
		Isabel Barrio Sanchez\thanks{
			Department of Mathematics, University of Pittsburgh, Pittsburgh, PA 15260,
			USA. Email: \href{mailto:ISB42@pitt.edu}{ISB42@pitt.edu}. 
			Partially supported by NSF grant DMS-2208220.}
		\and 
		Wenlong Pei\thanks{
			Department of Mathematics, University of Pittsburgh, Pittsburgh, PA 15260,
			USA. Email: \href{mailto:wep17@pitt.edu}{wep17@pitt.edu}.} 
		\and
		Catalin Trenchea\thanks{
			Department of Mathematics, University of Pittsburgh, Pittsburgh, PA 15260,
			USA. Email: \href{mailto:trenchea@pitt.edu}{trenchea@pitt.edu}. 
			Partially supported by NSF grant DMS-2208220. 
	}}
	\date{\emty}
	\maketitle

	\begin{abstract}		
		In this report, we study the long-time stability of the family of one-leg DLN methods for the two-dimensional incompressible Navier-Stokes equations. 
		The family of DLN methods (with one parameter $\theta$), non-linear energy stable ($G$-stable) and second-order accurate under arbitrary time grids, has been widely applied to the simulations of various fluid models with success. 
		We derive a new version of the $G$-stability identity for the family of DLN methods under uniform time grids and mild time constraints. 
		Then we utilize this crucial auxiliary tool and the discrete uniform Gr\"onwall inequality lemma to prove the uniform-in-time stability 
		of the numerical solutions. 
		Essentially, the bounds are independent of the time interval and the initial conditions, consistent with the theories of the continuous case. 
	\end{abstract}

	
	
	\section{Introduction}  
		The incompressible Navier-Stokes equations (NSE), which describe the relationship between velocity, pressure, and body force in viscous fluid flows, have been applied to numerous fields, including weather prediction, climate modeling, blood flow simulation, and oceanic fluid dynamics \cite{FQV09_Springer,Gun03_SIAM,Tem77_NHPC,Val06_CUP,War11_CUP,WP05_USB}. 
		Herein, we consider the imcompresible NSE on the domain $\Omega \subset \mathbb{R}^2$ with boundary $\partial\Omega$ of class $C^2$
		\begin{align}
			& 
			\label{eq:NSE}
			\tag{NSE}
			\begin{array}{ll}
				u_t + (u\cdot \nabla) u - \nu \Delta u + \nabla p = f,
				\\
				\nabla \cdot u = 0,
			\end{array}
		\end{align}
		where $u=(u_1,u_2)$ is the velocity, $p$ the pressure, $\nu$ the kinematics viscosity, and $f \in L^{\infty}\big(\mathbb{R}_+; (L^{2}(\Omega))^2 \big)$ body forces applied to the fluid.
		The equations are supplemented with the initial condition $u(x,0) = u_0(x) \in (H_{0}^{1}(\Omega))^{2}$ and non-slip boundary condition $u |_{\partial \Omega} = 0$. 
		As known to all, the model energy in \eqref{eq:NSE} is uniformly bounded in time \cite{FMRT01_CUP,Tem97_Springer},
		\begin{align}
			\| u(t) \|^2 \leq \| u_{0} \|^2 e^{-\nu \lambda_1 t} + \frac{1}{\nu^2 \lambda_1^2} (1 - e^{-\nu \lambda_1 t}) \| f \|_{L^{\infty}(\mathbb{R}_+; L^{2}(\Omega)^2)}^2,
			\label{eq:u-L2-bound}
		\end{align}
		where $\| \cdot \|$ is the $L^2(\Omega)$-norm and $\lambda_1$ is the smallest eigenvalue of the Stokes operator. 
		Moreover, $u$ can be uniformly bounded in $H^1$-norm by a function of the initial condition $u_0$ and body force $f$
		\begin{align}
			\| \nabla u(t) \|^2 \leq K\big( \| \nabla u_0\|, \| f \|_{L^{\infty}(\mathbb{R}_+; (L^{2}(\Omega))^2)} \big). 
			\label{eq:u-H1-bound}
		\end{align}
		As the time reaches a certain level $T_{\ast}(\| \nabla u_0\|, \| f \|_{L^{\infty}(\mathbb{R}_+; L^{2}(\Omega)^2)})$, 
		the dependence of bounds on the initial condition $u_0$ can be removed \cite{Tem97_Springer}, i.e. for $t \geq T_{\ast}$,
		\begin{align}
			\| u(t) \|_{{H^1(\Omega)}^2}^2 \leq K\big( \| f \|_{L^{\infty}(\mathbb{R}_+; L^{2}(\Omega)^2)} \big),
			\label{eq:u-H1-bound-2}
		\end{align}
		which implies the dependence of initial conditions is transient due to the existence of a continuous global attractor. 
		
		For reliable simulation in practice, it's essential to select numerical algorithms that possess the behavior of their continuous counterparts. 
		Inspired by the work of Tone and Wirosoetisno for the fully implicit backward Euler time discretization \cite{TW06_SIAMNA}, great efforts for both semi-discretization in time \cite{GTWWW12_SIAMNA,RT23_AML,STP26_arXiv,WWZ26_arXiv,Wan10_MC,Wan12_NM} and fully-discretized approximation \cite{AKR17_NMPDE,ART15_AML,BS19_AMC,CSZ18_JSC,CW16_SIAMNA,HOR17_NM,She90_AA,Ton07_NMPDE,XW22_Entropy,ZY22_JSC} have been devoted to the proof of this significant result. 
		
		The family of methods of Dahlquist, Liniger and Nevanlinna \cite{DLN83_SIAMNA} with one parameter $\theta \in [0,1]$ (herein the DLN methods), known for non-linear stable ($G$-stable) and second order accurate under arbitrary time grids, has demonstrated its success in approximations of typical dynamic systems \cite{LPT23_ACSE} and various fluid models such as the unsteady Stokes-Darcy model \cite{QHPL21_JCAM,QCWLL23_ANM}, incompressible NSE \cite{LPQT21_NMPDE,Pei24_NA,Pei25_NA}, Smagrinsky model \cite{SP24_IJNAM}, phase field model \cite{CLPX25_JSC} and the fourth-order active fluids \cite{ZGPZ26_JSC,ZGPZ26_ANM}. 
		We derive a new version of $G$-stability identity for the family of the constant time-stepping DLN methods (with $\theta \in (0,1)$), which is a pivotal auxiliary tool for the proof of long-time stability results in \eqref{eq:u-L2-bound}-\eqref{eq:u-H1-bound-2}. 
		Hence, we employ the family of fully implicit DLN methods (with $\theta \in (0,1)$) for the two-dimensional incompressible NSE under uniform time grids and prove that the uniform-in-time bounds for the solutions are irrelevant to initial conditions, using only the semi-discretization in time to avoid the Courant-Friedrichs-Lewy (CFL) condition that most fully discrete formulations incur.
		
		Given the uniform time grids $\{t_n\}$ with the constant step size $\Delta t = t_n - t_{n-1}$, 
		the family of DLN schemes for \eqref{eq:NSE} (with parameter $\theta \in (0,1)$) for the initial value problem of form $y'(t) = g(t,y(t)), \ y(0) = y_0$ reads
		\begin{align}  \label{eq:1legDLN}
			\sum_{\ell =0}^{2}{\alpha _{\ell }}y_{n-1+\ell } 
			= \Delta t 
			g\Big( \sum_{\ell =0}^{2}{\beta _{\ell}\, }t_{n-1+\ell } , \sum_{\ell =0}^{2}{\beta _{\ell }\,}y_{n-1+\ell} \Big) , \qquad n=1,\ldots		 
			\tag{DLN}
		\end{align}%
		where $\Delta t$ is the time step, $y_n$ represents the DLN solution at $t_n$, and the $\{\alpha _{\ell },\beta _{\ell }\}_{\ell =0,1,2}$
		coefficients in \eqref{eq:1legDLN} are
		\begin{align*}
			\left(
			\begin{array}{ll}
				\alpha _{2} & \beta _{2}
				\vspace{0.2cm} \\
				\alpha _{1} & \beta _{1}
				\vspace{0.2cm} \\
				\alpha _{0} & \beta _{0}
			\end{array}%
			\right) =\left(
			\begin{array}{lll}
				\frac{1}{2}(\theta +1) &  & \frac{1}{4} (2  + \theta - {\theta }^{2} )\vspace{0.2cm%
				} \\
				-\theta  &  & \frac{1}{2} {\theta }^{2} \vspace{0.2cm} \\
				\frac{1}{2}(\theta -1) &  & \frac{1}{4}  ( 2  - \theta - {\theta }^{2} )%
			\end{array}%
			\right).		
		\end{align*}
		For any sequence $\{z_n\}_{n}^{\infty}$, we denote $z_{n,\beta} = \sum_{\ell=0}^2 \beta_{\ell} z_{n-1+\ell}$ for convenience. 
		\begin{sche}[Fully-implicit DLN algorithms for \eqref{eq:NSE}]
			\label{sche:DLN-NSE}
			Given two previous solutions of velocity and pressue $\{u_n, p_{n} \}$, $\{u_{n-1}, p_{n-1} \}$, we solve $u_{n+1}$ and $p_{n+1}$ by 
			\begin{align}
				\label{eq:DLN-NSE}
				\begin{split}
					\begin{cases}
						\displaystyle \frac{1}{\Delta t} \sum_{\ell=0}^2 \alpha_{\ell} u_{n-1+\ell} - \nu \Delta u_{n,\beta} + u_{n,\beta} \cdot \nabla u_{n,\beta} + \nabla p_{n,\beta} = f(t_{n,\beta}), \\
						\displaystyle \nabla \cdot u_{n,\beta} = 0.
					\end{cases}
				\end{split}
			\end{align}
		\end{sche}

		For the remainder of the report, we mainly focus on the proof of long-time stability of \Cref{sche:DLN-NSE} in $L^2(\Omega)$ and $H^1(\Omega)$ norm, showing that the discrete counterparts of \eqref{eq:u-L2-bound}-\eqref{eq:u-H1-bound-2} hold.
		The report is organized as follows. 
		In Section \ref{sec:note}, we present notations and necessary preliminaries, including a new energy identity as an essential tool for the deviation of the discrete global attractor. 
		Rigorous proof of long-time stability of \Cref{sche:DLN-NSE} in $L^2(\Omega)$- and $H^1(\Omega)$-norm are provided in Section \ref{sec:stability-L2} and Section \ref{sec:stability-H1} respectively. 
		Section \ref{sec:conclusion} summarizes the main results.

		\section{Notations and preliminaries} \label{sec:note}
		Given the domain $\Omega \subset \mathbb{R}^2$, we need the following Sobolev spaces for mathematical setting of the problem
		\begin{align*}
			&V = \{ v \in H_{0}^1(\Omega)^2: \nabla v = 0 \}, \\
			&H = \{ v \in L^2(\Omega)^2: \nabla \cdot v = 0, v \cdot \vec{n} |_{\partial \Omega} = 0 \}, 
		\end{align*}
		where $\vec{n}$ is the unit vector outward on $\partial \Omega$. 
		The space $H$ is endowed with the $L^2(\Omega)$ inner product $(u, v) = \int_{\Omega} u \cdot v dx$ 
		and the corresponding norm $\| u \| = (u,u)^{1/2}$ for $u,v \in H$.   
		The space $V$ is equipped with the inner product $(\nabla u, \nabla v)$
		\begin{align*}
			(\nabla u, \nabla v) = \int_{\Omega} \sum_{i,j=1}^2 \frac{\partial u}{\partial x_i} (x) \frac{\partial v}{\partial x_i} (x) dx, 
		\end{align*}
		and the corrsponding norm $\| \nabla u \| = ( \nabla u, \nabla v)^{1/2}$ for $u,v \in V$.  
		Throughout the whole paper, we assume that the body force $f \in L^{\infty}(\mathbb{R}_{+};H)$ and denote $\| f \|_{\infty} = \| f \|_{L^{\infty}(\mathbb{R}_{+};H)}$.
		We define the bounded linear operator $A$ from $V$ into its dual space $V'$ by 
		\begin{align*}
			\langle Au, v \rangle_{V',V} = (\nabla u, \nabla v) 
		\end{align*} 
		By the regularity theory for the Stokes equation \cite{Tem77_NHPC}, the domain of $A$ on space $H$ is $D(A) = {H^2(\Omega)}^2 \cap V$. Obviously, we have the inclusions $D(A) \subset V \subset H$. 
		We denote the smallest eigenvalue of $A$ as $\lambda_1 >0$ and have Poincar\'e-Friedrichs inequality 
		\begin{align}
			\label{eq:PF-ineq} 
			\| u \| \leq \frac{1}{\sqrt{\lambda_1}} \| \nabla u \|, \qquad u \in V, 
		\end{align}
		and Ladyzhenskaya inequality in two spatial dimensions \cite{Gal11_Springer,Lad58_SPD,Lad59_CPAM}
		\begin{align}
			\label{eq:Lady-ineq}
			\| u \|_{L^{4}(\Omega)} \leq 2^{-1/4} \| u \|^{1/2} \| \nabla u \|^{1/2}, \qquad u \in V. 
		\end{align}
		We need the following $H^2$-bound for the Laplacian \cite[pp.7]{Lad69_GBSP}
		\begin{align}
			\| D^2 u \| \leq C_{\Omega} \| \Delta u \|, \qquad u \in H^2(\Omega)^2 \cap H_{0}^{1}(\Omega)^2, 
			\label{eq:equi-D2-ineq} 
		\end{align}
		where $C_{\Omega}$ is the positive constant only depending on $\Omega$ and $C_{\Omega} = 1$ if $\Omega$ is convex \cite[pp.132]{Gri85_Pitman}. 

		We denote the trilinear operator $b(u,v,w) = (u \cdot \nabla v, w)$ for any $u,v,w \in V$. 
		By Gauss divergence theorem, $b$ is skew-symmetric:
		\begin{align}
			b(u,v,w) = - b(u,w,v), \qquad b(u,v,v) = 0, \quad \forall u,v,w \in V. 
			\label{eq:b-skew-symm}
		\end{align}
		By H\"older's inequality, Ladyzhenskaya inequality in \eqref{eq:Lady-ineq} and Agmon inequality, it's easy to derive the following estimates for the operator $b$ 
		\begin{align}
			&|b(u,v,w)| \leq C_{b,2} \| u \|^{1/2} \| \nabla u \|^{1/2} \| \nabla v \| \| w \|^{1/2} \| \nabla w \|^{1/2}, 
			\quad \forall u,v,w \in V.  \label{eq:b-bound-2}
		\end{align}%
\begin{confidential}
	\color{darkblue}
	For $u \in D(A)$
	\begin{align*}
		|\langle Au, v \rangle_{V',V}| = |(\nabla u, \nabla v)| = |- (\Delta u, v) | \leq \| \Delta u \| \| v \|,
	\end{align*}
	Thus $Au \in V'$ with $(V, \| \cdot \|)$ and 
	there exists $h \in L^{2}(\Omega)$ such that $Au = h\in V \subset H$ by Riesz representation theorem. 
	For $u \in H^2(\Omega) \cap H_{0}^1(\Omega)$, by Agmon inequality
	\begin{align*}
		\| u \|_{L^{\infty}(\Omega)} \leq C \| u \|^{1/2} \| u \|_{H^2(\Omega)}^{1/2}. 
	\end{align*} 
	By H\"older's inequality and Agmon inequality
	\begin{align*}
		| b(u,v,w)| \leq \| u \|_{L^{\infty}(\Omega)} \| \nabla v \| \| w \| 
		\leq C \| u \|^{1/2} \| u \|_{H^2(\Omega)}^{1/2} \| \nabla v \| \| w \|
	\end{align*}
	By Stokes regular estimate, we have $\| u \|_{H^2(\Omega)} \leq C \| Au \|$ and thus we have first bound. 
	For second bound, by H\"older's inequality, Ladyzhenskaya inequality in \eqref{eq:Lady-ineq}
	\begin{align*}
		|b(u,v,w)| \leq \| u \|_{L^{4}} \| \nabla v \| \| w \|_{L^4} 
		\leq C \| u \|^{1/2} \| \nabla u \|^{1/2} \| \nabla v \| \| w \|^{1/2} \| \nabla w \|^{1/2}. 
	\end{align*}
	\normalcolor
\end{confidential}
		By \eqref{eq:b-bound-2}, we can define the bilinear operator $B$ from $V \times V$ into $V'$ by 
		\begin{align*}
			\langle B(u,v), w \rangle_{V',V} = b(u,v,w), \qquad \forall u,v,w \in V. 
		\end{align*}
		By the above operators we've introduced, \eqref{eq:NSE} can be written as \cite{Ler33_JMPA}
		\begin{align*}
			u_t + \nu Au + B(u,u) = f, \quad u(0) = u_0,
		\end{align*}
		and \Cref{sche:DLN-NSE} becomes 
		\begin{align}
			\frac{1}{\Delta t} \sum_{\ell=0}^2 \alpha_{\ell} u_{n-1+\ell} + \nu A u_{n,\beta} + B(u_{n,\beta}, u_{n,\beta}) = f(t_{n,\beta}). 
			\label{eq:DLN-NSE-dual}
		\end{align}

		\begin{definition}
			For $\theta \in (0,1)$, we define the symmetric semi-positive definite $G(\theta)$-matrix
			\begin{align*}
				G(\theta) = 
				\begin{pmatrix}
					\vspace{0.8mm}
					\frac{1}{4} (1 + \theta) \mathbb{I}_2 & 0 \\
					0 & \frac{1}{4} (1 - \theta) \mathbb{I}_2
				\end{pmatrix},
			\end{align*}
			where $\mathbb{I}_2 \in \mathbb{R}^{2 \times 2}$ is the identity matrix. The corresponding $G$-norm is defined as 
			\begin{align*}
				\begin{Vmatrix}
					u \\ v
				\end{Vmatrix}_{G(\theta)}^2
				= 
				[u^\top \ v^\top]  G(\theta) 
				\begin{bmatrix}
					u \\ v
				\end{bmatrix}
				= \frac{1}{4} (1 + \theta) \| u \|^2 + \frac{1}{4}(1 - \theta) \| v \|^{2}, \ \ \forall u, v \in L^{2}(\Omega).
			\end{align*}
    	\end{definition}
		\begin{lemma}
			Given the sequence $\{ y_{n} \}_{n=0}^{\infty} \subset L^{2}(\Omega)$ and $\theta \in (0,1)$ the following $G$-stability identity holds 
			\begin{align}
				\Big( \sum_{\ell=0}^2 \alpha_{\ell} y_{n-1+\ell}, y_{n,\beta} \Big)
				= 
				\begin{Vmatrix}
					y_{n+1} \\
					y_{n}
				\end{Vmatrix}_{G(\theta)}^2
				- 
				\begin{Vmatrix}
					y_{n} \\
					y_{n-1}
				\end{Vmatrix}_{G(\theta)}^2
				+ \Big\| \sum_{\ell}^{2} a_{\ell} y_{n-1+\ell} \Big\|^{2},
				\label{eq:G-stab-eq}
			\end{align}
			where the coefficients of numerical dissipation are 
			\begin{align*}
				a_1 = - \frac{\sqrt{\theta (1 - \theta^2)}}{\sqrt{2}}, \quad 
				a_2 = - \frac{1}{2} a_1, \quad a_0 = - \frac{1}{2} a_1. 
			\end{align*} 
		\end{lemma}
		\begin{proof}
			See \cite{LPQT21_NMPDE,LPT21_AML}. 
		\end{proof}

		\begin{lemma}
			For any $n \geq 1$ and $\theta \in (0,1)$, the approximations of \Cref{sche:DLN-NSE} satisfy
			\begin{align}
				\label{eq:stab-eq1} 
				&\begin{Vmatrix}
						u_{n+1} \\
						u_{n}
				\end{Vmatrix}_{G(\theta)}^2
				- 
				\begin{Vmatrix}
					u_{n} \\
					u_{n-1}
				\end{Vmatrix}_{G(\theta)}^2
				+ \Big\| \sum_{\ell}^{2} a_{\ell} u_{n-1+\ell} \Big\|^{2}
				+ \frac{\nu \Delta t \lambda_1}{2} \| u_{n,\beta} \|^2 
				\leq \frac{\Delta t \| f \|_{\infty}^2}{2 \nu \lambda_1}.  
			\end{align}
		\end{lemma}
		\begin{proof}
			We test \eqref{eq:DLN-NSE-dual} with $\Delta t u_{n,\beta}$, use 
			$G$-stability identity in \eqref{eq:G-stab-eq} and the skew-symmetric property of $b$ in \eqref{eq:b-skew-symm} to achieve
			\begin{align*}
				&\begin{Vmatrix}
						u_{n+1} \\
						u_{n}
				\end{Vmatrix}_{G(\theta)}^2
				- 
				\begin{Vmatrix}
					u_{n} \\
					u_{n-1}
				\end{Vmatrix}_{G(\theta)}^2
				+ \Big\| \sum_{\ell}^{2} a_{\ell} u_{n-1+\ell} \Big\|^{2}
				+ \nu \Delta t \| \nabla u_{n,\beta} \|^2 
				\leq \Delta t \big( f(t_{n,\beta}), u_{n,\beta} \big).
			\end{align*}
			By Cauchy-Schwarz inequality, Young's inequality and Poincar\'e-Friedrichs inequality in \eqref{eq:PF-ineq}
\begin{confidential}
	\color{darkblue}
	\begin{align*}
		&\begin{Vmatrix}
			u_{n+1} \\
			u_{n}
		\end{Vmatrix}_{G(\theta)}^2
		- 
		\begin{Vmatrix}
			u_{n} \\
			u_{n-1}
		\end{Vmatrix}_{G(\theta)}^2
		+ \Big\| \sum_{\ell}^{2} a_{\ell} u_{n-1+\ell} \Big\|^{2}
		+ \nu \Delta t \| \nabla u_{n,\beta} \|^2 \\
		&\leq \Delta t \| f(t_{n,\beta}) \| \| u_{n,\beta} \| 
		\leq \frac{\Delta t}{\sqrt{\nu \lambda_1}} \| f(t_{n,\beta}) \| \sqrt{\nu} \| \nabla u_{n,\beta} \|
		\leq 
		\frac{\Delta t}{2 \nu \lambda_1} \| f(t_{n,\beta}) \|^2 + \frac{\Delta t \nu}{2} 
		\| \nabla u_{n,\beta} \|^2 
	\end{align*}
	\normalcolor
\end{confidential}
			\begin{align}
				\label{eq:GstabNSE1}
				&\begin{Vmatrix}
					u_{n+1} \\
					u_{n}
				\end{Vmatrix}_{G(\theta)}^2
				- 
				\begin{Vmatrix}
					u_{n} \\
					u_{n-1}
				\end{Vmatrix}_{G(\theta)}^2
				+ \Big\| \sum_{\ell}^{2} a_{\ell} u_{n-1+\ell} \Big\|^{2}
				+ \frac{\nu \Delta t}{2} \| \nabla u_{n,\beta} \|^2 
				\leq \frac{\Delta t}{2 \nu \lambda_1} \| f(t_{n,\beta}) \|^2. 
			\end{align}
			We apply Poincar\'e-Friedrichs inequality in \eqref{eq:PF-ineq} to the above inequality and achieve \eqref{eq:stab-eq1}. 
		\end{proof}
		
		\begin{lemma} \label{lem:H-stab}
			If the time step $\Delta t$ satisfies 
			\begin{align}
				\label{eq:dt-limit-1}
				\Delta t < \frac{1}{\nu \lambda_1} \min \Big\{ 
				\frac{8 \theta (1 - \theta^2)}{( 8 - 6 \theta^2 + 3 \theta^4)}, 2(1-\theta) \Big\}:= C_{\Delta t},
			\end{align}
			there exist $\epsilon>0$, $a, b, c \in \mathbb{R}$, and the positive definite matrix 
			\begin{align}  
				\label{eq:Hmatrix}
				H(\theta) = \left(
				\begin{array}{cc}
					h_{11}(\theta) \mathbb{I}_d & 0 \vspace{.2cm} \\
					0 & h_{22}(\theta) \mathbb{I}_d%
				\end{array}
				\right), 			
			\end{align} 
			such that the left hand side of \eqref{eq:stab-eq1} can be written as 
			\begin{align}
				\label{eq:H-stab}
				& 
				\begin{Vmatrix}
					{u_{n+1}} \\
					{u_{n}}%
				\end{Vmatrix}%
				_{G(\theta )}^{2} 
				- 
				\begin{Vmatrix}
					{u_{n}} \\
					{u_{n-1}}%
				\end{Vmatrix}%
				_{G(\theta )}^{2}
				+ \Big\|\sum_{\ell =0}^{2} a_{\ell} u_{n-1+\ell} \Big\|^{2}
				+ \frac{\nu \Delta t \lambda_1}{2} \| u_{n,\beta} \|^2 \\
				& 
				= (1+\epsilon)  
				\begin{Vmatrix}
					{u_{n+1}} \\
					{u_{n}}%
				\end{Vmatrix}%
				_{H(\theta )}^{2}
				- 
				\begin{Vmatrix}
					{u_{n}} \\
					{u_{n-1}}%
				\end{Vmatrix}%
				_{H(\theta )}^{2}
				+ \big\|  a u_{n+1} + b u_n + cu _{n-1}\big\|^2,
				\notag
			\end{align}
			where the $H$-norm associated with the $H(\theta)$ matrix in \eqref{eq:Hmatrix} is defined as 
			\begin{align}  \label{eq:H-norm}
				\begin{Vmatrix}
					u \\
					v\end{Vmatrix}_{H(\theta)}^{2}
				= 
				[u^\top \ v^\top] H(\theta)
				\begin{bmatrix}
					u \\  v
				\end{bmatrix}
				= h_{11}(\theta) \| u \|^{2}_{{\mathbb{R}}^{d}} 
				+ h_{22}(\theta) \| v \|^2_{{\mathbb{R}}^{d}},
				\ \forall u,v\in \mathbb{R}^{d}.
			\end{align}
			Here 
			\begin{align}
				\label{eq:h-lower-bound}
				&h_{11}(\theta) > \frac{\theta^3}{2} \Big[ 1 - \frac{1}{2} \theta (1 - \theta) \Big] > 0, \qquad
				h_{22}(\theta) > \frac{1}{16} \theta(1 - \theta)^2 (1+\theta) (2+\theta) >0,
			\end{align}
			and 
			\begin{align}
				&\max\{ h_{11}(\theta), h_{22}(\theta) \} < C_{h}(\theta), \qquad 
				\frac{1}{4} < \frac{1}{\epsilon} < \frac{C_{\epsilon}(\theta)}{\nu \lambda_1 \Delta t}, 
				\label{eq:1-epsi-cond}
			\end{align}
			for some functions $C_{h}(\theta)$, $C_{\epsilon}(\theta) >0$ only depending on $\theta$. 
		\end{lemma}
		\begin{proof}
			See Appendix \ref{sec:appendix}.
		\end{proof}

        \begin{lemma}
            Given the sequence $\{ y_n\}_{n=1}^{\infty} \subset L^{2}(\Omega)$ and $\theta \in (0,1)$, the following two identities hold  
            \begin{align}
                &\Big(\sum_{\ell=0}^{2} \alpha_{\ell} y_{n-1+\ell}, y_{n+1} \Big) \label{eq:identity-1}\\
                &= 
                \begin{Vmatrix}
                    {y_{n+1}} \\
                    {y_{n}}%
                \end{Vmatrix}_{G(\theta )}^{2}
                - 
                \begin{Vmatrix}
                    {y_{n}} \\
                    {y_{n-1}}%
                \end{Vmatrix}_{G(\theta )}^{2}
                + \frac{\theta}{2}  \| y_{n+1} -  y_{n} \|^2
                + \frac{1}{4} ( 1 - \theta ) \| y_{n+1} -  y_{n-1} \|^2, \notag \\
                &\big( y_{n,\beta},  y_{n+1} \big) \label{eq:identity-2} \\
                &= (2\beta_2 -1) \| y_{n+1} \|^2 
                + \frac{\beta_1}{2} \| y_{n+1} + y_{n} \|^2
                + \frac{\beta_0}{2} \| y_{n+1} + y_{n-1} \|^2
                \notag \\ 
                &+
                \Big( \frac{\beta_0+\beta_1}{2} \| y_{n+1} \|^2
                + \frac{\beta_0}{2}  \| y_{n} \|^2 
                \Big)	
                - 
                \Big( \frac{\beta_0+\beta_1}{2} \| y_{n} \|^2
                + \frac{\beta_0}{2}  \| y_{n-1} \|^2 
                \Big). \notag
            \end{align}
        \end{lemma}
		\begin{proof}
            Just algebraic calculation. 
        \end{proof}

\begin{confidential}
	\color{darkblue}
    \begin{align*}
        \frac{1}{2} \big[ a^2 + b^2 - (a - b)^2 \big]
        = \frac{1}{2} \big[ a^2 + b^2 - (a^2 + b^2 - 2ab) \big]
        = \frac{1}{2} \big[ \cancel{a^2} \bcancel{+ b^2} \cancel{- a^2} \bcancel{- b^2} + 2ab \big] = ab. 
    \end{align*}
	\begin{align*}
		& 
		\big({\alpha _{2}}{y_{n+1}}+{\alpha _{1}}{y_{n}}+{\alpha _{0}}{y_{n-1}}  , y_{n+1} \big)
		= \alpha _{2} \|y_{n+1}\|^2 + \alpha _{1} ( {y_{n}} , y_{n+1} )  + \alpha _{0} ( y_{n-1} , y_{n+1} ) \\
		& 
		= \frac{1}{2}(\theta+1) \|y_{n+1}\|^2 
		- \theta ( {y_{n}} , y_{n+1} )  
		+ \frac{1}{2} (\theta-1) ( y_{n-1} , y_{n+1} ) \\
		& 
		= \frac{1}{2}(\theta+1) \|y_{n+1}\|^2 
		- \frac{\theta}{2} \big( \|y_{n+1}\|^2 + \| y_{n} \|^2 - \| y_{n+1} -  y_{n} \|^2 \big)  \\
		& 
		\qquad
		+ \frac{1}{2} (\theta-1) \Big( \frac{1}{2} \|y_{n+1}\|^2 + \frac{1}{2} \| y_{n-1} \|^2  -  \frac{1}{2} \| y_{n+1} - y_{n-1} \|^2\big) \\
		& 
		= \Big[ \frac{1}{2}(\cancel{\theta}+1) - \cancel{\frac{\theta}{2}} + \frac{1}{4} (\theta-1) \Big] \|y_{n+1}\|^2 
		- \frac{\theta}{2} \| y_{n} \|^2 
		- \frac{1}{4} (1 - \theta) \| y_{n-1} \|^2 \\
		& 
		\qquad
		+ \frac{\theta}{2}  \| y_{n+1} -  y_{n} \|^2
		+ \frac{1}{4} ( 1 - \theta ) \| y_{n+1} -  y_{n-1} \|^2 \\
		& 
		= \frac{1+\theta}{4}  \|y_{n+1}\|^2 
		- \frac{\theta}{2} \| y_{n} \|^2 
		- \frac{1}{4} (1 - \theta) \| y_{n-1} \|^2 \\
		& 
		\qquad
		+ \frac{\theta}{2}  \| y_{n+1} -  y_{n} \|^2
		+ \frac{1}{4} ( 1 - \theta ) \| y_{n+1} -  y_{n-1} \|^2 \\
		& 
		= \Big( \frac{1+\theta}{4}  \|y_{n+1}\|^2  + \frac{1-\theta}{4} \| y_{n} \|^2  \Big)
		- \Big( \frac{1+\theta}{4} \| y_{n} \|^2  +  \frac{1-\theta}{4} \| y_{n-1} \|^2 \Big) \\
		& 
		\qquad
		+ \frac{\theta}{2}  \| y_{n+1} -  y_{n} \|^2
		+ \frac{1}{4} ( 1 - \theta ) \| y_{n+1} -  y_{n-1} \|^2 \\
		& 
		= 
		\begin{Vmatrix}
			{y_{n+1}} \\
			{y_{n}}%
		\end{Vmatrix}_{G(\theta )}^{2}
		- 
		\begin{Vmatrix}
			{y_{n}} \\
			{y_{n-1}}%
		\end{Vmatrix}_{G(\theta )}^{2}
		+ \frac{\theta}{2}  \| y_{n+1} -  y_{n} \|^2
		+ \frac{1}{4} ( 1 - \theta ) \| y_{n+1} -  y_{n-1} \|^2,
	\end{align*}
	\normalcolor
\end{confidential}
\begin{confidential}
	\color{darkblue}  
	\begin{align*}
		( y_{n,\beta }, y_{n+1} )
		&= \beta_2 \| y_{n+1} \|^2 + \beta_1 ( y_{n+1}, y_{n}) 
		+ \beta_0 ( y_{n+1}, y_{n-1}) \\
		&= \beta_2 \| y_{n+1} \|^2 
		+ \beta_1 \frac{1}{2} \Big( \| y_{n+1} + y_{n} \|^2 - \| y_{n+1} \|^2 - \| y_{n} \|^2 \Big) \\
		&\hspace{2.3cm}
		+ \beta_0 \frac{1}{2} \Big( \| y_{n+1} + y_{n-1} \|^2 - \| y_{n+1} \|^2 - \| y_{n-1} \|^2 \Big) \\
		&= \beta_2 \| y_{n+1} \|^2
		+ \frac{\beta_1}{2} \| y_{n+1} + y_{n} \|^2
		- \frac{\beta_1}{2} \| y_{n+1} \|^2 
		- \frac{\beta_1}{2} \|y_{n} \|^2 \\
		&\hspace{2.3cm} 
		+ \frac{\beta_0}{2} \| y_{n+1} + y_{n-1} \|^2
		- \frac{\beta_0}{2} \| y_{n+1} \|^2
		- \frac{\beta_0}{2} \| y_{n-1} \|^2 \\
		&= \Big( \beta_2 - \frac{\beta_1}{2} - \frac{\beta_0}{2}  \Big)
		\| y_{n+1} \|^2
		- \frac{\beta_1}{2}  \| y_{n} \|^2
		- \frac{\beta_0}{2}  \| y_{n-1} \|^2
		\\
		& 
		\hspace{2.3cm}
		+ \frac{\beta_1}{2} \| y_{n+1} + y_{n} \|^2
		+ \frac{\beta_0}{2} \| y_{n+1} + y_{n-1} \|^2
		\\
		& 
		= \Big[ \beta_2 - \frac{1}{2} (\beta_1 + \beta_0) \Big] \| y_{n+1} \|^2
		- \frac{\beta_1}{2}  \|y_{n} \|^2
		- \frac{\beta_0}{2}  \|y_{n-1} \|^2
		\\
		& 
		\hspace{2.3cm}
		+ \frac{\beta_1}{2} \| y_{n+1} + y_{n} \|^2
		+ \frac{\beta_0}{2} \| y_{n+1} + y_{n-1} \|^2
		\\
		& 
		= \big[ \beta_2 - (\beta_1 + \beta_0) \big] \| y_{n+1} \|^2 \\
		&\quad +
		\Big( \frac{\beta_0+\beta_1}{2} \| y_{n+1} \|^2
		+ \frac{\beta_0}{2}  \| y_{n} \|^2 \Big)	
		- 
		\Big( \frac{\beta_0+\beta_1}{2} \| y_{n} \|^2
		+ \frac{\beta_0}{2}  \|y_{n-1} \|^2 \Big)	
		\\
		& 
		\hspace{2.3cm}
		+ \frac{\beta_1}{2} \| y_{n+1} + y_{n} \|^2
		+ \frac{\beta_0}{2} \| y_{n+1} + y_{n-1} \|^2 \\
		& 
		= \big[ \beta_2 - (1 - \beta_2) \big] \| y_{n+1} \|^2 \\
		&\quad +
		\Big( \frac{\beta_0+\beta_1}{2} \| y_{n+1} \|^2
		+ \frac{\beta_0}{2}  \| y_{n} \|^2 \Big)	
		- 
		\Big( \frac{\beta_0+\beta_1}{2} \| y_{n} \|^2
		+ \frac{\beta_0}{2}  \| y_{n-1} \|^2 \Big)	
		\\
		& 
		\hspace{2.3cm}
		+ \frac{\beta_1}{2} \| y_{n+1} + y_{n} \|^2
		+ \frac{\beta_0}{2} \| y_{n+1} + y_{n-1} \|^2 \\
		&= \big( 2 \beta_2 - 1 \big) \| y_{n+1} \|^2 \\
		&\quad +
		\Big( \frac{\beta_0+\beta_1}{2} \| y_{n+1} \|^2
		+ \frac{\beta_0}{2} \| y_{n} \|^2 \Big)	
		- 
		\Big( \frac{\beta_0+\beta_1}{2} \| y_{n} \|^2
		+ \frac{\beta_0}{2}  \| y_{n-1} \|^2 \Big)	
		\\
		& 
		\hspace{2.3cm}
		+ \frac{\beta_1}{2} \| y_{n+1} + y_{n} \|^2
		+ \frac{\beta_0}{2} \| y_{n+1} + y_{n-1} \|^2
	\end{align*}
	\begin{align*}
		2 \beta_2 - 1 
		=& 2 \cdot \frac{1}{4} (2 + \theta - \theta^2) - 1 
		= \frac{1}{2} (2 + \theta - \theta^2) - 1 
		= \cancel{1} + \frac{1}{2} \theta - \frac{1}{2} \theta^2 \cancel{- 1} 
		= \frac{1}{2} \theta (1 - \theta) > 0.
	\end{align*}
	\normalcolor
\end{confidential}

	\section{Long-time stability in $L^2$-norm} \label{sec:stability-L2}
    We make the following notation 
	\begin{align}
        &K_1( \|u_0\|, \|u_0\|, \|f\|_{\infty}) 
		= C_{h}(\theta) \big( \|u_{1}\|^2 + \|u_0\|^2 \big) 
		+ \frac{C_{\epsilon}(\theta)}{2 \nu^2 \lambda_1^2} 
		\|f\|_{\infty}^2, \label{eq:K1} \\
        &K_2(\| u_0 \|, \|u_1\|, \| f \|_{\infty})
        =: \bigg\{ \frac{K_1}{\frac{\theta^3}{2} \Big[ 1 - \frac{1}{2} \theta (1 - \theta) \Big]} \bigg\}^{1/2}, \label{eq:K2} 
	\end{align}
    where $C_{h}(\theta)$ and $C_{\epsilon}(\theta)$ are positive constants in Lemma \ref{lem:H-stab}.
    
    \begin{theorem}
        \label{thm:stab-L2}
        If the time step $\Delta t$ satisfies \eqref{eq:dt-limit-1}, the approximate velocity of \Cref{sche:DLN-NSE} satisfies 
        \begin{align}
            \label{eq:L2bound0}
            & 
            \begin{Vmatrix}
                {u_{N}} \\
                {u_{N-1}}%
            \end{Vmatrix}%
            _{H(\theta )}^{2}
            \leq
            \frac{1}{(1+\epsilon)^{N-1}}
            \begin{Vmatrix}
                {u_{1}} \\
                {u_{0}}%
            \end{Vmatrix}%
            _{H(\theta )}^{2}
            + \frac{C_{\epsilon}(\theta)}{2 \nu^2 \lambda_1^2} 
            \|f\|^2_{\infty}, \quad N = 2, \cdots 
        \end{align}
        where $\| \cdot \|_{H(\theta)}$ is the $H$-norm defined in \eqref{eq:H-norm}. 
        From \eqref{eq:L2bound0}, we have 
		\begin{align}
			\| u_{N} \| \leq K_2, \quad N = 2, \cdots. 
			\label{eq:u-boundL2-K2} 
		\end{align}
		Moreover for any $N \geq 1$, and $i = 1, \cdots,N-1$ 
		\begin{align}
			& 
			\label{eq:stabilityL2-3}
			\nu \Delta t \sum_{n=i}^{N-1} 
			\| \nabla u_{n,\beta} \|^2
			\leq 2
			\begin{Vmatrix}
				{u_{i}} \\
				{u_{i-1}}
			\end{Vmatrix}
			_{G(\theta)}^{2}
			+ \frac{(N-i) \Delta t}{\nu \lambda_1} \|f\|^2_{\infty},
		\end{align}
		and 
		\begin{align}
			&\Delta t \frac{\nu(2\beta_2-1)}{8} \sum_{j=i}^{N-1} \| \nabla u_{j+1} \|^2 \label{eq:u_{n+1}L^2(H^1)} \\
			&+
			\begin{Vmatrix}
				{u_{N}} \\
				{u_{N-1}}
			\end{Vmatrix}_{G(\theta )}^{2}
			+ \Delta t \nu \Big[  \frac{1}{4} \| \nabla u_{N} \|^2
			+ \Big(  \frac{\beta_0}{2}  + \frac{(2\beta_2-1)}{8}  \Big)\|\nabla u_{N-1} \|^2 \Big] \notag \\
			&+ \sum_{j=i}^{N-1} \Big(  \frac{\theta}{2} \| u_{j+1} -  u_{j} \|^2
			+ \frac{1}{4} ( 1 - \theta ) \| u_{j+1} -  u_{j-1} \|^2 \Big) \notag \\
			&+ \nu \Delta t \sum_{j=i}^{N-1}\Big( \frac{\beta_1}{2} \|( \nabla u_{j+1} - \nabla u_{j} )\|^2
			+ \frac{\beta_0}{2} \|( \nabla u_{j+1} - \nabla u_{j-1} )\|^2 \Big) \notag \\
			\leq&
			\Big( 1 + \frac{ K_2^4(\beta_0^2 + \beta_1^2)^2 } {\nu^4 (2\beta_2-1)^3} \Big) 
			\begin{Vmatrix}
				{u_{i}} \\
				{u_{i-1}}
			\end{Vmatrix}_{G(\theta)}^{2} 
			+ \Delta t \nu \Big[  \frac{1}{4} \|\nabla u_{i} \|^2
			+ \Big( \frac{\beta_0}{2}  + \frac{(2\beta_2-1)}{8}  \Big)\|\nabla u_{i-1} \|^2 \Big] \notag \\
			&+ \frac{(N-i) \Delta t}{\nu(2 \beta_{2} - 1) \lambda_1}
			\Big( 2 + \frac{ K_2^4(\beta_0^2 + \beta_1^2)^2 } { 2 \nu^4 (2\beta_2-1)^2} \Big) \| f \|_{\infty}^2 \notag
		\end{align}
    \end{theorem}
    \begin{proof}
        By \eqref{eq:stab-eq1} and \eqref{eq:H-stab}, we drop the numerical dissipation term and use induction to obtain 
\begin{confidential}
    \color{darkblue}
    \begin{align*}
			(1+\varepsilon)  
			\begin{Vmatrix}
				{u_{n+1}} \\
				{u_{n}}%
			\end{Vmatrix}%
			_{H(\theta )}^{2}
			\leq 
			\begin{Vmatrix}
				{u_{n}} \\
				{u_{n-1}}%
			\end{Vmatrix}%
			_{H(\theta )}^{2} + 
			\frac{\Delta t}{2\nu \lambda_1} \|f(t_{n,\beta })\|^2
		\end{align*}
        \begin{align*}
            \begin{Vmatrix}
				{u_{N}} \\
				{u_{N-1}}
			\end{Vmatrix}
			_{H(\theta )}^{2} 
			\leq& \frac{1}{1 + \varepsilon} 
			\begin{Vmatrix}
				{u_{N-1}} \\
				{u_{N-2}}
			\end{Vmatrix}_{H(\theta )}^{2}  
			+ \Big( \frac{1}{1 + \varepsilon} \Big) \frac{\Delta t}{2\nu \lambda_1} \|f\|^2_{\infty} \\
			\leq& \frac{1}{1 + \varepsilon} \Big[ \frac{1}{1 + \varepsilon} 
			\begin{Vmatrix}
				{u_{N-2}} \\
				{u_{N-3}}
			\end{Vmatrix}_{H(\theta )}^{2}
			+ \Big( \frac{1}{1 + \varepsilon} \Big) \frac{\Delta t}{2\nu \lambda_1} \|f\|^2_{\infty} \Big] 
			+ \Big( \frac{1}{1 + \varepsilon} \Big) \frac{\Delta t}{2\nu \lambda_1} \|f\|^2_{\infty} \\
        \end{align*}
    \normalcolor
\end{confidential}
		\begin{align*}
			\begin{Vmatrix}
				{u_{N}} \\
				{u_{N-1}}
			\end{Vmatrix}_{H(\theta)}^{2}  
			\leq& \frac{1}{1 + \varepsilon} 
			\begin{Vmatrix}
				{u_{N-1}} \\
				{u_{N-2}}
			\end{Vmatrix}_{H(\theta )}^{2}
			+ \Big( \frac{1}{1 + \varepsilon} \Big) \frac{\Delta t}{2\nu \lambda_1} \|f\|^2_{\infty} \\
			=& \frac{1}{(1 + \varepsilon)^2}  
			\begin{Vmatrix}
				{u_{N-2}} \\
				{u_{N-3}}
			\end{Vmatrix}_{H(\theta)}^{2}, 
			+ \Big[ \frac{1}{(1 + \varepsilon)^2} + \frac{1}{1 + \varepsilon} \Big] \frac{\Delta t}{2\nu \lambda_1} \|f\|^2_{\infty} \\
			\leq& \cdots \\
			\leq& \frac{1}{(1 + \varepsilon)^{N-1}}  
			\begin{Vmatrix}
				{u_{1}} \\
				{u_{0}}
			\end{Vmatrix}_{H(\theta)}^{2} \!\!
			+ \!\Big[ \frac{1}{(1 + \varepsilon)^{N-1}}+ \cdots + \frac{1}{(1 + \varepsilon)^2} + \frac{1}{1 + \varepsilon} \Big] \frac{\Delta t}{2\nu \lambda_1} \|f\|^2_{\infty} \\
			\leq& \frac{1}{(1 + \varepsilon)^{N-1}}  
			\begin{Vmatrix}
				{u_{1}} \\
				{u_{0}}
			\end{Vmatrix}_{H(\theta)}^{2}
			+ \frac{1}{\varepsilon} \frac{\Delta t}{2\nu \lambda_1} \|f\|^2_{\infty}, \quad \forall N = 2, \cdots.  
		\end{align*}
        which implies \eqref{eq:L2bound0} by \eqref{eq:1-epsi-cond}. 
		We combine \eqref{eq:h-lower-bound} and \eqref{eq:L2bound0} to achieve
        \begin{align*}
            \frac{\theta^3}{2} \Big[ 1 - \frac{1}{2} \theta (1 - \theta) \Big]
            \| u_{N} \|^2  
            \leq 
			\begin{Vmatrix}
				{u_{N}} \\
				{u_{N-1}}
			\end{Vmatrix}_{H(\theta)}^{2} \leq K_1,
        \end{align*}
		where $K_1$ is defined in \eqref{eq:K1}.
        Thus by the notation in \eqref{eq:K2}, we have $\| u_N \|^2 \leq K_2$ for all $N \geq 2$. 
		Adding up \eqref{eq:GstabNSE1} over $n$ from $i$ to $N-1$, and dropping the numerical dissipation term, 
		\begin{align*}
			\frac{\nu \Delta t }{2}
			\sum_{n=i}^{N-1}
			\| \nabla u_{n,\beta} \|^2
			\leq& 
			\begin{Vmatrix}
				{u_{i}} \\
				{u_{i-1}}%
			\end{Vmatrix}%
			_{G(\theta)}^{2}
			+ \frac{\Delta t}{2\nu \lambda_1} \sum_{n=i}^{N-1} \|f(t_{n,\beta})\|^2 \\
			\leq& 
			\begin{Vmatrix}
				{u_{i}} \\
				{u_{i-1}}%
			\end{Vmatrix}%
			_{G(\theta )}^{2}
			+ \frac{(N-i) \Delta t}{2\nu \lambda_1} \|f\|^2_{\infty},
			\notag \\
		\end{align*}
		which yields \eqref{eq:stabilityL2-3}. 
		For the proof of \eqref{eq:u_{n+1}L^2(H^1)}, 
		we apply the dual operator in \eqref{eq:DLN-NSE-dual} to $u_{n+1}$ 
		\begin{align}
			\label{eq:test with un+1}
			& 
			\Big(\sum_{\ell=0}^{2} \alpha_{\ell} u_{n-1+\ell}, u_{n+1} \Big)
			+ \nu \Delta t ( \nabla u_{n,\beta }, \nabla u_{n+1})
			+ \Delta t b ( u_{n,\beta } \cdot \nabla u_{n,\beta }, u_{n+1}) \\
			& = (f(t_{n,\beta }),u_{n+1}).
			\notag
		\end{align}
		For the non-linear term in \eqref{eq:test with un+1}, by the skew-symmetric property of $b$ in \eqref{eq:b-skew-symm}, H\"older's inequality, Ladyzhenskaya inequallity in \eqref{eq:Lady-ineq}, uniform bound of $u$ in \eqref{eq:u-boundL2-K2}, and Young's inequality 
		\begin{align}
			&b( u_{n,\beta } \cdot \nabla u_{n,\beta } , u_{n+1}) 
			\label{eq:u_nbeta-dot-nabla-unbeta} \\
			=& \beta_{1} (u_{n,\beta} \cdot \nabla u_{n}, u_{n+1} )
			+ \beta_{0} b(u_{n,\beta} \cdot \nabla u_{n-1}, u_{n+1} ) \notag \\
			\leq& 
			\beta_1 2^{-1/2} \| u_{n,\beta }\|^{\frac{1}{2}} \| \nabla u_{n,\beta }\|^{\frac{1}{2}} \| \nabla  u_{n}\|  \| u_{n+1} \|^{\frac{1}{2}}  \| \nabla u_{n+1} \|^{\frac{1}{2}} \notag \\
			& 
			+ \beta_0 2^{-1/2} \| u_{n,\beta }\|^{\frac{1}{2}} \| \nabla u_{n,\beta }\|^{\frac{1}{2}} \| \nabla  u_{n-1}\|  \| u_{n+1} \|^{\frac{1}{2}}  \| \nabla u_{n+1} \|^{\frac{1}{2}} \notag \\ 
			\leq& 
			\beta_1 2^{-1/2} \Big( \sum_{\ell=0}^{2} \beta_{\ell} \| u_{n-1+\ell} \| \Big)^{\frac{1}{2}} \| \nabla u_{n,\beta }\|^{\frac{1}{2}} \| \nabla  u_{n}\| K_2^{1/2} \| \nabla u_{n+1} \|^{\frac{1}{2}} \notag \\
			&
			+ \beta_0 2^{-1/2} \Big( \sum_{\ell=0}^{2} \beta_{\ell} \| u_{n-1+\ell} \| \Big)^{\frac{1}{2}} \| \nabla u_{n,\beta }\|^{\frac{1}{2}}
			\| \nabla  u_{n-1}\| K_2^{1/2} \| \nabla u_{n+1} \|^{\frac{1}{2}} \notag \\
			\leq& \beta_1 2^{-1/2} K_2 \| \nabla u_{n,\beta }\|^{\frac{1}{2}} \| \nabla  u_{n}\| \| \nabla u_{n+1} \|^{\frac{1}{2}} 
			+ \beta_0 2^{-1/2} K_2 \| \nabla u_{n,\beta }\|^{\frac{1}{2}}
			\| \nabla  u_{n-1} \| \| \nabla u_{n+1} \|^{\frac{1}{2}} \notag \\
			\leq& 
			\frac{1}{2}\frac{\nu(2\beta_2-1)}{4}  \| \nabla  u_{n}\|^2
			+ \frac{1}{2} \frac{4}{\nu(2\beta_2-1)} \Big( \beta_1 2^{-1/2} K_2
			\| \nabla u_{n,\beta }\|^{\frac{1}{2}}  \| \nabla u_{n+1} \|^{\frac{1}{2}} \Big)^2 \notag \\
			& 
			+ \frac{1}{2} \frac{\nu(2\beta_2-1)}{4}  \| \nabla  u_{n-1}\|^2
			+ \frac{1}{2} \frac{4}{\nu(2\beta_2-1)} \Big( \beta_0 2^{-1/2} K_2
			\| \nabla u_{n,\beta }\|^{\frac{1}{2}}    \| \nabla u_{n+1} \|^{\frac{1}{2}} \Big)^2 \notag \\
			=&
			\frac{\nu(2\beta_2-1)}{8} \big( \| \nabla  u_{n}\|^2 
			+ \| \nabla  u_{n-1}\|^2 \big)
			+ \frac{(\beta_1^2 + \beta_0^2) K_2^2}{\nu(2\beta_2-1)}  
			\| \nabla u_{n,\beta }\|  \| \nabla u_{n+1} \| \notag \\
			\leq& \frac{\nu(2\beta_2-1)}{8} \big( 4 \| \nabla u_{n+1} \|^2 + \| \nabla  u_{n}\|^2 + \| \nabla  u_{n-1}\|^2 \big)
			+ \frac{ K_2^4(\beta_0^2 + \beta_1^2)^2 } {2\nu^3 (2\beta_2-1)^3} \| \nabla u_{n,\beta }\|^2. \notag 
		\end{align}
		We use \eqref{eq:identity-1}, \eqref{eq:identity-2} and \eqref{eq:u_nbeta-dot-nabla-unbeta} to have
\begin{confidential}
	\color{darkblue}
	\begin{align*}
		& 
		\big({\alpha _{2}}{u_{n+1}}+{\alpha _{1}}{u_{n}}+{\alpha _{0}}{u_{n-1}}, u_{n+1} \big)
		+ \nu \Delta t (\nabla u_{n,\beta }, \nabla u_{n+1})
		+ \Delta t b(u_{n,\beta} \cdot \nabla u_{n,\beta }, u_{n+1}) \\
		& 
		\geq 
		\begin{Vmatrix}
			{u_{n+1}} \\
			{u_{n}}
		\end{Vmatrix}_{G(\theta )}^{2}
		- 
		\begin{Vmatrix}
			{u_{n}} \\
			{u_{n-1}}
		\end{Vmatrix}_{G(\theta )}^{2}
		+ \frac{\theta}{2}  \| u_{n+1} -  u_{n} \|^2
		+ \frac{1}{4} ( 1 - \theta ) \| u_{n+1} -  u_{n-1} \|^2 \\
		& \quad
		+ \nu \Delta t \Big( \frac{3\beta_2 -1}{2} \| \nabla u_{n+1} \|^2
		- \frac{\beta_1}{2}  \|\nabla u_{n} \|^2
		- \frac{\beta_0}{2}  \|\nabla u_{n-1} \|^2 \Big) \\
		& \quad
		+ \nu \Delta t \Big(  \frac{\beta_1}{2} \| ( \nabla u_{n+1} + \nabla u_{n} ) \|^2
		+ \frac{\beta_0}{2} \| ( \nabla u_{n+1} + \nabla u_{n-1} ) \|^2 \Big) \\
		& \quad 
		- \frac{\nu \Delta t (2\beta_2-1)}{8} \big( 4 \| \nabla u_{n+1} \|^2 + \| \nabla  u_{n}\|^2 + \| \nabla  u_{n-1}\|^2 \big)
		- \frac{\Delta t K_2^4(\beta_0^2 + \beta_1^2)^2 } {2\nu^3 (2\beta_2-1)^3} \| \nabla u_{n,\beta }\|^2 \\
		& 
		= 
		\begin{Vmatrix}
			{u_{n+1}} \\
			{u_{n}}
		\end{Vmatrix}_{G(\theta)}^{2}
		- 
		\begin{Vmatrix}
			{u_{n}} \\
			{u_{n-1}}
		\end{Vmatrix}_{G(\theta)}^{2}
		+ \frac{\theta}{2}  \| u_{n+1} -  u_{n} \|^2
		+ \frac{1}{4}(1 - \theta) \| u_{n+1} -  u_{n-1} \|^2 \\
		&\quad
		+ \nu \Delta t \frac{3\beta_2 -1}{2} \| \nabla u_{n+1} \|^2
		- \nu \Delta t  \frac{\beta_1}{2}  \|\nabla u_{n} \|^2
		- \nu \Delta t  \frac{\beta_0}{2}  \|\nabla u_{n-1} \|^2 \\
		& \quad
		+ \nu \Delta t \Big( \frac{\beta_1}{2} \|( \nabla u_{n+1} + \nabla u_{n} )\|^2
		+ \frac{\beta_0}{2} \|( \nabla u_{n+1} + \nabla u_{n-1} )\|^2 \Big) \\
		& \quad
		- \Delta t \frac{\nu(2\beta_2-1)}{8}  \| \nabla  u_{n}\|^2
		- \Delta t  \frac{\nu(2\beta_2-1)}{8}  \| \nabla  u_{n-1}\|^2 \\
		& \quad
		- \Delta t \frac{\nu (2\beta_2-1)}{ 2 } \| \nabla u_{n+1} \|^2
		- \Delta t \frac{ K_2^4(\beta_0^2 + \beta_1^2)^2 } { 2 \nu^3 (2\beta_2-1)^3} \| \nabla u_{n,\beta }\|^2  \\
		& 
		= 
		\begin{Vmatrix}
			{u_{n+1}} \\
			{u_{n}}
		\end{Vmatrix}_{G(\theta)}^{2}
		- 
		\begin{Vmatrix}
			{u_{n}} \\
			{u_{n-1}}
		\end{Vmatrix}_{G(\theta)}^{2}
		+ \frac{\theta}{2}  \| u_{n+1} -  u_{n} \|^2
		+ \frac{1}{4}(1 - \theta) \| u_{n+1} -  u_{n-1} \|^2 \\
		& \quad
		+ \Delta t \nu 
		\Big[\frac{3\beta_2 -1}{2} - \frac{(2\beta_2-1)}{2} \Big]
		\| \nabla u_{n+1} \|^2
		- \Delta t \nu \Big[\frac{\beta_1}{2} + \frac{(2\beta_2-1)}{8} \Big] \|\nabla u_{n} \|^2
		- \Delta t \nu \Big[\frac{\beta_0}{2} + \frac{(2\beta_2-1)}{8}  \Big] \|\nabla u_{n-1} \|^2 \\
		& \quad
		+ \nu \Delta t \Big( \frac{\beta_1}{2} \|( \nabla u_{n+1} + \nabla u_{n}) \|^2
		+ \frac{\beta_0}{2} \|( \nabla u_{n+1} + \nabla u_{n-1} )\|^2 \Big)
		- \Delta t 
		\frac{ K_2^4(\beta_0^2 + \beta_1^2)^2 } { 2 \nu^3 (2\beta_2-1)^3} \| \nabla u_{n,\beta }\|^2  \\
		& 
		= 
		\begin{Vmatrix}
			{u_{n+1}} \\
			{u_{n}}
		\end{Vmatrix}_{G(\theta)}^{2}
		- 
		\begin{Vmatrix}
			{u_{n}} \\
			{u_{n-1}}
		\end{Vmatrix}_{G(\theta)}^{2}
		+ \frac{\theta}{2}  \| u_{n+1} -  u_{n} \|^2
		+ \frac{1}{4}(1 - \theta) \| u_{n+1} -  u_{n-1} \|^2 \\
		& \quad
		+ \Delta t \nu  \frac{\beta_2}{2} \| \nabla u_{n+1} \|^2
		- \Delta t \nu \Big[\frac{\beta_1}{2} + \frac{(2\beta_2-1)}{8} \Big] \|\nabla u_{n} \|^2
		- \Delta t \nu \Big[\frac{\beta_0}{2}  + \frac{(2\beta_2-1)}{8} \Big]\|\nabla u_{n-1} \|^2 \\
		& \quad
		+ \nu \Delta t \Big(  \frac{\beta_1}{2} \|( \nabla u_{n+1} + \nabla u_{n} )\|^2
		+ \frac{\beta_0}{2} \|( \nabla u_{n+1} + \nabla u_{n-1} )\|^2 \Big)
		- \Delta t \frac{ K_2^4(\beta_0^2 + \beta_1^2)^2 } { 2 \nu^3 (2\beta_2-1)^3} \| \nabla u_{n,\beta }\|^2 \\
		& 
		= 
		\begin{Vmatrix}
			{u_{n+1}} \\
			{u_{n}}
		\end{Vmatrix}_{G(\theta)}^{2}
		- 
		\begin{Vmatrix}
			{u_{n}} \\
			{u_{n-1}}
		\end{Vmatrix}_{G(\theta)}^{2}
		+ \frac{\theta}{2}  \| u_{n+1} -  u_{n} \|^2
		+ \frac{1}{4}(1 - \theta) \| u_{n+1} -  u_{n-1} \|^2 \\
		& \quad
		+ \Delta t \nu \Big[ \frac{(2\beta_2-1)}{4} + \frac{1}{4} \Big]
		\| \nabla u_{n+1} \|^2
		- \Delta t \nu \Big[ \frac{\beta_1}{2} + \frac{(2\beta_2-1)}{8} \Big] \|\nabla u_{n} \|^2
		- \Delta t \nu \Big[\frac{\beta_0}{2}  + \frac{(2\beta_2-1)}{8} \Big]\|\nabla u_{n-1} \|^2
		\\
		& \quad
		+ \nu \Delta t \Big( \frac{\beta_1}{2} \|( \nabla u_{n+1} - \nabla u_{n} ) \|^2
		+ \frac{\beta_0}{2} \| ( \nabla u_{n+1} - \nabla u_{n-1} ) \|^2
		\Big)
		- \Delta t \frac{ K_2^4(\beta_0^2 + \beta_1^2)^2 } { 2 \nu^3 (2\beta_2-1)^3} \| \nabla u_{n,\beta }\|^2 \\
		& 
		= 
		\begin{Vmatrix}
			{u_{n+1}} \\
			{u_{n}}
		\end{Vmatrix}_{G(\theta)}^{2}
		- 
		\begin{Vmatrix}
			{u_{n}} \\
			{u_{n-1}}
		\end{Vmatrix}_{G(\theta)}^{2}
		+ \frac{\theta}{2}  \| u_{n+1} -  u_{n} \|^2
		+ \frac{1}{4}(1 - \theta) \| u_{n+1} -  u_{n-1} \|^2 \\
		& \quad
		+\Delta t \frac{\nu(2\beta_2-1)}{4} \| \nabla u_{n+1} \|^2
		+ \Delta t  \frac{\nu}{4} \| \nabla u_{n+1} \|^2
		+ \Delta t \cancel{\Big( \nu  \frac{\beta_0}{2}  + \frac{\nu(2\beta_2-1)}{8}  \Big)\|\nabla u_{n} \|^2}
		- \Delta t \cancel{\Big( \nu  \frac{\beta_0}{2}  + \frac{\nu(2\beta_2-1)}{8}  \Big)\|\nabla u_{n} \|^2} \\
		& \quad
		- \Delta t \nu \Big[\frac{\beta_1}{2} + \frac{(2\beta_2-1)}{8} \Big] \|\nabla u_{n} \|^2
		- \Delta t \nu \Big[\frac{\beta_0}{2}  + \frac{(2\beta_2-1)}{8} \Big]\|\nabla u_{n-1} \|^2 \\
		& \quad
		+ \nu \Delta t \Big(  \frac{\beta_1}{2} \| ( \nabla u_{n+1} + \nabla u_{n} ) \|^2
		+ \frac{\beta_0}{2} \| ( \nabla u_{n+1} + \nabla u_{n-1} ) \|^2
		\Big)
		- \Delta t \frac{ K_2^4(\beta_0^2 + \beta_1^2)^2 } { 2 \nu^3 (2\beta_2-1)^3} \| \nabla u_{n,\beta }\|^2 \\
		& 
		= 
		\begin{Vmatrix}
			{u_{n+1}} \\
			{u_{n}}
		\end{Vmatrix}_{G(\theta)}^{2}
		- 
		\begin{Vmatrix}
			{u_{n}} \\
			{u_{n-1}}
		\end{Vmatrix}_{G(\theta)}^{2}
		+ \frac{\theta}{2}  \| u_{n+1} -  u_{n} \|^2
		+ \frac{1}{4}(1 - \theta) \| u_{n+1} -  u_{n-1} \|^2 \\
		& \quad
		+
		\Delta t \frac{\nu(2\beta_2-1)}{4} \| \nabla u_{n+1} \|^2
		+ \Delta t \nu \Big[  \frac{1}{4} \| \nabla u_{n+1} \|^2
		+ \Big( \frac{\beta_0}{2}  + \frac{(2\beta_2-1)}{8}  \Big)\|\nabla u_{n} \|^2 \Big] \\
		& \quad
		- \Delta t \nu \Big[ \frac{\beta_0}{2}  + \frac{(2\beta_2-1)}{8}  + \frac{\beta_1}{2} + \frac{(2\beta_2-1)}{8}  \Big] \|\nabla u_{n} \|^2
		- \Delta t \nu \Big[ \frac{\beta_0}{2}  + \frac{(2\beta_2-1)}{8}  \Big] \|\nabla u_{n-1} \|^2 \\
		& \quad
		+ \nu \Delta t \Big(  \frac{\beta_1}{2} \|( \nabla u_{n+1} + \nabla u_{n} )\|^2
		+ \frac{\beta_0}{2} \| ( \nabla u_{n+1} + \nabla u_{n-1} ) \|^2 \Big)
		- \Delta t \frac{ K_2^4(\beta_0^2 + \beta_1^2)^2 } { 2 \nu^3 (2\beta_2-1)^3} \| \nabla u_{n,\beta }\|^2 \\
		& 
		= 
		\begin{Vmatrix}
			{u_{n+1}} \\
			{u_{n}}
		\end{Vmatrix}_{G(\theta)}^{2}
		- 
		\begin{Vmatrix}
			{u_{n}} \\
			{u_{n-1}}
		\end{Vmatrix}_{G(\theta)}^{2}
		+ \frac{\theta}{2}  \| u_{n+1} -  u_{n} \|^2
		+ \frac{1}{4}(1 - \theta) \| u_{n+1} -  u_{n-1} \|^2 \\
		& \quad
		+
		\Delta t \frac{\nu(2\beta_2-1)}{4} \| \nabla u_{n+1} \|^2
		+ \Delta t \nu \Big[  \frac{1}{4} \| \nabla u_{n+1} \|^2
		+ \Big(  \frac{\beta_0}{2}  + \frac{(2\beta_2-1)}{8}  \Big)\|\nabla u_{n} \|^2 \Big]
		\\
		& \qquad \qquad \qquad \qquad \qquad \qquad \quad     
		- \Delta t \nu \Big[  \frac{1}{4} \|\nabla u_{n} \|^2
		+ \Big( \frac{\beta_0}{2}  + \frac{(2\beta_2-1)}{8}  \Big)\|\nabla u_{n-1} \|^2 \Big] \\
		& \quad 
		+ \nu \Delta t \Big( \frac{\beta_1}{2} \| ( \nabla u_{n+1} + \nabla u_{n} ) \|^2
		+ \frac{\beta_0}{2} \| ( \nabla u_{n+1} + \nabla u_{n-1} ) \|^2 \Big)
		- \Delta t \frac{ K_2^4(\beta_0^2 + \beta_1^2)^2 } { 2 \nu^3 (2\beta_2-1)^3} \| \nabla u_{n,\beta }\|^2,
	\end{align*}
	and therefore
	\begin{align*}
		& 
		\begin{Vmatrix}
			{u_{n+1}} \\
			{u_{n}}
		\end{Vmatrix}_{G(\theta)}^{2}
		- 
		\begin{Vmatrix}
			{u_{n}} \\
			{u_{n-1}}
		\end{Vmatrix}_{G(\theta)}^{2}
		+ \frac{\theta}{2}  \| u_{n+1} -  u_{n} \|^2
		+ \frac{1}{4}(1 - \theta) \| u_{n+1} -  u_{n-1} \|^2 \\
		& \quad
		+
		\Delta t \frac{\nu(2\beta_2-1)}{4} \| \nabla u_{n+1} \|^2
		+ \Delta t \nu \Big[  \frac{1}{4} \| \nabla u_{n+1} \|^2
		+ \Big(  \frac{\beta_0}{2}  + \frac{(2\beta_2-1)}{8}  \Big)\|\nabla u_{n} \|^2 \Big]
		\\
		& \qquad \qquad \qquad \qquad \qquad \qquad \quad     
		- \Delta t \nu \Big[  \frac{1}{4} \|\nabla u_{n} \|^2
		+ \Big( \frac{\beta_0}{2}  + \frac{(2\beta_2-1)}{8}  \Big)\|\nabla u_{n-1} \|^2 \Big] \\
		& \quad 
		+ \nu \Delta t \Big( \frac{\beta_1}{2} \| ( \nabla u_{n+1} + \nabla u_{n} ) \|^2
		+ \frac{\beta_0}{2} \| ( \nabla u_{n+1} + \nabla u_{n-1} ) \|^2 \Big)
		- \Delta t \frac{ K_2^4(\beta_0^2 + \beta_1^2)^2 } { 2 \nu^3 (2\beta_2-1)^3} \| \nabla u_{n,\beta }\|^2 
		\\
		& 
		\leq
		\big({\alpha _{2}}{u_{n+1}}+{\alpha _{1}}{u_{n}}+{\alpha _{0}}{u_{n-1}}, u_{n+1} \big)
		+ \nu \Delta t (\nabla u_{n,\beta}, \nabla u_{n+1})
		+ \Delta t b( u_{n,\beta } \cdot \nabla u_{n,\beta }, u_{n+1})
		\\
		& 
		= 
		\Delta t (f(t_{n,\beta}) , u_{n+1})
		\leq \Delta t \| f(t_{n,\beta}) \| \| u_{n+1} \| 
		\leq \frac{\Delta t}{\sqrt{\lambda_1}} \| f(t_{n,\beta}) \| \| \nabla u_{n+1} \| \\
		&= 
		\Delta t  \sqrt{\frac{4}{\nu(2 \beta_{2} - 1) \lambda_1}} \| f(t_{n,\beta}) \| \sqrt{\frac{\nu(2 \beta_{2} - 1)}{4}} \| \nabla u_{n+1} \| \\
		&\leq 
		\frac{\Delta t}{2} \Big( \frac{4}{\nu(2 \beta_{2} - 1) \lambda_1} \| f(t_{n,\beta}) \|^2 + \frac{\nu(2 \beta_{2} - 1)}{4} \| \nabla u_{n+1} \|^2 \Big) \\
		&= 
		\frac{2 \Delta t}{\nu(2 \beta_{2} - 1) \lambda_1} \| f \|_{\infty}^2
		+ \frac{\Delta t \nu(2 \beta_{2} - 1)}{8} \| \nabla u_{n+1} \|^2 
	\end{align*}
	\normalcolor
\end{confidential}
		\begin{align}
			& 
			\begin{Vmatrix}
				{u_{n+1}} \\
				{u_{n}}
			\end{Vmatrix}_{G(\theta)}^{2}
			- 
			\begin{Vmatrix}
				{u_{n}} \\
				{u_{n-1}}
			\end{Vmatrix}_{G(\theta)}^{2}
			+ \frac{\theta}{2}  \| u_{n+1} -  u_{n} \|^2
			+ \frac{1}{4}(1 - \theta) \| u_{n+1} -  u_{n-1} \|^2 
			\label{eq:test-un+1-eq2} \\
			& 
			+\Delta t \frac{\nu(2\beta_2-1)}{4} \| \nabla u_{n+1} \|^2
			+ \Delta t \nu \Big[  \frac{1}{4} \| \nabla u_{n+1} \|^2
			+ \Big(  \frac{\beta_0}{2}  + \frac{(2\beta_2-1)}{8}  \Big)\|\nabla u_{n} \|^2 \Big]
			\notag \\
			& \qquad \qquad \qquad \qquad \qquad \qquad \  
			- \Delta t \nu \Big[  \frac{1}{4} \|\nabla u_{n} \|^2
			+ \Big( \frac{\beta_0}{2}  + \frac{(2\beta_2-1)}{8}  \Big)\|\nabla u_{n-1} \|^2 \Big] \notag \\
			& 
			+\! \nu \Delta t \Big( \frac{\beta_1}{2} \| ( \nabla u_{n+1} \!+\! \nabla u_{n} ) \|^2
			\!+\! \frac{\beta_0}{2} \| ( \nabla u_{n+1} \!+\! \nabla u_{n-1} ) \|^2 \Big)
			\!- \!\Delta t \frac{ K_2^4(\beta_0^2 + \beta_1^2)^2 } { 2 \nu^3 (2\beta_2-1)^3} \| \nabla u_{n,\beta }\|^2 
			\notag \\
			& 
			\leq
			\big({\alpha _{2}}{u_{n+1}}+{\alpha _{1}}{u_{n}}+{\alpha _{0}}{u_{n-1}}, u_{n+1} \big)
			\!+\! \nu \Delta t (\nabla u_{n,\beta}, \nabla u_{n+1})
			\!+\! \Delta t b( u_{n,\beta } \cdot \nabla u_{n,\beta }, u_{n+1}). \notag  
		\end{align}
		By Cauchy-Schwarz inequality, Poincar\'e inequality in \eqref{eq:PF-ineq}, Young's inequality and \eqref{eq:test-un+1-eq2}, \eqref{eq:test with un+1} becomes 
		\begin{align*}
			& 
			\begin{Vmatrix}
				{u_{n+1}} \\
				{u_{n}}
			\end{Vmatrix}_{G(\theta )}^{2}
			+ \frac{\theta}{2}  \| u_{n+1} -  u_{n} \|^2
			+ \frac{1}{4} ( 1 - \theta ) \| u_{n+1} -  u_{n-1} \|^2 \\
			& \quad
			+ \Delta t \frac{\nu(2\beta_2-1)}{8} \| \nabla u_{n+1} \|^2
			+ \Delta t \nu \Big[  \frac{1}{4} \| \nabla u_{n+1} \|^2
			+ \Big(  \frac{\beta_0}{2}  + \frac{(2\beta_2-1)}{8} \Big)\|\nabla u_{n} \|^2 \Big] \\
			& \quad
			+ \nu \Delta t \Big(  \frac{\beta_1}{2} \| ( \nabla u_{n+1} + \nabla u_{n} ) \|^2
			+ \frac{\beta_0}{2} \| ( \nabla u_{n+1} + \nabla u_{n-1} ) \|^2
			\Big) \\
			\leq&
			\begin{Vmatrix}
				{u_{n}} \\
				{u_{n-1}}
			\end{Vmatrix}_{G(\theta )}^{2}
			+ \Delta t \nu \Big[ \frac{1}{4} \|\nabla u_{n} \|^2
			+ \Big( \frac{\beta_0}{2}  + \frac{(2\beta_2-1)}{8} \Big) \|\nabla u_{n-1} \|^2 \Big] \\
			& \quad 
			\Delta t \frac{ K_2^4(\beta_0^2 + \beta_1^2)^2 } { 2 \nu^3 (2\beta_2-1)^3} \| \nabla u_{n,\beta }\|^2
			+ \frac{2 \Delta t}{\nu(2 \beta_{2} - 1) \lambda_1} \| f \|_{\infty}^2.
		\end{align*}
		Adding up for $n$ from $i$ to $N-1$ yields
\begin{confidential}
	\color{darkblue}
	\begin{align*}
		& 
		\begin{Vmatrix}
			{u_{N}} \\
			{u_{N-1}}
		\end{Vmatrix}_{G(\theta )}^{2}
		+ \sum_{j=i}^{N-1} \Big(  \frac{\theta}{2} \| u_{j+1} -  u_{j} \|^2
		+ \frac{1}{4} ( 1 - \theta ) \| u_{j+1} -  u_{j-1} \|^2 \Big) \\
		& 
		+ \Delta t \frac{\nu(2\beta_2-1)}{8} \sum_{j=i}^{N-1} \| \nabla u_{j+1} \|^2
		+ \Delta t \nu \Big[  \frac{1}{4} \| \nabla u_{N} \|^2
		+ \Big(  \frac{\beta_0}{2}  + \frac{(2\beta_2-1)}{8}  \Big)\|\nabla u_{N-1} \|^2 \Big] \\
		& 
		+ \nu \Delta t \sum_{j=i}^{N-1}\Big(  \frac{\beta_1}{2} \| ( \nabla u_{j+1} - \nabla u_{j} ) \|^2
		+ \frac{\beta_0}{2} \| ( \nabla u_{j+1} - \nabla u_{j-1} ) \|^2 \Big)
		\\
		\leq&
		\begin{Vmatrix}
			{u_{i}} \\
			{u_{i-1}}
		\end{Vmatrix}_{G(\theta )}^{2}
		+ \Delta t \nu \Big[ \frac{1}{4} \|\nabla u_{i} \|^2
		+ \Big( \frac{\beta_0}{2}  + \frac{(2\beta_2-1)}{8}  \Big)\|\nabla u_{i-1} \|^2 \Big] \\
		& 
		+ \Delta t \frac{ K_2^4(\beta_0^2 + \beta_1^2)^2 } { 2 \nu^3 (2\beta_2-1)^3} \sum_{j=i}^{N-1} \| \nabla u_{j,\beta }\|^2
		+ \frac{2 (N-i) \Delta t}{\nu(2 \beta_{2} - 1) \lambda_1} \| f \|_{\infty}^2,
	\end{align*}
	\normalcolor
\end{confidential}
		\begin{align*}
			& 
			\begin{Vmatrix}
				{u_{N}} \\
				{u_{N-1}}
			\end{Vmatrix}_{G(\theta )}^{2}
			+ \Delta t \nu \Big[  \frac{1}{4} \| \nabla u_{N} \|^2
			+ \Big(  \frac{\beta_0}{2}  + \frac{(2\beta_2-1)}{8}  \Big)\|\nabla u_{N-1} \|^2 \Big]
			\\
			& 
			+ \Delta t \frac{\nu(2\beta_2-1)}{8} \sum_{j=i}^{N-1} \| \nabla u_{j+1} \|^2
			+ \sum_{j=i}^{N-1} \Big(  \frac{\theta}{2} \| u_{j+1} -  u_{j} \|^2
			+ \frac{1}{4} ( 1 - \theta ) \| u_{j+1} -  u_{j-1} \|^2 \Big)
			\\
			& 
			+ \nu \Delta t \sum_{j=i}^{N-1}\Big( \frac{\beta_1}{2} \|( \nabla u_{j+1} - \nabla u_{j} )\|^2
			+ \frac{\beta_0}{2} \|( \nabla u_{j+1} - \nabla u_{j-1} )\|^2 \Big)
			\\ 
			\leq&
			\begin{Vmatrix}
				{u_{i}} \\
				{u_{i-1}}
			\end{Vmatrix}_{G(\theta)}^{2}
			+ \Delta t \nu \Big[  \frac{1}{4} \|\nabla u_{i} \|^2
			+ \Big( \frac{\beta_0}{2}  + \frac{(2\beta_2-1)}{8}  \Big)\|\nabla u_{i-1} \|^2 \Big] \\
			& 
			+ \frac{ K_2^4(\beta_0^2 + \beta_1^2)^2 } { 2 \nu^3 (2\beta_2-1)^3} 
			\Delta t  \sum_{j=i}^{N-1} \| \nabla u_{j,\beta }\|^2
			+ \frac{2 (N-i) \Delta t}{\nu(2 \beta_{2} - 1) \lambda_1} \| f \|_{\infty}^2,
		\end{align*}
		which by \eqref{eq:stabilityL2-3} yields \eqref{eq:u_{n+1}L^2(H^1)}.
\begin{confidential}
	\color{darkblue}
	\begin{align*}
		\Delta t \sum_{j=i}^{N-1} \| \nabla u_{j, \beta} \|^2
		\leq 
		\frac{2}{\nu}
		\begin{Vmatrix}
			{u_{i}} \\
			{u_{i-1}}
		\end{Vmatrix}_{G(\theta )}^{2}
		+ \frac{1}{\nu^2 \lambda_1} (N-i) \Delta t \|f\|^2_{\infty} 
	\end{align*}
	\begin{align*}
		&\Delta t \frac{\nu(2\beta_2-1)}{8} \sum_{j=i}^{N-1} \| \nabla u_{j+1} \|^2 \\
		&+
		\begin{Vmatrix}
			{u_{N}} \\
			{u_{N-1}}
		\end{Vmatrix}_{G(\theta )}^{2}
		+ \Delta t \nu \Big[  \frac{1}{4} \| \nabla u_{N} \|^2
		+ \Big(  \frac{\beta_0}{2}  + \frac{(2\beta_2-1)}{8}  \Big)\|\nabla u_{N-1} \|^2 \Big] \\
		&+ \sum_{j=i}^{N-1} \Big(  \frac{\theta}{2} \| u_{j+1} -  u_{j} \|^2
		+ \frac{1}{4} ( 1 - \theta ) \| u_{j+1} -  u_{j-1} \|^2 \Big) \\
		&+ \nu \Delta t \sum_{j=i}^{N-1}\Big( \frac{\beta_1}{2} \|( \nabla u_{j+1} - \nabla u_{j} )\|^2
		+ \frac{\beta_0}{2} \|( \nabla u_{j+1} - \nabla u_{j-1} )\|^2 \Big) \\
		\leq&
		\begin{Vmatrix}
			{u_{i}} \\
			{u_{i-1}}
		\end{Vmatrix}_{G(\theta)}^{2}
		+ \Delta t \nu \Big[ \frac{1}{4} \|\nabla u_{i} \|^2
		+ \Big( \frac{\beta_0}{2}  + \frac{(2\beta_2-1)}{8}  \Big)\|\nabla u_{i-1} \|^2 \Big] \\
		&+ \frac{ K_2^4(\beta_0^2 + \beta_1^2)^2 } { 2 \nu^3 (2\beta_2-1)^3}
		\Big( \frac{2}{\nu}
		\begin{Vmatrix}
			{u_{i}} \\
			{u_{i-1}}
		\end{Vmatrix}_{G(\theta )}^{2}
		+ \frac{1}{\nu^2 \lambda_1} (N-i) \Delta t \|f\|^2_{\infty} \Big)
		+ \frac{2 (N-i) \Delta t}{\nu(2 \beta_{2} - 1) \lambda_1} \| f \|_{\infty}^2 \\
		=&\Big( 1 + \frac{ K_2^4(\beta_0^2 + \beta_1^2)^2}{\nu^4 (2\beta_2-1)^3} \Big) 
		\begin{Vmatrix}
			{u_{i}} \\
			{u_{i-1}}
		\end{Vmatrix}_{G(\theta)}^{2} 
		+ \Delta t \nu \Big[ \frac{1}{4} \|\nabla u_{i} \|^2
		+ \Big(\frac{\beta_0}{2} + \frac{(2\beta_2-1)}{8} \Big)\|\nabla u_{i-1} \|^2 \Big] \\
		&+ \frac{(N-i) \Delta t}{\nu(2 \beta_{2} - 1) \lambda_1}
		\Big( 2 + \frac{ K_2^4(\beta_0^2 + \beta_1^2)^2 } { 2 \nu^4 (2\beta_2-1)^2} \Big) \| f \|_{\infty}^2 
	\end{align*}
	\normalcolor
\end{confidential}
    \end{proof}

	We denote 
	\begin{align}
		\widehat{C}_{h}(\theta) 
		= \frac{C_h(\theta)}{\frac{\theta^3}{2} \Big[ 1 - \frac{1}{2} \theta (1 - \theta) \Big]}, \quad 
		\rho_0 
		= \frac{C_{\epsilon}(\theta)}{\nu^2 \lambda_1^2 \theta^3 \Big[ 1 - \frac{1}{2} \theta (1 - \theta) \Big]} 
		\|f\|^2_{\infty}.  \label{eq:note-rho}
	\end{align}
	We have the following lemma about the upper bound of model energy independent of the initial conditions. 

	\begin{corollary}
		\label{coro:L2-bound-no-init}
		Under the time step restriction in \eqref{eq:dt-limit-1}, 
		if $N$ is large enough such that 
		\begin{align}
			(N-1) \Delta t > T_{\ast} \big( \| u_{0}\|, \| u_{1}\|, \|f\|_{\infty} \big)
			= \frac{4 C_{\epsilon}(\theta)}{\nu \lambda_1} \ln \Big( \frac{ \widehat{C}_{h}(\theta) \big( \|u_1 \|^2 + \|u_0 \|^2 \big)}{\rho_0} \Big),
			\label{eq:N-T-ast-L2}
		\end{align}
		the approximate velocity in \Cref{sche:DLN-NSE} satisfies
		\begin{align}
			\label{eq:stabilityL2-0}
			& \|u_{N}\|^2 \leq 2 \rho_0,
		\end{align}
	\end{corollary}
	\begin{proof}
		By \eqref{eq:L2bound0} in Theorem \ref{thm:stab-L2}, \eqref{eq:h-lower-bound}-\eqref{eq:1-epsi-cond} in Lemma \ref{lem:H-stab}
		\begin{align}
			&\frac{\theta^3}{2} \Big[ 1 - \frac{1}{2} \theta (1 - \theta) \Big] \| u_{N} \|^{2} 
			\label{eq:L2-uniform-eq1} \\
			\leq& 
			\frac{1}{(1+\epsilon)^{N-1}} C_{h}(\theta) \big( \| u_1 \|^2  + \| u_0 \|^2 \big)
			+ \frac{C_{\varepsilon}(\theta)}{2 \nu^2 \lambda_1^2} 
			\|f\|^2_{\infty}. \notag 
		\end{align}
		We use the notation in \eqref{eq:note-rho} and have 
		\begin{align*}
			\| u_{N} \|^{2} 
			\leq \frac{1}{(1+\epsilon)^{N-1}} \widehat{C}_{h}(\theta) \big( \| u_1 \|^2  + \| u_0 \|^2 \big) + \rho_0.
		\end{align*}
		By the fact that
		\begin{align*}
			1 + x \geq \exp(x/4), \qquad \forall x \in (0,4), 
		\end{align*}
\begin{confidential}
	\color{darkblue}
	\begin{align*}
		&F(x) = 1 + x - \exp(x/4), \quad F(0) = 1 + 0 - 1 = 0, \\
		&\text{if } F'(x) = 1 - \frac{1}{4}e^{x/4} > 0, \quad 
		\frac{1}{4}e^{x/4} < 1, \quad e^{x/4} < 4, \quad 
		\frac{x}{4} < \log (4), \quad x < 4 \log(4) (>4).
	\end{align*}
	\normalcolor
\end{confidential}
		and $\epsilon \in (0,4)$ in \eqref{eq:1-epsi-cond}, \eqref{eq:L2-uniform-eq1} becomes 
		\begin{align}
			\| u_{N} \|^{2}
			\leq \exp \big( - (N - 1)\epsilon/4 \big) \widehat{C}_{h}(\theta) \big( \| u_1 \|^2  + \| u_0 \|^2 \big) + \rho_0. 
			\label{eq:L2-uniform-eq2}
		\end{align}
\begin{confidential}
	\color{darkblue}
	We want 
	\begin{align*}
		&\exp \big( - (N - 1)\epsilon/4 \big) \widehat{C}_{h}(\theta) \big( \| u_1 \|^2  + \| u_0 \|^2 \big) < \rho_0, \ \ 
		\frac{\widehat{C}_{h}(\theta) \big( \| u_1 \|^2  + \| u_0 \|^2 \big)}{\rho_0} < \exp \big( (N - 1)\epsilon/4 \big) \\
		& \log \Big( \frac{\widehat{C}_{h}(\theta) \big( \| u_1 \|^2  + \| u_0 \|^2 \big)}{\rho_0} \Big) < \frac{(N-1) \epsilon}{4}, \ \ 
		\frac{4}{\epsilon} \log \Big( \frac{\widehat{C}_{h}(\theta) \big( \| u_1 \|^2  + \| u_0 \|^2 \big)}{\rho_0} \Big) < N-1,
	\end{align*}
	it suffices to have 
	\begin{align*}
		\frac{4 C_{\epsilon}(\theta)}{\nu \lambda_1 \Delta t}
		\log \Big( \frac{\widehat{C}_{h}(\theta) \big( \| u_1 \|^2  + \| u_0 \|^2 \big)}{\rho_0} \Big) < N-1, 
	\end{align*}
	which is by \eqref{eq:1-epsi-cond}.
	\normalcolor
\end{confidential}
		Let $N_{0}$ to be the minimal integer $N$ such that 
		\begin{align*}
			N - 1 > \frac{4 C_{\epsilon}(\theta)}{\nu \Delta t \lambda_1} \ln \Big( \frac{ \widehat{C}_{h}(\theta) \big( \|u_1 \|^2 + \|u_0 \|^2 \big)}{\rho_0} \Big).
		\end{align*}
		Then for $N \geq N_0$, we have \eqref{eq:N-T-ast-L2} and by \eqref{eq:1-epsi-cond}
		\begin{align*}
			N - 1 
			> \frac{4}{\epsilon} \ln \Big( \frac{ \widehat{C}_{h}(\theta) \big( \|u_1 \|^2 + \|u_0 \|^2 \big)}{\rho_0} \Big)
		\end{align*}
		or 
		\begin{align}
			\exp \big( - (N-1) \epsilon/4 \big) \widehat{C}_{h}(\theta) \big( \| u_1 \|^2  + \| u_0 \|^2 \big) < \rho_0.
			\label{eq:rho-inequality}
		\end{align}
		We combine \eqref{eq:L2-uniform-eq2} and \eqref{eq:rho-inequality} to achieve \eqref{eq:stabilityL2-0}. 
	\end{proof}

	\section{Long-time stability in $H^1$-norm} \label{sec:stability-H1}

	To seek a uniform bound for $\|\nabla u_{N}\|$ independent of initial conditions,  
	we first derive the bound on a finite time interval and then use this bound on successive intervals, via a discrete uniform Gr\"onwall Lemma \cite[pp.37-38]{TW06_SIAMNA}. 
	We make the following notations
	\begin{align}
		& 
		\kappa_1 = \frac{27 (2 - \theta^2) C_{\Omega}^2 K_2^2}{16 \nu^3},
		\label{eq:kappa-1}
		\\
		& 
		\kappa_2 = 3 + \frac{8 (2 - \theta^2) K_2^2 }{\nu^2 (2 - \theta)^2(1+\theta)},
		\label{eq:kappa-2}
		\\
		& 
		\kappa_3 = \frac{27 (2 - \theta^2) C_{\Omega}^2 \rho_0}{8 \nu^3},
		\label{eq:kappa-3}
		\\
		& 
		\kappa_4 = 3 + \frac{16 (2 - \theta^2) \rho_0}{\nu^2 (2 - \theta)^2(1+\theta)},
		\label{eq:kappa-4}
		\\
		& 
		A_{n+1} =
		\displaystyle
		\begin{Vmatrix}
			\vspace{.2cm}
			\nabla {u_{n+1}} \\
			\nabla {u_{n}}
		\end{Vmatrix}_{G(\theta )}^{2}, \label{eq:An}
	\end{align}
	and proceed by proving an intermediate recurrence inequality for $A_n$.
	\begin{lemma}
		\label{lemma:auxi-ineq}
		Assuming that 
		$\Delta t$ satisfies \eqref{eq:dt-limit-1}, we have
		\begin{align}
			\label{eq:H1-ineq1}
			A_{n+1} - A_n
			\leq&
			\frac{1}{2\nu} \Delta t \|f \|^2_{\infty}
			+ \Delta t \kappa_1 A_{n+1} \|\nabla u_{n,\beta}\|^2 + 
			\Delta t \kappa_1 A_n \|\nabla u_{n,\beta}\|^2,
		\end{align}
		and 
		\begin{align}  
			& 
			A_{n+1}
			\leq \kappa_2 A_n + \frac{\Delta t}{\nu} \|f\|^2_{\infty},
			\label{eq:H1-ineq2}
		\end{align}
		Moreover, if $n$ is large enough such that $(n-2) \Delta t > T_{\ast}$, 
		\begin{align}
			\label{eq:H1-ineq1-uniform}
			A_{n+1} - A_n
			\leq&
			\frac{1}{2\nu} \Delta t \|f \|^2_{\infty}
			+ \Delta t \kappa_3 A_{n+1} \|\nabla u_{n,\beta}\|^2 
			+ \Delta t \kappa_3 A_n \|\nabla u_{n,\beta}\|^2, 
		\end{align}
		and 
		\begin{align}  
			& A_{n+1}
			\leq 
			\kappa_4 A_n + \frac{\Delta t}{\nu} \|f\|^2_{\infty},
			\label{eq:H1-ineq2-uniform}
		\end{align}
		where $\kappa_1$, $\kappa_2$, $\kappa_3$, $\kappa_4$, and $A_n$ are given in \eqref{eq:kappa-1}-\eqref{eq:An}.
	\end{lemma}
	\begin{proof}
		We apply the operator in \eqref{eq:DLN-NSE-dual} on $- \Delta u_{n,\beta}$, and utilize $G$-stability identity in \eqref{eq:G-stab-eq}, Cauchy-Schwarz inequality, Young's inequality yield 
		\begin{align}
			\label{eq:GstabH1-eq1}
			& 
			A_{n+1} - A_{n}
			+ 
			\Big\|\sum_{\ell =0}^{2}{a_{\ell }} \, \nabla u_{n-1+\ell }%
			\Big\|^{2}
			+ \frac{\nu}{2} \Delta t \Big\| \Delta u_{n,\beta} \Big\|^2
			\\
			& 
			\leq
			\frac{1}{2\nu} \Delta t \|f(t_{n,\beta })\|^2_{L^{2}(\Omega)} 
			+ \Delta t b (u_{n,\beta}, u_{n,\beta}, \Delta u_{n,\beta}).
			\notag
		\end{align}
\begin{confidential}
	\color{darkblue}
	\begin{align}
		& 
		\begin{Vmatrix}
			\nabla {u_{n+1}} \\
			\nabla {u_{n}}
		\end{Vmatrix}_{G(\theta )}^{2} 
		- 
		\begin{Vmatrix}
			\nabla  {u_{n}} \\
			\nabla {u_{n-1}}
		\end{Vmatrix}_{G(\theta )}^{2}
		+ \Big\|\sum_{\ell =0}^{2}{a_{\ell }} \, \nabla u_{n-1+\ell }%
		\Big\|^{2}
		- \Delta t b (u_{n,\beta}, u_{n,\beta}, \Delta u_{n,\beta})
		+ \nu \Delta t \Big\| \Delta u_{n,\beta} \Big\|^2 \notag \\
		& 
		= - \Delta t \big( f(t_{n,\beta }) , \Delta u_{n,\beta}  \big)			
		\leq
		\Delta t \|f(t_{n,\beta })\| \|  \Delta u_{n,\beta} \|			
		\leq
		\frac{1}{2\nu} \Delta t \|f(t_{n,\beta })\|^2
		+ \frac{\nu}{2} \Delta t \| \Delta u_{n,\beta} \|^2, \notag
	\end{align}
	\normalcolor
\end{confidential}
		By H\"older's inequality, Ladyzhenskaya inequality in \eqref{eq:Lady-ineq} and the $H^2$-bound for the Laplacian in \eqref{eq:equi-D2-ineq}
		\begin{align}
			& 
			b (u_{n,\beta}, u_{n,\beta}, \Delta u_{n,\beta}) 
			\label{eq:b-eq1} \\
			\leq& 
			\| u_{n,\beta} \|_{L^4} \| \nabla u_{n,\beta} \|_{L^4}  \| \Delta u_{n,\beta} \| \notag \\
			\leq&
			2^{-1/4} \| u_{n,\beta} \|^{1/2} \| \nabla u_{n,\beta} \|^{1/2} 
			2^{-1/4} \| \nabla u_{n,\beta} \|^{1/2}  \| D^2 u_{n,\beta} \|^{1/2}  
			\| \Delta u_{n,\beta} \| \notag \\
			\leq& 
			2^{-1/2} C_{\Omega}^{1/2} \| u_{n,\beta} \|^{1/2} \| \nabla u_{n,\beta} \| \| \Delta u_{n,\beta} \|^{3/2}  \notag \\
			\leq&
			\frac{27 C_{\Omega}^2}{16 \nu^3} \| u_{n,\beta} \|^{2} \| \nabla u_{n,\beta} \|^4
			+ \frac{\nu}{4} \| \Delta u_{n,\beta} \|^{2}. \notag
		\end{align}
\begin{confidential}
	\color{darkblue}
	in 2D 
	\begin{align}
		\label{eq:Ladyz1}
		& 
		b (u_{n,\beta}, u_{n,\beta}, \Delta u_{n,\beta})
		\leq
		\| u_{n,\beta} \|_{L^4} \| \nabla u_{n,\beta} \|_{L^4}  \| \Delta u_{n,\beta} \|_{L^2}
		\\
		& 
		=
		\Big(\frac{C_{\Omega}}{2} \Big)^{1/2} \| u_{n,\beta} \|^{1/2}_{L^2} \| \nabla u_{n,\beta} \|_{L^2}    \| \Delta u_{n,\beta} \|^{3/2}_{L^2}  
		\notag
		\\
		& 
		=
		\underbrace{\Big(\frac{C_{\Omega}}{2} \Big)^{1/2} \frac{1}{\sqrt{a}}\| u_{n,\beta} \|^{1/2}_{L^2} \| \nabla u_{n,\beta} \|_{L^2}}_{4} 
		\underbrace{a^{1/2}  \| \Delta u_{n,\beta} \|^{3/2}_{L^2} }_{4/3}
		\notag
		\\
		& 
		\leq
		\frac{1}{4} \bigg( \Big(\frac{C_{\Omega}}{2} \Big)^{1/2} \frac{1}{\sqrt{a}}\| u_{n,\beta} \|^{1/2}_{L^2} \| \nabla u_{n,\beta} \|_{L^2} \bigg)^4
		+ \frac{3}{4} \bigg( a^{1/2}  \| \Delta u_{n,\beta} \|^{3/2}_{L^2} \bigg)^{4/3}
		\notag
		\\
		& 
		=
		\frac{1}{4} \Big(\frac{C_{\Omega}}{2} \Big)^{2} \frac{1}{a^2}\| u_{n,\beta} \|^{2}_{L^2} \| \nabla u_{n,\beta} \|^4_{L^2}
		+ \frac{3}{4} a^{2/3}  \| \Delta u_{n,\beta} \|^{2}_{L^2}.
		\notag
	\end{align}
	Now taking 
	$$ 
	a = \Big( \frac{\nu}{3}\Big)^{3/2}
	$$
	one gets that
	\begin{align*}
		& 
		b (u_{n,\beta}, u_{n,\beta}, \Delta u_{n,\beta})
		\leq
		\frac{27 C_{\Omega}^2}{16 \nu^3} \| u_{n,\beta} \|^{2}_{L^2} \| \nabla u_{n,\beta} \|^4_{L^2}
		+ \frac{\nu}{4}  \| \Delta u_{n,\beta} \|^{2}_{L^2}
		,
	\end{align*}
	so we have 
	\begin{align*}
		& 
		\begin{Vmatrix}
			\nabla {u_{n+1}} \\
			\nabla {u_{n}}
		\end{Vmatrix}
		_{G(\theta )}^{2} 
		- 
		\begin{Vmatrix}
			\nabla  {u_{n}} \\
			\nabla {u_{n-1}}
		\end{Vmatrix}
		_{G(\theta )}^{2}
		+ \Big\|\sum_{\ell =0}^{2}{a_{\ell }} \, \nabla u_{n-1+\ell }%
		\Big\|^{2}
		+ \frac{\nu}{4} \Delta t \Big\| \Delta u_{n,\beta} \Big\|^2 \notag
		\\
		& 
		\leq
		\frac{1}{2\nu} \Delta t 
		\|f\|^2_\infty
		+ \frac{27 C_{\Omega}^2}{16\nu^3} \Delta t \| u_{n,\beta} \|^{2}_{L^2} \| \nabla u_{n,\beta} \|^4. 
	\end{align*}
	\normalcolor
\end{confidential}
		We combine \eqref{eq:GstabH1-eq1} and \eqref{eq:b-eq1} to obtain 
		\begin{align}
			& 
			A_{n+1} - A_n
			\leq
			\frac{1}{2\nu} \Delta t \|f\|^2_{\infty}
			+ \frac{27 C_{\Omega}^2}{16 \nu^3} \Delta t \| u_{n,\beta} \|^{2}_{L^2} \| \nabla u_{n,\beta} \|^4.
			\label{eq:An-ineq-1}
		\end{align}
		Now, using the Cauchy-Schwarz inequality, we have
		\begin{align}
			\|\nabla u_{n,\beta}\|^2 =& 
			\Big\| \beta_2 \nabla u_{n+1} + \Big( \frac{1}{2} -\beta_2 \Big) \nabla u_{n} + \Big(\frac{1}{2} - \beta_0\Big) \nabla u_{n} + \beta_0 \nabla u_{n-1} \Big\|^2 
			\label{eq:u-n-beta-grad-L2} \\
			\leq& 
			2 \Big\| \beta_2 \nabla u_{n+1} + \Big( \frac{1}{2} -\beta_2 \Big) \nabla u_{n}\Big\|^2
			+ 2 \Big\| \Big( \frac{1}{2} - \beta_0\Big) \nabla u_{n} + \beta_0 \nabla u_{n-1} \Big\|^2  \notag \\
			\leq& 
			\underbrace{2 \Big[ \beta_2^2 \frac{2^2}{1+\theta}
			+ \Big( \frac{1}{2} -\beta_2 \Big)^2 \frac{2^2}{1-\theta} \Big]}_{2 - \theta^2} A_{n+1}  
			+ \underbrace{2 \Big[ \Big( \frac{1}{2} - \beta_0\Big)^2 \frac{2^2}{1+\theta} 
			+ \beta_0^2 \frac{2^2}{1 - \theta} \Big]}_{2 - \theta^2} A_n. \notag 
		\end{align}
\begin{confidential}
	\color{mygrey}
	\begin{align*}
		\|u_{n,\beta}\|^2
		& 
		= \|\sum_{\ell=0}^2 \beta_\ell u_{n-1+\ell} \|^2
		= \| \beta_2 u_{n+1} + \beta_1 u_{n} + \beta_0 u_{n-1} \|^2  
		\\
		& 
		= 
		\| \beta_2 u_{n+1} + (1 - \beta_2 - \beta_0) u_{n} + \beta_0 u_{n-1} \|^2  
		= 
		\Big\| \beta_2 u_{n+1} + \Big( \frac{1}{2} -\beta_2 \Big) u_{n} + \Big( \frac{1}{2} - \beta_0\Big) u_{n} + \beta_0 u_{n-1} \Big\|^2  
		\\
		& 
		\leq 
		2 \Big\| \beta_2 u_{n+1} + \Big( \frac{1}{2} -\beta_2 \Big) u_{n}\Big\|^2
		+ 
		2 \Big\| \Big( \frac{1}{2} - \beta_0\Big) u_{n} + \beta_0 u_{n-1} \Big\|^2 
		\\
		& 
		=
		2 \Big\| \beta_2 \frac{2}{\sqrt{1+\theta}} \frac{\sqrt{1+\theta}}{2} u_{n+1} 
		+ \Big( \frac{1}{2} -\beta_2 \Big) \frac{2}{\sqrt{1-\theta}}  \frac{\sqrt{1-\theta}}{2}  u_{n}\Big\|^2
		\\
		& \qquad
		+ 
		2 \Big\| \Big( \frac{1}{2} - \beta_0\Big) \frac{2}{\sqrt{1+\theta}} \frac{\sqrt{1+\theta}}{2}  u_{n} 
		+ \beta_0 \frac{2}{\sqrt{1 - \theta}} \frac{\sqrt{1- \theta}}{2}  u_{n-1} \Big\|^2 
		\\
		\tag{{\color{darkviolet}use C-S: $(\sum a_i b_i)^2 \leq (\sum a_i^2) (\sum b_i^2)$}}
		\\
		& 
		\leq
		2 \Big[ \beta_2^2 \frac{2^2}{1+\theta}
		+ \Big( \frac{1}{2} -\beta_2 \Big)^2 \frac{2^2}{1-\theta} \Big]
		\cdot	
		\big\| [ u_{n+1} , u_n ] \big\|^2_{G(\theta)}		
		\\
		& \qquad
		+ 
		2 \Big[ \Big( \frac{1}{2} - \beta_0\Big)^2 \frac{2^2}{1+\theta} 
		+ \beta_0^2 \frac{2^2}{1 - \theta} \Big]
		\cdot
		\| [  u_{n} , u_{n-1} ]\|_{G(\theta)}^2 
		\\
		& 
		=
		2 \Big[ \beta_2^2 \frac{2^2}{1+\theta}
		+ \Big( \frac{1}{2} -\beta_2 \Big)^2 \frac{2^2}{1-\theta} \Big]
		\cdot	
		\big\| [ u_{n+1} , u_n ] \big\|^2_{G(\theta)}		
		\\
		& \qquad
		+ 
		2 \Big[ \Big( \frac{1}{2} - \beta_0\Big)^2 \frac{2^2}{1+\theta} 
		+ \beta_0^2 \frac{2^2}{1 - \theta} \Big]
		\cdot
		\| [  u_{n} , u_{n-1} ] \|_{G(\theta)}^2 
		,
	\end{align*}
	\begin{align*}
		&
		2 \Big[ \beta_2^2 \frac{2^2}{1+\theta}
		+ \Big( \frac{1}{2} -\beta_2 \Big)^2 \frac{2^2}{1-\theta} \Big] \\
		& = 
		2 \Big[ \Big( \frac{1}{4} (2  + \theta - {\theta }^{2} ) \Big)^2 \frac{2^2}{1+\theta}
		+ \Big( \frac{1}{2} - \frac{1}{4} (2  + \theta - {\theta }^{2} ) \Big)^2 \frac{2^2}{1-\theta} \Big]\\
		& = 
		2 \Big[ \frac{1}{16} (2  + \theta - {\theta }^{2} ) ^2 \frac{2^2}{1+\theta}
		+ \frac{1}{16} \theta^2 (1 - {\theta })^{2}  \frac{2^2}{1-\theta} \Big]
		\\
		& = 
		\frac{1}{2}
		\Big[  (2  + \theta - {\theta }^{2} ) ^2 \frac{1}{1+\theta}
		+ \theta^2 (1 - {\theta })^{2}  \frac{1}{1-\theta} \Big] \\
		& = 
		\frac{1}{2} \frac{1}{(1+\theta)(1-\theta)}
		\Big[ (1 - {\theta }) (2  + \theta - {\theta }^{2} ) ^2 
		+ \theta^2 (1 + {\theta }) (1 - {\theta })^{2}  \Big]\\
		& = 
		\frac{1}{2} \frac{1}{(1+\theta)(1-\theta)}
		\Big[ 4 (1 - {\theta }) + 4\theta(1-\theta)^2 
		+ \underbrace{\theta^2(1 - {\theta })^{3} 
		+ \theta^2 (1 + {\theta }) (1 - {\theta })^{2} }_{ = \theta^2(1-\theta)^2 (1-\theta + 1+ \theta) = 2 \theta^2(1-\theta)^2 } \Big] \\
		& = 
		\frac{1}{2} \frac{1}{(1+\theta)(1-\theta)}
		\Big[ 4 (1 - {\theta }) + 4\theta(1-\theta)^2 +  2 \theta^2(1-\theta)^2 \Big] \\
		& = 
		\frac{1}{(1+\theta)}\Big[ 2 + 2\theta(1-\theta)+ \theta^2(1-\theta)  
		\Big]
		= 
		\frac{1}{(1+\theta)}\Big[2 + 2 \theta - 2\theta^2+ \theta^2 - \theta^3
		\Big]\\
		& = 
		\frac{1}{(1+\theta)}\big(  2  + 2\theta - \theta^2 - \theta^3 \big)
		= 
		\frac{1}{(1+\theta)}
		\big(  2(1  + \theta) - \theta^2(1 + \theta) \big)
		= 2 - \theta^2, \\
		\\
		&2 \Big[ \Big( \frac{1}{2} - \beta_0\Big)^2 \frac{2^2}{1+\theta} 
		+ \beta_0^2 \frac{2^2}{1 - \theta} \Big]
		\\
		& = 
		8 \Big[ \Big( \frac{1}{2} -  \frac{1}{4}  ( 2  - \theta - {\theta }^{2} ) \Big)^2 \frac{1}{1+\theta} 
		+ \frac{1}{4^2}  ( 2  - \theta - {\theta }^{2} )^2
		\frac{1}{1 - \theta} \Big]
		\\
		& = 
		8 \Big[   \frac{1}{16}  \theta^2 (1 + {\theta })^2 \frac{1}{1+\theta} 
		+ \frac{1}{16}  ( 2  - \theta - {\theta }^{2} )^2
		\frac{1}{1 - \theta} \Big]\\
		& 
		= 
		\frac{1}{2}   \frac{1}{(1+\theta)(1-\theta)}  \Big[    \theta^2 (1 + {\theta })^2 (1-\theta)
		+ ( 2  - \theta - {\theta }^{2} )^2 (1 + \theta) \Big]\\
		& 
		= 
		\frac{1}{2}   \frac{1}{(1+\theta)(1-\theta)}  \Big[    \theta^2 (1 + {\theta })^2 (1-\theta)
		+ \Big( 4  - 4\theta(1+\theta)  + \theta^{2} (1+\theta)^2 \Big) (1 + \theta) \Big]\\
		& 
		= 
		\frac{1}{2}   \frac{1}{(1+\theta)(1-\theta)}  \Big[    \theta^2 (1 + {\theta })^2 (1-\theta)
		+ 4(1+\theta)  - 4\theta(1+\theta)^2  + \theta^{2} (1+\theta)^3 \Big]
		\\
		& 
		= 
		\frac{1}{2}   \frac{1}{(1+\theta)(1-\theta)}  \Big[   
		4(1+\theta)  - 4\theta(1+\theta)^2  + \underbrace{ \theta^{2} (1+\theta)^3  +  \theta^2 (1 + {\theta })^2 (1-\theta) }_{= \theta^{2} (1+\theta)^2 (1+\theta + 1 - \theta) = 2 \theta^{2} (1+\theta)^2} \Big]
		\\
		& 
		= 
		\frac{1}{2}   \frac{1}{(1+\theta)(1-\theta)}  \Big[   
		4(1+\theta)  - 4\theta(1+\theta)^2  + 2 \theta^{2} (1+\theta)^2  \Big]
		= 
		\frac{1}{(1-\theta)}  \Big[   
		2  - 2\theta(1+\theta)  + \theta^{2} (1+\theta) \Big] \\
		& 
		= 
		\frac{1}{(1-\theta)} \Big[ 2 - 2\theta - 2\theta^2 + \theta^{2} + \theta^3 \Big]
		= \frac{1}{(1-\theta)} \Big[ 2 - 2\theta - \theta^2  + \theta^3 \Big]
		= 
		\frac{1}{(1-\theta)} \Big[ 2(1 - \theta) - \theta^2( 1 - \theta) \Big]
		\\
		& = 2 - \theta^2.	
	\end{align*}
	\normalcolor
\end{confidential}
		By \eqref{eq:u-boundL2-K2} in Theorem \ref{thm:stab-L2}, \eqref{eq:An-ineq-1}, \eqref{eq:u-n-beta-grad-L2}, 
		\begin{align*}
			& 
			A_{n+1} - A_n \leq
			\frac{\Delta t}{2\nu} \|f\|^2_{\infty}
			+ \frac{27 C_{\Omega}^2 (2 - \theta^2) K_2^2}{16 \nu^3} \Delta t (A_{n+1} + A_n) \| \nabla u_{n,\beta} \|^2.
		\end{align*}
		which results in \eqref{eq:H1-ineq1} by using the notation in \eqref{eq:kappa-1}.
		Similarly, by \eqref{eq:stabilityL2-0} in Corollary \ref{coro:L2-bound-no-init}, 
		we replace $K_2^2$ by $2 \rho_0$ to have \eqref{eq:H1-ineq1-uniform}. 

		Then we apply \eqref{eq:DLN-NSE-dual} on $\Delta t ( \alpha _{2}{u_{n+1}} + {\alpha _{1}}{u_{n}} + {\alpha _{0}}{u_{n-1}} )$, 
		use $G$-stability identity in \eqref{eq:G-stab-eq}, Cauchy-Schwarz inequality, and discard the numerical dissipation term to obtain
\begin{confidential}
	\color{darkgreen}
	\begin{align}  
		& 
		\| {\alpha _{2}}{u_{n+1}}+{\alpha _{1}}{u_{n}}+{\alpha _{0}}{u_{n-1}} \|^2
		+ \nu \Delta t  \big( \nabla {u_{n,\beta }} , \nabla ( \alpha _{2}{u_{n+1}} + {\alpha _{1}}{u_{n}} + {\alpha _{0}}{u_{n-1}} ) \big)
		\notag
		\\
		& 
		+ 
		\Delta t b \big( u_{n,\beta },u_{n,\beta } , \alpha _{2}{u_{n+1}} + {\alpha _{1}}{u_{n}} + {\alpha _{0}}{u_{n-1}} \big)
		= 
		\Delta t \big( f(t_{n,\beta }) , \alpha _{2}{u_{n+1}} + {\alpha _{1}}{u_{n}} + {\alpha _{0}}{u_{n-1}} \big), \notag
	\end{align}
	\begin{align}  
		& 
		\| {\alpha _{2}}{u_{n+1}}+{\alpha _{1}}{u_{n}}+{\alpha _{0}}{u_{n-1}} \|^2
		+ \nu \Delta t  
		\Big(
		\begin{Vmatrix}
			\nabla u_{n+1} \\ \nabla u_{n}
		\end{Vmatrix}_{G(\theta )}^{2}
		-
		\begin{Vmatrix}
			\nabla u_{n} \\ \nabla u_{n-1}
		\end{Vmatrix}_{G(\theta )}^{2}
		+
		\underbrace{\Big\|\sum_{\ell =0}^{2}{a_{\ell }} \, \nabla u_{n-1+\ell } \Big\|^{2}}_{\geq 0}
		\Big)
		\notag
		\\
		& 
		= 
		- 
		\Delta t b \big( u_{n,\beta },u_{n,\beta } , \alpha _{2}{u_{n+1}} + {\alpha _{1}}{u_{n}} + {\alpha _{0}}{u_{n-1}} \big)
		+
		\Delta t \big( f(t_{n,\beta }) , \alpha _{2}{u_{n+1}} + {\alpha _{1}}{u_{n}} + {\alpha _{0}}{u_{n-1}} \big)
		\notag
		\\
		& 
		\tag{use Cauchy-Schwarz}
		\\
		& 
		\leq 
		- 
		\Delta t  b \big( u_{n,\beta },u_{n,\beta } , \alpha _{2}{u_{n+1}} + {\alpha _{1}}{u_{n}} + {\alpha _{0}}{u_{n-1}} \big)
		+
		(\Delta t)^2  \frac{1}{2} \|f(t_{n,\beta })\|^2 
		+ \frac{1}{2} \| \alpha _{2}{u_{n+1}} + {\alpha _{1}}{u_{n}} + \alpha _{0} u_{n-1} \|^2, \notag
	\end{align}
	\normalcolor
\end{confidential}
		\begin{align}  
			& 
			\frac{1}{2} \| {\alpha _{2}}{u_{n+1}}+{\alpha _{1}}{u_{n}}+{\alpha _{0}}{u_{n-1}} \|^2
			+ \nu \Delta t  A_{n+1}
			\label{eq:ineq1}
			\\
			& 
			\leq 
			\nu \Delta t  A_n + \frac{(\Delta t )^2}{2} \|f\|^2_{\infty}
			- \Delta t b \big( u_{n,\beta },u_{n,\beta } , \alpha _{2}{u_{n+1}} + {\alpha _{1}}{u_{n}} + {\alpha _{0}}{u_{n-1}} \big).
			\notag
		\end{align}
		By skew-symmetric property of $b$ in \eqref{eq:b-skew-symm}, Ladyzhenskaya inequality in \eqref{eq:Lady-ineq}, Young's inequality, \eqref{eq:u-n-beta-grad-L2} and \eqref{eq:u-boundL2-K2}  
		\begin{align}
			&- b(u_{n,\beta},u_{n,\beta}, \alpha _{2}{u_{n+1}}+{\alpha _{1}}{u_{n}} + {\alpha _{0}}{u_{n-1}}) 
			\label{eq:b-n-beta-H1} \\
			=& - b \big( u_{n,\beta}, u_{n,\beta}, 
			\frac{\alpha_2}{\beta_2} u_{n,\beta} -  \frac{\alpha_2}{\beta_2} u_{n,\beta} 
			+ \alpha _{2}{u_{n+1}} + {\alpha _{1}}{u_{n}} + {\alpha _{0}}{u_{n-1}}  \big) \notag \\
			=& - b \big( u_{n,\beta}, u_{n,\beta}, - \frac{\alpha_2}{\beta_2}  \beta_1 u_n  -  \frac{\alpha_2}{\beta_2}  \beta_0 u_{n-1} 
			+ {\alpha _{1}}{u_{n}} + {\alpha _{0}}{u_{n-1}} \big) \notag \\
			=& b \big( u_{n,\beta}, \big( \alpha_{1} - \frac{\alpha_2}{\beta_2} \beta_1 \big) u_n  
			+ \big( \alpha _{0} -  \frac{\alpha_2}{\beta_2} \beta_0 \big) u_{n-1}, u_{n,\beta} \big) \notag \\
			\leq&
			2^{-1/2} \| u_{n,\beta} \|  \| \nabla u_{n,\beta} \|
			\Big\| \big( \alpha _{1} -  \frac{\alpha_2}{\beta_2}  \beta_1 \big) \nabla u_n  
			+ \big( \alpha_{0} - \frac{\alpha_2}{\beta_2} \beta_0 \big) \nabla u_{n-1} \Big\| \notag \\
			\leq& \frac{\nu}{2 (2 - \theta^2)} \| \nabla u_{n,\beta} \|^2
			+ \frac{(2 - \theta^2)}{4 \nu} \| u_{n,\beta} \|^2 
			\Big\| \big( \alpha _{1} -  \frac{\alpha_2}{\beta_2}  \beta_1 \big) \nabla u_n  
			+ \big( \alpha_{0} - \frac{\alpha_2}{\beta_2} \beta_0 \big) \nabla u_{n-1} \Big\|^2 \notag \\
			\leq& \frac{\nu}{2} (A_{n+1} + A_{n})
			+ \frac{(2 - \theta^2) K_2^2}{4 \nu \beta_2^2} 
			\Big\| \big( \alpha _{1} \beta_2 - \alpha_2 \beta_1 \big) \nabla u_n + \big( \alpha_{0} \beta_2  - \alpha_2 \beta_0 \big) \nabla u_{n-1} \Big\|^2. \notag 
		\end{align}
\begin{confidential}
	\color{mygrey}
	Now we proceed in a similar way to \eqref{eq:Ladyz1} and bound the last term in the RHS of \eqref{eq:ineq1}:
	\begin{align*}
		& 
		- b (u_{n,\beta}, u_{n,\beta}, \alpha _{2}{u_{n+1}} + {\alpha _{1}}{u_{n}} + {\alpha _{0}}{u_{n-1}}) \\
		& 
		= - b \big( u_{n,\beta} ,  u_{n,\beta} , 
		\frac{\alpha_2}{\beta_2} u_{n,\beta} -  \frac{\alpha_2}{\beta_2} u_{n,\beta} 
		+ \alpha _{2}{u_{n+1}} + {\alpha _{1}}{u_{n}} + {\alpha _{0}}{u_{n-1}}  \big) \notag \\
		& 
		\notag
		= - b \big( u_{n,\beta}, u_{n,\beta}, \frac{\alpha_2}{\beta_2}  u_{n,\beta} 
		-  \frac{\alpha_2}{\beta_2} \big( \cancel{\beta_2 u_{n+1}} + \beta_1 u_n + \beta_0 u_{n-1} )
		+ \cancel{\alpha _{2}{u_{n+1}}} + {\alpha _{1}}{u_{n}} + {\alpha _{0}}{u_{n-1}} \big) \\
		& 
		= - b \big( u_{n,\beta}, u_{n,\beta}, 
		\frac{\alpha_2}{\beta_2}  \xcancel{u_{n,\beta} }
		-  \frac{\alpha_2}{\beta_2} \big( \beta_1 u_n + \beta_0 u_{n-1} )
		+ {\alpha _{1}}{u_{n}} + {\alpha _{0}}{u_{n-1}}  \big)
		\tag{use of skewsymmetry}
		\\
		& 
		= - b \big( u_{n,\beta} ,  u_{n,\beta} , 
		-  \frac{\alpha_2}{\beta_2}  \beta_1 u_n  -  \frac{\alpha_2}{\beta_2}  \beta_0 u_{n-1} 
		+ {\alpha _{1}}{u_{n}} + {\alpha _{0}}{u_{n-1}}  \big)
		\notag
		\\
		& 
		= - b \big( u_{n,\beta} ,  u_{n,\beta} , 
		\big[ \alpha _{1} -  \frac{\alpha_2}{\beta_2}  \beta_1 \big] u_n  
		+ \big[ \alpha _{0} -  \frac{\alpha_2}{\beta_2}  \beta_0 \big] u_{n-1} 
		\big)
		\tag{use skew symmetry}
		\\
		& 
		= b\big( u_{n,\beta}, \big[ \alpha _{1} -  \frac{\alpha_2}{\beta_2}  \beta_1 \big] u_n  
		+ \big[ \alpha _{0} - \frac{\alpha_2}{\beta_2}  \beta_0 \big] u_{n-1},
		u_{n,\beta} \big)
		\tag{use Ladyzhenskaya}
		\\
		& 
		\leq
		2^{-1/2} \| u_{n,\beta} \|  \| \nabla u_{n,\beta} \|
		\Big\| \big[ \alpha _{1} -  \frac{\alpha_2}{\beta_2}  \beta_1 \big] \nabla u_n  
		+ \big[ \alpha _{0} -  \frac{\alpha_2}{\beta_2}  \beta_0 \big] \nabla u_{n-1} \Big\| \notag 
	\end{align*}
		\begin{align}
			& 
			\| (\alpha _{1}\beta_2 -  \alpha_2 \beta_1) \frac{2}{\sqrt{1+\theta}} \frac{\sqrt{1+\theta}}{2} \nabla u_n  
			+ ( \alpha _{0} \beta_2 -  \alpha_2 \beta_0)  \frac{2}{\sqrt{1-\theta}} \frac{\sqrt{1- \theta}}{2}  \nabla u_{n-1} 
			\|^2 
			\notag
			\\
			& 
			\leq
			\Big( (\alpha _{1}\beta_2 -  \alpha_2 \beta_1)^2 \frac{4}{1+\theta} 
			+ ( \alpha _{0} \beta_2 -  \alpha_2 \beta_0)^2  \frac{4}{1-\theta} \Big)
			\Big( \frac{1+\theta}{4} \| \nabla u_n \|^2  +  \frac{1- \theta}{4}  \|\nabla u_{n-1} \|^2 \Big)
			\notag
			\\
			& 
			=
			\Big( (\alpha _{1}\beta_2 -  \alpha_2 \beta_1)^2 \frac{4}{1+\theta} 
			+ ( \alpha _{0} \beta_2 -  \alpha_2 \beta_0)^2  \frac{4}{1-\theta}  \Big) \| [ \nabla u_n , \nabla u_{n-1} ] \|^2_{G(\theta)},
			\notag 
		\end{align}
		\begin{align*}
			\alpha_1\beta_2 - \alpha_2\beta_1 
			& 
			= (-\theta)\frac{1}{4}(2+\theta-\theta^2)
			- \frac{1}{2}(\theta+1)\ \frac{1}{2}\theta^2
			= - \frac{1}{4} \big( 2\theta + \theta^2 - \cancel{\theta^3}
			+ \cancel{\theta^3} + \theta^2 \big)
			\\
			& 
			= - \frac{1}{2} \theta (\theta+1), 
			\\
			\alpha_0\beta_2 - \alpha_2\beta_0 
			& 
			= \frac{1}{2} (\theta - 1) \frac{1}{4} (2+\theta-\theta^2) - \frac{1}{2} (\theta+1) \frac{1}{4}(2-\theta-\theta^2)
			\\
			& 
			= \frac{1}{8} \Big( ( \xcancel{2\theta} + \theta^2  - \cancel{\theta^3} - 2 - \bcancel{\theta} + \theta^2)
			-  (\xcancel{2\theta} - \theta^2 - \cancel{\theta^3} + 2 - \bcancel{\theta} - \theta^2) \Big)
			\\
			& 
			= \frac{1}{8} (4\theta^2 - 4)
			= \frac{1}{2} (\theta - 1)(\theta+1), 
		\end{align*}
		and
		\begin{align*}
			& 
			(\alpha _{1}\beta_2 -  \alpha_2 \beta_1)^2 \frac{4}{1+\theta} 
			+ ( \alpha _{0} \beta_2 -  \alpha_2 \beta_0)^2  \frac{4}{1-\theta} 
			\\
			& 
			= 
			\frac{1}{4} \theta^2 (1+\theta)^{\bcancel 2} \frac{4}{\bcancel{1+\theta}} 
			+ \frac{1}{4} (1 - \theta)^{\cancel{2}} (1 + \theta)^2 \frac{4}{\cancel{1-\theta}} 		
			\\
			& 
			= 
			\theta^2 (1+\theta) + (1 - \theta) (1 + \theta)^2
			= 
			(1+\theta) \big( \cancel\theta^2 + 1 - \cancel \theta^2\big)
			= 1+\theta.
		\end{align*}
	\normalcolor
\end{confidential}
		By similar argument to \eqref{eq:u-n-beta-grad-L2},
		\begin{align}
			&\Big\| \big( \alpha _{1} \beta_2 - \alpha_2 \beta_1 \big) \nabla u_n + \big( \alpha_{0} \beta_2  - \alpha_2 \beta_0 \big) \nabla u_{n-1} \Big\|^2 
			\label{eq:u-n-beta-grad-L2-2} \\
			\leq& \Big[ \underbrace{(\alpha _{1}\beta_2 - \alpha_2 \beta_1)^2 \frac{4}{1+\theta} + ( \alpha _{0} \beta_2 - \alpha_2 \beta_0)^2 \frac{4}{1-\theta}}_{= 1+\theta} \Big] A_{n}.  \notag 
		\end{align}
		By \eqref{eq:b-n-beta-H1} and \eqref{eq:u-n-beta-grad-L2-2}, 
		\eqref{eq:ineq1} becomes 
		\begin{align*}
			&\frac{\nu \Delta t}{2} A_{n+1} 
			\leq
			\frac{3 \nu \Delta t}{2} A_{n}
			+ \frac{4 (2 - \theta^2) K_2^2 \Delta t}{\nu (2 - \theta)^2(1+\theta)} A_n + \frac{(\Delta t )^2}{2} \|f\|^2_{\infty},
		\end{align*}
		which yields \eqref{eq:H1-ineq2}. 
		If $n$ is large enough such that $(n-2) \Delta t > T_{\ast}$, by \eqref{eq:stabilityL2-0} in Corollary \ref{coro:L2-bound-no-init}, we replace $K_2^2$ by $2 \rho_0$ to obtain \eqref{eq:H1-ineq2-uniform}.
\begin{confidential}
	\color{darkblue}
	\begin{align*}
		&\nu \Delta t  A_{n+1} \\ 
		\leq& 
		\nu \Delta t  A_n + \frac{(\Delta t )^2}{2} \|f\|^2_{\infty}
		+ \frac{\nu \Delta t}{2} (A_{n+1} + A_{n})
		+ \frac{(2 - \theta^2) (1 + \theta) K_2^2 \Delta t}{4 \nu \beta_2^2} 
		A_n, \\
		&\frac{\nu \Delta t}{2} A_{n+1} \\
		\leq& \frac{3 \nu \Delta t}{2} A_{n}
		+ \frac{(2 - \theta^2) (1 + \theta) K_2^2 \Delta t}{4 \nu \big( \frac{1}{4}(2 + \theta - \theta^2) \big)^2} A_n + \frac{(\Delta t )^2}{2} \|f\|^2_{\infty} \\
		=& \frac{3 \nu \Delta t}{2} A_{n}
		+ \frac{(2 - \theta^2) (1 + \theta) K_2^2 \Delta t}{\frac{1}{4} \nu (2 - \theta)^2(1+\theta)^2} A_n + \frac{(\Delta t )^2}{2} \|f\|^2_{\infty}\\
		=& \frac{3 \nu \Delta t}{2} A_{n}
		+ \frac{4 (2 - \theta^2) K_2^2 \Delta t}{\nu (2 - \theta)^2(1+\theta)} A_n + \frac{(\Delta t )^2}{2} \|f\|^2_{\infty}. 
	\end{align*}
	\begin{align*}
		A_{n+1} \leq
		3 A_n + \frac{8 (2 - \theta^2) K_2^2 }{\nu^2 (2 - \theta)^2(1+\theta)} A_n + \frac{\Delta t}{\nu} \|f\|^2_{\infty}
		= \Big( 3 + \frac{8 (2 - \theta^2) K_2^2 }{\nu^2 (2 - \theta)^2(1+\theta)} \Big) A_n + \frac{\Delta t}{\nu} \|f\|^2_{\infty}.
	\end{align*}
	\normalcolor
\end{confidential}
	\end{proof}

	\begin{lemma}
		\label{lemma:Gron-ineq}
		Given $k > 0$, an integer $n_{\ast} > 0$, and positive sequences $\{ \xi_n \}$, $\{ \eta_n \}$, $\{ \zeta_n \}$ such that 
		\begin{align*}
			\xi_n \leq \xi_{n-1} (1 + k \eta_{n-1}) + k \zeta_n, \qquad n = 1, 2, \cdots, n_{\ast}, 
		\end{align*}
		we have for any $n \in \{ 2, \cdots, n_{\ast} \}$, 
		\begin{align*}
			\xi_n \leq \xi_0 \exp \Big( \sum_{i=0}^{n-1} k \eta_i \Big) + \sum_{i=1}^{n} k \zeta_i \exp \Big( \sum_{j=i}^{n-1} k \eta_j \Big) + k \zeta_n. 
		\end{align*}
	\end{lemma}
	\begin{proof}
		See \cite[pp.35-36]{TW06_SIAMNA}.
	\end{proof}

	We use Lemma \ref{lemma:Gron-ineq} to prove the boundedness of $\| \nabla u_n \|$ on finite time interval $[0,T]$. 
	\begin{theorem}
		\label{thm:H1-bound-finite}
		Under the time step condition \eqref{eq:dt-limit-1}, for $n = 2, 3, \cdots, N: = \lfloor T/ \Delta t \rfloor$, we have 
		\begin{align}
			A_{n} \leq& K_{3} \Big( \| u_1 \|, \| u_0 \|, \| \nabla u_1 \|, \| \nabla u_0 \|, \|f\|_{\infty}, (n-1) \Delta t \Big), 
			\label{eq:An-bound-K3}
		\end{align}
		where $K_3$ is the following function increaing in all of its arguments
		\begin{align*}
			K_3 = &\bigg\{ \!
			\begin{Vmatrix}
				\nabla u_{1} \\ \nabla u_0
			\end{Vmatrix}_{G(\theta)}^2	\!+ \! 
			\frac{\kappa_1 C_{\Delta t} \|f\|^2_{\infty}}{\nu^2} 
			\Big( 2
			\begin{Vmatrix}
				u_{1} \\ u_0
			\end{Vmatrix}_{G(\theta)}^2 \!+ \frac{(n-1) \Delta t}{\nu} \|f\|^2_{\infty} \Big)
			\!+\! \frac{(n-1) \Delta t}{2 \nu} \|f\|^2_{\infty} \!
			\bigg\}   \\
			&\times \exp \bigg( \frac{\kappa_1 (\kappa_2 +1)}{\nu} \Big( 2
			\begin{Vmatrix}
				u_{1} \\ u_0
			\end{Vmatrix}_{G(\theta)}^2	
			+ \frac{(n-1) \Delta t}{\nu} \|f\|^2_{\infty} \Big) \bigg)
		\end{align*}
		and $C_{\Delta t}$ is the upper bound for $\Delta t$ in the restriction \eqref{eq:dt-limit-1}. 
	\end{theorem}
	\begin{proof}
		We combine \eqref{eq:H1-ineq1}-\eqref{eq:H1-ineq2} and obtain, for $n = 1, 2, \cdots, N-1$ 
		\begin{align*}
			A_{n+1} - A_{n} 
			\leq& \frac{\Delta t}{2 \nu} \|f \|^2_{\infty} 
			+ \Delta t \kappa_1 \|\nabla u_{n,\beta}\|^2 \Big( \kappa_2 A_n + \frac{\Delta t}{\nu} \|f\|^2_{\infty} \Big)  
			+ \Delta t \kappa_1 A_n \|\nabla u_{n,\beta}\|^2 \\
			=& \kappa_1 (\kappa_2 + 1) \Delta t \|\nabla u_{n,\beta}\|^2 A_n 
			+ \frac{\kappa_1 \Delta t}{\nu} \|f\|^2_{\infty} \Delta t \|\nabla u_{n,\beta}\|^2 
			+ \frac{\Delta t}{2 \nu} \|f \|^2_{\infty} \\
			\leq& \kappa_1 (\kappa_2 + 1) \Delta t \|\nabla u_{n,\beta}\|^2 A_n 
			+ \frac{\kappa_1 C_{\Delta t}}{\nu} \|f\|^2_{\infty} \Delta t \|\nabla u_{n,\beta}\|^2 
			+ \frac{\Delta t}{2 \nu} \|f \|^2_{\infty}.
		\end{align*} 
		We apply the discrete Gr\"onwall in Lemma \ref{lemma:Gron-ineq} to to the above inequality and let 
		\begin{align*}
			&k = \Delta t, \quad \xi_n = A_{n+1}, \quad 
			\eta_{n-1} = \kappa_1 (\kappa_2 + 1) \|\nabla u_{n,\beta}\|^2, \\
			&\zeta_n = \frac{\kappa_1 C_{\Delta t}}{\nu} \|f\|^2_{\infty} \|\nabla u_{n,\beta}\|^2 
			+ \frac{1}{2 \nu} \|f \|^2_{\infty}
		\end{align*}
		in Lemma \ref{lemma:Gron-ineq} to obtain 
\begin{confidential}
	\color{darkblue}
	\begin{align*}
		A_{n+1} \leq& A_1 \exp\Big(\kappa_1 (\kappa_2 + 1)\sum_{i=1}^{n} \Delta t \|\nabla u_{i,\beta}\|^2 \Big) 
		+ \sum_{i=1}^{n} \Delta t \zeta_i \exp\Big( \kappa_1 (\kappa_2 +1)\sum_{j=i}^{n-1} \Delta t \|\nabla u_{j,\beta}\|^2 \Big) \notag \\
		&\qquad \qquad \qquad \qquad \qquad \qquad \qquad \qquad + \Delta t \zeta_n \notag \\
		\leq& \Big( A_1 + \sum_{i=1}^{n} \Delta t \zeta_i \Big)
		\exp\Big( \kappa_1 (\kappa_2 +1)\sum_{i=1}^{n} \Delta t \|\nabla u_{i,\beta}\|^2 \Big). 
	\end{align*}
	\normalcolor
\end{confidential}
		\begin{align}
			A_{n+1} 
			\leq& \Big( A_1 + \sum_{i=1}^{n} \Delta t \zeta_i \Big)
			\exp\Big( \kappa_1 (\kappa_2 +1)\sum_{i=1}^{n} \Delta t \|\nabla u_{i,\beta}\|^2 \Big). 
			\label{eq:An-bound-ineq1} 
		\end{align}
		By \eqref{eq:stabilityL2-3}, we have
		\begin{align}
			\Delta t \sum_{i=1}^{n} \| \nabla u_{i,\beta} \|^2 
			\leq \frac{1}{\nu} \Big( 2
			\begin{Vmatrix}
				u_{1} \\ u_0
			\end{Vmatrix}_{G(\theta)}^2	
			+ \frac{n \Delta t}{\nu} \|f\|^2_{\infty} \Big),
			\label{eq:u-nabla-sum}
		\end{align}
		and 
		\begin{align}
			\sum_{i=1}^{n} \Delta t \zeta_i 
			=& \frac{\kappa_1 C_{\Delta t}}{\nu} \|f\|^2_{\infty} \sum_{i=1}^{n} \Delta t \|\nabla u_{n,\beta} \|^2
			+ \frac{ n\Delta t}{2 \nu} \|f\|^2_{\infty} 
			\label{eq:zeta-sum} \\
			\leq& \frac{\kappa_1 C_{\Delta t}}{\nu^2} \|f\|^2_{\infty} 
			\Big( 2
			\begin{Vmatrix}
				u_{1} \\ u_0
			\end{Vmatrix}_{G(\theta)}^2 + \frac{n \Delta t}{\nu} \|f\|^2_{\infty} \Big)
			+ \frac{n \Delta t}{2 \nu} \|f\|^2_{\infty}.  \notag 
		\end{align}
		We combine \eqref{eq:An-bound-ineq1}-\eqref{eq:zeta-sum} to obtain \eqref{eq:An-bound-K3}.
	\end{proof}

	Then we extend the result to infinite time by the following more general version of discrete Gr\"onwall inequality. 
	\begin{lemma}
		\label{lemma:Gron-uniform}
		Given $k > 0$, positive integers $n_1$, $n_2$, $n_{\ast}$ such that $n_1 < n_{\ast}$, $n_1 + n_2 + 1 \leq n_{\ast}$, and positive sequences $\{ \xi_n \}$, $\{ \eta_n \}$, $\{ \zeta_n \}$ such that 
		\begin{align*}
			\xi_n \leq \xi_{n-1} (1 + k \eta_{n-1}) + k \zeta_n, \qquad n = n_1, \cdots, n_{\ast}, 
		\end{align*}
		and given the bounds 
		\begin{align*}
			\sum_{n=n'}^{n'+n_2} k \eta_n \leq a_1(n_1,n_{\ast}), \quad
			\sum_{n=n'}^{n'+n_2} k \zeta_n \leq a_2(n_1,n_{\ast}), \quad
			\sum_{n=n'}^{n'+n_2} k \xi_n \leq a_3(n_1,n_{\ast}), 
		\end{align*}
		for any $n'$ such that $n_1 \leq n' \leq n_{\ast} - n_2$,
		we have 
		\begin{align}
			\xi_n \leq \Big( \frac{a_3(n_1, n_{\ast})}{k n_2} + a_{2}(n_1,n_{\ast}) \Big) \exp\big( a_1 (n_1,n_{\ast}) \big),
		\end{align}
		for any $n$ such that $n_1 + n_2 + 1 \leq n \leq n_{\ast}$. 
	\end{lemma}
	\begin{proof}
		See \cite[pp.37-38]{TW06_SIAMNA}.
	\end{proof}
	\ \\ 

	We fix some positive number $r > 5 C_{\Delta t}$ and have 
	$N_r = \lfloor r/ \Delta t \rfloor > 5$ if the time step restriction in \eqref{eq:dt-limit-1} holds. 
	We denote 
	\begin{align}
		&K_4 (\| f\|_{\infty}, T_{\ast}): = \frac{(1+\theta)\rho_0}{ 2 \widehat{C}_h(\theta)} \exp \Big( \frac{\nu \lambda_1 T_\ast}{4 C_{\epsilon}(\theta)} \Big)
		+ \frac{ T_{\ast} + 2 C_{\Delta t}}{\nu} \| f \|_{\infty}^2, 
		\label{eq:K-4} \\
		&K_5 (\| \nabla u_1 \|, \| \nabla u_0 \|, \| f \|_{\infty}, T_{\ast},r) 
		\label{eq:K-5} \\
		&:= \Big( A_1 + \frac{\kappa_1 C_{\Delta t} \|f\|^2_{\infty} K_4}{\nu^2} + \frac{T_{\ast} + 2 C_{\Delta t}}{2 \nu} \|f\|^2_{\infty} \Big) \exp \big(\frac{\kappa_1 (\kappa_2 + 1)K_4}{\nu}\big), \notag \\
		&K_6 (\| \nabla u_1 \|, \| \nabla u_0 \|, \| f \|_{\infty}, T_{\ast})
		\notag \\
		&:= \frac{(2 \beta_2 -1)r}{ 4(4+3\theta) K_5}  \exp \Big( \frac{\kappa_3 (\kappa_4 + 1)}{\nu} \big( 2 \rho_0 
		+ \frac{r}{\nu \lambda_1} \| f \|^2_{\infty} \big) \Big), \notag \\
		&\rho_1(\| f\|_{\infty},r) \label{eq:rho-1} \\
		&:=\bigg\{1 + 
			\frac{16 \rho_0}{\nu (2 \beta_2 -1) r} 
			\Big( 1 + \frac{4 \rho_0^4(\beta_0^2 + \beta_1^2)^2 } { \nu^4 (2\beta_2-1)^3} \Big) 
			+ \frac{32}{ \nu^2 (2 \beta_2-1)^2 \lambda_1} 
			\Big(1 + \frac{\rho_0^2 (\beta_0^2 + \beta_1^2)^2}{\nu^4 (2 \beta_2-1)^2} \Big) \| f \|_{\infty}^2 \notag \\
			&\ +\frac{\kappa_3 C_{\Delta t}}{\nu^2} \| f \|^2_{\infty}
			\Big( 2 \rho_0 + \frac{r}{\nu \lambda_1} \| f \|^2_{\infty} \Big)
			+ \frac{ r \|f \|^2_{\infty}}{2 \nu} \bigg\} 
			\exp \Big( \! \frac{2 \kappa_3 (\kappa_4 + 1)}{\nu} \big( 2 \rho_0 
			+ \frac{r}{\nu \lambda_1} \| f \|^2_{\infty} \big) \!\Big),\notag \\
		&\rho_2(\| f\|_{\infty},r) \label{eq:rho-2} \\
		&:=\Big[ \rho_1 + \frac{\kappa_3 C_{\Delta t} \| f \|_{\infty}^2}{\nu^2} \Big( 2 (1+\theta) \rho_0 + \frac{r}{\nu} \| f \|_{\infty}^2 \Big) + \frac{r}{2 \nu} \| f \|_{\infty}^2 \Big] \times \notag \\
		&\qquad \times \exp \Big( \frac{2 \kappa_3(\kappa_4 + 1)}{\nu} \Big( 2 (1+\theta) \rho_0 + \frac{r}{\nu} \| f \|_{\infty}^2 \Big) \Big), \notag \\
		&\rho_3 (\| f \|_{\infty}, r)
		:= \frac{(2 \beta_2 -1)r}{ 4(4+3\theta) \rho_1} \exp \Big( \frac{\kappa_3 (\kappa_4 + 1)}{\nu} \big( 2 \rho_0 
		+ \frac{r}{\nu \lambda_1} \| f \|^2_{\infty} \big) \Big). \notag 
	\end{align}

	\begin{theorem}
		Under the time step condition 
		\begin{align}
			\Delta t < \min \big\{ C_{\Delta t}, K_6, \rho_3 \big\}, 
			\label{eq:dt-cond-limit-2} 
		\end{align}
		for $N$ large enough such that $(N - 2) \Delta t > T_{\ast} + r$, we have 
		\begin{align}
			A_{N} \leq \rho_2, 
			\label{eq:AN-rho-2-bound}
		\end{align}
		where $\rho_2$ in \eqref{eq:rho-2} is independent of the initial conditions $u_1$, $u_0$.
	\end{theorem}
	\begin{proof}
		Let $N_0$ be the minumum integer such that $(N_0 - 2)> T_{\ast}$.
		By \eqref{eq:H1-ineq1-uniform} and \eqref{eq:H1-ineq2-uniform} in Lemma \ref{lemma:auxi-ineq}, for $n = N_0 - 1, N_0, \cdots$, we have 
\begin{confidential}
	\color{darkblue}
	\begin{align*}
		A_{n+1} - A_{n} 
		\leq& \frac{\Delta t}{2\nu} \| f \|_{\infty}^2 
		+ \Delta t \kappa_3 A_{n+1} \| \nabla u_{n,\beta} \|^2 
		+ \Delta t \kappa_3 A_{n} \| \nabla u_{n,\beta} \|^2 \\
		\leq& \frac{\Delta t}{2\nu} \| f \|_{\infty}^2 
		+ \Delta t \kappa_3 \| \nabla u_{n,\beta} \|^2 \big( \kappa_4 A_n + \frac{\Delta t}{\nu} \| f \|_{\infty}^2 \big) 
		+ \Delta t \kappa_3 A_{n} \| \nabla u_{n,\beta} \|^2 \\
		=& \frac{\Delta t}{2\nu} \| f \|_{\infty}^2 
		+ \Delta t \kappa_3 \kappa_4 \| \nabla u_{n,\beta} \|^2 A_n 
		+ \Delta t \kappa_3 \| \nabla u_{n,\beta} \|^2 A_{n}
		+ \frac{{\Delta t}^2 \kappa_3}{\nu} \| \nabla u_{n,\beta} \|^2 \| f \|_{\infty}^2 \\
		\leq& \frac{\Delta t}{2\nu} \| f \|_{\infty}^2
		+ \kappa_3 (\kappa_4 + 1) \Delta t \| \nabla u_{n,\beta} \|^2 A_{n}
		+ \frac{\kappa_3 C_{\Delta t}}{\nu} \| f \|_{\infty}^2 \Delta t \| \nabla u_{n,\beta} \|^2.
	\end{align*}
	\normalcolor
\end{confidential}
		\begin{align}
			A_{n+1} - A_{n} \leq&
			\kappa_3 (\kappa_4 + 1) \Delta t \| \nabla u_{n,\beta} \|^2 A_{n} 
			+ \frac{\kappa_3 C_{\Delta t} \| f \|_{\infty}^2}{\nu} \Delta t \| \nabla u_{n,\beta} \|^2 
			+ \frac{\Delta t \| f \|_{\infty}^2}{2\nu}, 
			\label{eq:An-uniform-ineq1} 
		\end{align} 
		where $\kappa_3$ and $\kappa_4$ in \eqref{eq:rho-1}-\eqref{eq:rho-2} are independent of the initial conditions $u_0$ and $u_{1}$. 
		Now we fixed some integer $r >0$ large enough such that $N_r = \lfloor r/ \Delta t \rfloor > 5$.
		We apply Lemma \ref{lemma:Gron-uniform} with $n_1 = N_0+1$, $n_2 = N_r - 3$, $n_{\ast} = n_1 + n_2 + 1= N_0 + N_r-1$ 
		and set 
		\begin{align*}
			&k = \Delta t, \quad \xi_n = A_{n+1}, \quad 
			\eta_{n-1} = \kappa_3 (\kappa_4 + 1) \|\nabla u_{n,\beta}\|^2, \\
			&\zeta_n = \frac{\kappa_3 C_{\Delta t}}{\nu} \|f\|^2_{\infty} \|\nabla u_{n,\beta}\|^2 
			+ \frac{1}{2 \nu} \|f \|^2_{\infty}, 
		\end{align*}
		in Lemma \ref{lemma:Gron-uniform} to have for $n' = N_0+1$, $N_0 + 2$ 
\begin{confidential}
	\color{darkblue}
	\begin{align*}
		&n_1 = N_0 + 1 \leq n' \leq = n_{\ast} - n_2 = n_1 + 1 = N_0 +2, \\
		&n_1 + n_2 + 1 = N_0 + 1 + N_r - 3 + 1 = N_0 + N_r - 1 = n_{\ast}.  
	\end{align*}
	\normalcolor
\end{confidential}
		\begin{align}
			&\sum_{n=n'}^{n'+n_2} k \eta_n 
			\label{eq:H1-eta-sum} \\
			=&
			\kappa_3 (\kappa_4 + 1) \sum_{n=n'}^{n'+n_2} \Delta t \| \nabla u_{n,\beta} \|^2  \notag \\
			\leq& \frac{\kappa_3 (\kappa_4 + 1)}{\nu} \Big( 2
			\begin{Vmatrix}
				u_{n'} \\ u_{n'-1} 
			\end{Vmatrix}_{G(\theta)}^2	
			+ \frac{((n'+n_2 +1) - n')\Delta t}{\nu \lambda_1} \| f \|^2_{\infty} \Big) \tag{\text{by} \eqref{eq:stabilityL2-3}} \\
			\leq& \frac{\kappa_3 (\kappa_4 + 1)}{\nu} \Big( \frac{1 + \theta}{2} \underbrace{\| u_{n'} \|^2}_{\leq 2 \rho_0} + \frac{1 - \theta}{2} \underbrace{\| u_{n'-1} \|^2}_{\leq 2 \rho_0} 
			+ \frac{(n_2 + 1) \Delta t}{\nu \lambda_1} \| f \|^2_{\infty} \Big) 
			\tag{\text{by} \eqref{eq:stabilityL2-0}} \\
			\leq& \frac{\kappa_3 (\kappa_4 + 1)}{\nu} \Big( 2 \rho_0 
			+ \frac{r}{\nu \lambda_1} \| f \|^2_{\infty} \Big), \notag 
		\end{align}
		and
		\begin{align*}
			\sum_{n=n'}^{n'+n_2} k \xi_n 
			=& \sum_{n=n'}^{n'+n_2} \Delta t \big( \frac{1+\theta}{4} \| \nabla u_{n+1} \|^2 + \frac{1-\theta}{4}\| \nabla u_{n} \|^2 \big) \\
			\leq& \sum_{n=n'}^{n'+n_2+1} \Delta t \big( \frac{1 + \theta}{4} + \frac{1 - \theta}{4} \big) \| \nabla u_n \|^2  \\
			=& \frac{1}{2} \sum_{n=N_0+1}^{N_0+ N_r} \Delta t \| \nabla u_n \|^2
		\end{align*}
		By the similar argument to the proof of \eqref{eq:u_{n+1}L^2(H^1)}, we have for $i \geq N_0$ and $N \geq i+1$
		\begin{align*}
			& 
			\nu \Delta t \frac{(2\beta_2-1)}{8} 
			\sum_{j=i+1}^{N} \| \nabla u_{j} \|^2 \\ 
			\leq&
			\Big(1 + \frac{(2 \rho_0)^2 (\beta_0^2 + \beta_1^2)^2}{\nu^4 (2\beta_2-1)^3} \Big) 
			\begin{Vmatrix}
				{u_{i}} \\
				{u_{i-1}}
			\end{Vmatrix}
			_{G(\theta )}^{2}
			+ \Delta t \nu \Big[ \frac{1}{4} \|\nabla u_{i} \|^2
			+ \Big( \frac{\beta_0}{2} + \frac{(2\beta_2-1)}{8} \Big)\|\nabla u_{i-1} \|^2 \Big] \notag \\
			& 
			+ \frac{(N-i) \Delta t}{ \nu (2 \beta_2-1) \lambda_1} 
			\Big( 2 + \frac{(2 \rho_0)^2 (\beta_0^2 + \beta_1^2)^2}{2 \nu^4 (2 \beta_2-1)^2} \Big) \| f \|_{\infty}^2 \notag \\
			\leq& \Big(1 + \frac{4 \rho_0^2 (\beta_0^2 + \beta_1^2)^2}{\nu^4 (2\beta_2-1)^3} \Big) \rho_0 
			+ \nu \Delta t \Big( \frac{1}{4} \| \nabla u_i \|^2 + \frac{(1 - \theta)(4 + 3\theta)}{16} \| \nabla u_{i-1} \|^2 \Big) \notag \\
			& + \frac{(N-i) \Delta t}{ \nu (2 \beta_2-1) \lambda_1} 
			\Big( 2 + \frac{2 \rho_0^2 (\beta_0^2 + \beta_1^2)^2}{\nu^4 (2 \beta_2-1)^2} \Big) \| f \|_{\infty}^2, \notag 
		\end{align*}
\begin{confidential}
	\color{darkblue}
	\begin{align*}
		& 
		\nu \Delta t \frac{(2\beta_2-1)}{8} 
		\sum_{j=i+1}^{N} \| \nabla u_{j} \|^2 \\ 
		\leq&
		\Big(1 + \frac{(2 \rho_0)^2 (\beta_0^2 + \beta_1^2)^2}{\nu^4 (2\beta_2-1)^3} \Big) 
		\begin{Vmatrix}
			{u_{i}} \\
			{u_{i-1}}
		\end{Vmatrix}
		_{G(\theta )}^{2}
		+ 
		\Delta t \nu \Big[ \frac{1}{4} \|\nabla u_{i} \|^2
		+ \Big( \frac{\beta_0}{2} + \frac{(2\beta_2-1)}{8} \Big)\|\nabla u_{i-1} \|^2 \Big]
		\notag \\
		& 
		+ \frac{(N-i) \Delta t}{ \nu (2 \beta_2-1) \lambda_1} 
		\Big( 2 + \frac{(2 \rho_0)^2 (\beta_0^2 + \beta_1^2)^2}{2 \nu^4 (2 \beta_2-1)^2} \Big) \| f \|_{\infty}^2
		\notag \\
		\leq&
		\Big(1 + \frac{4 \rho_0^2 (\beta_0^2 + \beta_1^2)^2}{\nu^4 (2\beta_2-1)^3} \Big) 
		\Big(\frac{1+\theta}{4} 2 \rho_0 + \frac{1-\theta}{4}2 \rho_0 \Big) 
		+ \nu \Delta t \Big( \frac{1}{4} \| \nabla u_i \|^2 + \frac{4 \beta_0 + 2 \beta_2 - 1 }{8} \| \nabla u_{i-1} \|^2 \Big) \notag \\
		& 
		+ \frac{(N-i) \Delta t}{ \nu (2 \beta_2-1) \lambda_1} 
		\Big( 2 + \frac{2 \rho_0^2 (\beta_0^2 + \beta_1^2)^2}{\nu^4 (2 \beta_2-1)^2} \Big) \| f \|_{\infty}^2 \notag \\
		=& \Big(1 + \frac{4 \rho_0^2 (\beta_0^2 + \beta_1^2)^2}{\nu^4 (2\beta_2-1)^3} \Big) \rho_0 
		+ \nu \Delta t \Big( \frac{1}{4} \| \nabla u_i \|^2 + \frac{(1 -\theta)(4+3\theta)}{16} \| \nabla u_{i-1} \|^2 \Big) \notag \\
		& + \frac{(N-i) \Delta t}{ \nu (2 \beta_2-1) \lambda_1} 
		\Big( 2 + \frac{2 \rho_0^2 (\beta_0^2 + \beta_1^2)^2}{\nu^4 (2 \beta_2-1)^2} \Big) \| f \|_{\infty}^2 \notag  \\
		\leq& \Big(1 + \frac{4 \rho_0^2 (\beta_0^2 + \beta_1^2)^2}{\nu^4 (2\beta_2-1)^3} \Big) \rho_0 
		+ \frac{\nu \Delta t (4+3\theta)}{4} A_{i}
		+ \frac{(N-i) \Delta t}{ \nu (2 \beta_2-1) \lambda_1} 
		\Big( 2 + \frac{2 \rho_0^2 (\beta_0^2 + \beta_1^2)^2}{\nu^4 (2 \beta_2-1)^2} \Big) \| f \|_{\infty}^2
	\end{align*}
	\begin{align*}
		&\frac{4 \beta_0 + 2 \beta_2 - 1 }{8} 
		= \frac{1}{8} \Big( \cancel{4} \frac{1}{\cancel{4}}(2 - \theta - \theta^2) + 2 \frac{1}{4}(2 + \theta - \theta^2) - 1 \Big) 
		= \frac{1}{8} \Big( 2 - \theta - \theta^2 + \frac{1}{2} (2 + \theta - \theta^2) - 1 \Big) \\
		=& \frac{1}{8} \Big( 2 - \theta - \theta^2 \cancel{+ 1} + \frac{1}{2}\theta - \frac{1}{2}\theta^2 \cancel{- 1} \Big)
		= \frac{1}{8} \Big( 2 - \frac{1}{2}\theta - \frac{3}{2} \theta^2 \Big)
		= \frac{4 - \theta - 3 \theta^2}{16} 
		= \frac{(1 - \theta)(4 + 3\theta)}{16}.
	\end{align*}
	\begin{align*}
		\frac{1}{4} \| \nabla u_i \|^2 + \frac{(1 -\theta)(4+3\theta)}{16} \| \nabla u_{i-1} \|^2 
		=& \frac{1}{1+\theta} \cdot \frac{1 + \theta}{4} \| \nabla u_i \|^2
		+ \frac{4+3\theta}{4} \cdot \frac{1-\theta}{4} \| \nabla u_{i-1} \|^2 \\
		\leq& \frac{4+3\theta}{4} \Big( \frac{1 + \theta}{4} \| \nabla u_i \|^2 + \frac{1-\theta}{4} \| \nabla u_{i-1} \|^2 \Big)
		= \frac{4+3\theta}{4} A_i 
	\end{align*}
	\normalcolor
\end{confidential}
		or equivalently, 
		\begin{align}
			\sum_{j=i+1}^{N} \Delta t  \| \nabla u_{j} \|^2 
			\leq& \frac{8 \rho_0}{\nu (2 \beta_2 -1)} 
			\Big(1 + \frac{4 \rho_0^2 (\beta_0^2 + \beta_1^2)^2}{\nu^4 (2\beta_2-1)^3} \Big)
			+ \frac{2 \Delta t (4 + 3\theta)}{(2 \beta_2 -1)}A_i  
			\label{eq:An-uniform-ineq2} \\
			&
			+ \frac{16 (N-i) \Delta t}{ \nu^2 (2 \beta_2-1)^2 \lambda_1} 
			\Big(1 + \frac{\rho_0^2 (\beta_0^2 + \beta_1^2)^2}{\nu^4 (2 \beta_2-1)^2} \Big) \| f \|_{\infty}^2.
			\notag 
		\end{align}
		We let $i = N_0$ and $N = N_0 + N_r$ in \eqref{eq:An-uniform-ineq2}
		and obtain 
		\begin{align}
			\sum_{n=n'}^{n'+n_2} k \xi_n 
			\label{eq:H1-xi-sum}
			\leq& \frac{1}{2} \sum_{n=N_0+1}^{N_0+ N_r} \Delta t \| \nabla u_n \|^2 \\
			\leq& 
			\frac{4 \rho_0}{\nu (2 \beta_2 -1)} 
			\Big( 1 + \frac{4 \rho_0^4(\beta_0^2 + \beta_1^2)^2 } { \nu^4 (2\beta_2-1)^3}  \Big) 
			+ \frac{\Delta t (4 + 3\theta)}{ (2 \beta_2 -1)} A_{N_0} \notag \\
			& 
			+ \frac{8 N_r \Delta t}{ \nu^2 (2 \beta_2-1)^2 \lambda_1} 
			\Big(1 + \frac{\rho_0^2 (\beta_0^2 + \beta_1^2)^2}{\nu^4 (2 \beta_2-1)^2} \Big) \| f \|_{\infty}^2. \notag 
		\end{align}
		By \eqref{eq:H1-eta-sum}
\begin{confidential}
	\color{darkblue}
	\begin{align*}
		\sum_{n = n'}^{n'+ n_2} k \zeta_n 
		=&\sum_{n = n'}^{n'+ n_2} \Big( \frac{\kappa_3 C_{\Delta t}}{\nu} \|f\|^2_{\infty} \Delta t \|\nabla u_{n,\beta}\|^2 + \frac{1}{2 \nu} \|f \|^2_{\infty} \Delta t \Big) \\
		=& \frac{\kappa_3 C_{\Delta t}}{\nu} \|f\|^2_{\infty}  \sum_{n = n'}^{n'+ n_2} \Delta t \|\nabla u_{n,\beta}\|^2
		+ \frac{\|f \|^2_{\infty}}{2 \nu} \sum_{n = n'}^{n'+ n_2} \Delta t \\
		\leq& \frac{\kappa_3 C_{\Delta t}}{\nu} \|f\|^2_{\infty} \frac{1}{\nu}
		\Big( 2 \rho_0 + \frac{r}{\nu \lambda_1} \| f \|^2_{\infty} \Big)
		+ \frac{\|f \|^2_{\infty}}{2 \nu} \big( \underbrace{\bcancel{n'}+ n_2 \bcancel{- n'} + 1}_{n_2 + 1 = N_r - 2} \big) \Delta t \\
		\leq& \frac{\kappa_3 C_{\Delta t}}{\nu^2} \| f \|^2_{\infty}
		\Big( 2 \rho_0 + \frac{r}{\nu \lambda_1} \| f \|^2_{\infty} \Big)
		+ \frac{ r \|f \|^2_{\infty}}{2 \nu} 
	\end{align*}
	\normalcolor
\end{confidential}
		\begin{align}
			\sum_{n = n'}^{n'+ n_2} k \zeta_n 
			=& \frac{\kappa_3 C_{\Delta t}}{\nu} \|f\|^2_{\infty}  \sum_{n = n'}^{n'+ n_2} \Delta t \|\nabla u_{n,\beta}\|^2
			+ \frac{\|f \|^2_{\infty}}{2 \nu} \sum_{n = n'}^{n'+ n_2} \Delta t \label{eq:H1-zeta-sum} \\
			\leq& \frac{\kappa_3 C_{\Delta t}}{\nu^2} \| f \|^2_{\infty}
			\Big( 2 \rho_0 + \frac{r}{\nu \lambda_1} \| f \|^2_{\infty} \Big)
			+ \frac{ r \|f \|^2_{\infty}}{2 \nu}. \notag 
		\end{align}
		We combine \eqref{eq:H1-eta-sum}, \eqref{eq:H1-xi-sum}, \eqref{eq:H1-zeta-sum} 
		and apply Lemma \ref{lemma:Gron-uniform} to \eqref{eq:An-uniform-ineq1}
		\begin{align}
			&A_{N_0 + N_r} \label{eq:A-N0-Nr-ineq} \\
			\leq& \bigg\{ 
			\frac{4 \rho_0}{\nu (2 \beta_2 -1) (N_r - 3) \Delta t} 
			\Big( 1 + \frac{4 \rho_0^4(\beta_0^2 + \beta_1^2)^2 } { \nu^4 (2\beta_2-1)^3} \Big) 
			+ \frac{\Delta t (4+3\theta)}{ (2 \beta_2 -1) (N_r - 3) \Delta t} 
			A_{N_0} \notag \\
			& 
			+ \frac{8 N_r \Delta t}{ \nu^2 (2 \beta_2-1)^2 \lambda_1 (N_r - 3) \Delta t} 
			\Big(1 + \frac{\rho_0^2 (\beta_0^2 + \beta_1^2)^2}{\nu^4 (2 \beta_2-1)^2} \Big) \| f \|_{\infty}^2 \notag \\
			& +\frac{\kappa_3 C_{\Delta t}}{\nu^2} \| f \|^2_{\infty}
			\Big( 2 \rho_0 + \frac{r}{\nu \lambda_1} \| f \|^2_{\infty} \Big)
			+ \frac{ r \|f \|^2_{\infty}}{2 \nu} \bigg\} 
			\exp \Big( \frac{\kappa_3 (\kappa_4 + 1)}{\nu} \big( 2 \rho_0 
			+ \frac{r}{\nu \lambda_1} \| f \|^2_{\infty} \big) \Big) \notag \\
			\leq& \bigg\{ 
			\frac{16 \rho_0}{\nu (2 \beta_2 -1) r} 
			\Big( 1 + \frac{4 \rho_0^4(\beta_0^2 + \beta_1^2)^2 } { \nu^4 (2\beta_2-1)^3} \Big) 
			+ \frac{4 (4+3\theta) \Delta t}{ (2 \beta_2 -1) r} A_{N_0}
			\notag \\
			& 
			+ \frac{32}{ \nu^2 (2 \beta_2-1)^2 \lambda_1} 
			\Big(1 + \frac{\rho_0^2 (\beta_0^2 + \beta_1^2)^2}{\nu^4 (2 \beta_2-1)^2} \Big) \| f \|_{\infty}^2 \notag \\
			& +\frac{\kappa_3 C_{\Delta t}}{\nu^2} \| f \|^2_{\infty}
			\Big( 2 \rho_0 + \frac{r}{\nu \lambda_1} \| f \|^2_{\infty} \Big)
			+ \frac{ r \|f \|^2_{\infty}}{2 \nu} \bigg\} 
			\exp \Big( \frac{\kappa_3 (\kappa_4 + 1)}{\nu} \big( 2 \rho_0 
			+ \frac{r}{\nu \lambda_1} \| f \|^2_{\infty} \big) \Big), \notag
		\end{align}
		where the last inequality is done by the facts that $N_r \Delta t < r$ and that $(N_r - 3) \Delta t > \frac{1}{4}(N_r + 1) \Delta t > \frac{r}{4}$.
		By \eqref{eq:N-T-ast-L2}
		\begin{align}
			\| u_1 \|^2 + \| u_0 \|^2
			= \frac{\rho_0}{ \widehat{C}_h(\theta)} \exp \Big( \frac{\nu \lambda_1 T_\ast}{4 C_{\epsilon}(\theta)} \Big). 
			\label{eq:initial-bound-L2}
		\end{align}
		By \eqref{eq:An-bound-K3} in Theorem \ref{thm:H1-bound-finite}, 
		\eqref{eq:K-4}, \eqref{eq:K-5}, \eqref{eq:initial-bound-L2}
		and the fact that $(N_0 - 1) \Delta t < T_{\ast} + 2 \Delta t$, 
		we have $A_{N_0} \leq K_5$. 
\begin{confidential}
	\color{darkblue}
	\begin{align*}
		2
		\begin{Vmatrix}
			u_{1} \\ u_0
		\end{Vmatrix}_{G(\theta)}^2 + \frac{(n-1) \Delta t}{\nu} \|f\|^2_{\infty}
		\leq& \frac{1 + \theta}{2} \big( \| u_1 \|^2 + \| u_0 \|^2 \big)
		+ \frac{(N_0-1) \Delta t}{\nu} \|f\|^2_{\infty} \\
		\leq& \frac{(1+\theta)\rho_0}{ 2 \widehat{C}_h(\theta)} \exp \Big( \frac{\nu \lambda_1 T_\ast}{4 C_{\epsilon}(\theta)} \Big)
		+ \frac{T_{\ast} + 2 C_{\Delta t}}{\nu} \|f\|^2_{\infty} = K_4.
	\end{align*}
	\begin{align*}
		A_{N_0} \leq \Big( A_1 + \frac{\kappa_1 C_{\Delta t} \|f\|^2_{\infty} K_4}{\nu^2} + \frac{T_{\ast} + 2 C_{\Delta t}}{2 \nu} \|f\|^2_{\infty} \Big) \exp \big(\frac{\kappa_1 (\kappa_2 + 1) K_4}{\nu}\big) = K_5.
	\end{align*}
	If $\Delta t$ small enough such that
	\begin{align*}
		\frac{4 (4+3\theta) K_5 \Delta t}{ (2 \beta_2 -1) r} 
		\leq \exp \bigg\{ \frac{\kappa_3 (\kappa_4 + 1)}{\nu} \Big( 2 \rho_0 
		+ \frac{r}{\nu \lambda_1} \| f \|^2_{\infty} \Big) \bigg\}
	\end{align*}
	\normalcolor
\end{confidential}
		If $\Delta t$ satisfies \eqref{eq:dt-cond-limit-2}, we have 
		\begin{align*}
			\frac{4 (4+3\theta) \Delta t}{ (2 \beta_2 -1) r} A_{N_0}
			\leq 
			\frac{4 (4+3\theta) \Delta t K_{5}}{ (2 \beta_2 -1) r} 
			\leq 
			\exp \Big( \frac{\kappa_3 (\kappa_4 + 1)}{\nu} \big( 2 \rho_0 
			+ \frac{r}{\nu \lambda_1} \| f \|^2_{\infty} \big) \Big), 
		\end{align*}
		and \eqref{eq:A-N0-Nr-ineq} becomes 
		\begin{align}
			&A_{N_0 + N_r} 
			\leq \rho_1 (\|f \|_{\infty}, r), \label{eq:A-N0-Nr-ineq-2} 
		\end{align}
		where $\rho_1 (\|f \|_{\infty}, r)$ in \eqref{eq:rho-1} is independent of initial conditions.
\begin{confidential}
	\color{darkblue}
	\begin{align*}
			&A_{N_0 + N_r} \\
			\leq& \bigg\{1 + 
			\frac{16 \rho_0}{\nu (2 \beta_2 -1) r} 
			\Big( 1 + \frac{4 \rho_0^4(\beta_0^2 + \beta_1^2)^2 } { \nu^4 (2\beta_2-1)^3} \Big) 
			+ \frac{32}{ \nu^2 (2 \beta_2-1)^2 \lambda_1} 
			\Big(1 + \frac{\rho_0^2 (\beta_0^2 + \beta_1^2)^2}{\nu^4 (2 \beta_2-1)^2} \Big) \| f \|_{\infty}^2 \notag \\
			& +\frac{\kappa_3 C_{\Delta t}}{\nu^2} \| f \|^2_{\infty}
			\Big( 2 \rho_0 + \frac{r}{\nu \lambda_1} \| f \|^2_{\infty} \Big)
			+ \frac{ r \|f \|^2_{\infty}}{2 \nu} \bigg\} 
			\exp \bigg\{ \! \frac{2 \kappa_3 (\kappa_4 + 1)}{\nu} \Big( 2 \rho_0 
			+ \frac{r}{\nu \lambda_1} \| f \|^2_{\infty} \Big) \!\bigg\} \notag \\
			=:& \rho_1 (\| f \|_{\infty}, r). \notag 
		\end{align*}
	\normalcolor
\end{confidential}
		We use similar argument to the proof of Theorem \ref{thm:H1-bound-finite} and replace the initial conditions by $u_{N_0 + N_r}$, $u_{N_0 + N_r - 1}$ to have for $n = N_0+N_r+1, \cdots, N_0 + 2N_r$
		\begin{align*}
			A_n \leq& 
			\bigg\{ A_{N_0 + N_r} + \frac{\kappa_3 C_{\Delta t} \| f \|_{\infty}^2}{\nu^2} \Big( 2
			\begin{Vmatrix}
				u_{N_0 + N_r} \\ u_{N_0 + N_r-1} 
			\end{Vmatrix}_{G(\theta)}^2 + 
			\frac{N_r \Delta t}{\nu} \| f \|_{\infty}^2 \Big) 
			+ \frac{N_r \Delta t}{2 \nu} \| f \|_{\infty}^2 \bigg\}  \\
			&\times \exp\Big(\frac{\kappa_3(\kappa_4 + 1)}{\nu} \Big( 2
			\begin{Vmatrix}
				u_{N_0 + N_r} \\ u_{N_0 + N_r-1} 
			\end{Vmatrix}_{G(\theta)}^2 + 
			\frac{N_r \Delta t}{\nu} \| f \|_{\infty}^2\Big) \Big). 
		\end{align*}
		By \eqref{eq:stabilityL2-0} in Corollary \ref{coro:L2-bound-no-init},  \eqref{eq:A-N0-Nr-ineq-2} and the fact $N_r \Delta t < r$, we have 
\begin{confidential}
	\color{darkblue}
	\begin{align*}
			A_n \leq& 
			\Big[ \rho_1 + \frac{\kappa_3 C_{\Delta t} \| f \|_{\infty}^2}{\nu^2} \Big( 2 (1+\theta) \rho_0 + \frac{r}{\nu} \| f \|_{\infty}^2 \Big) + \frac{r}{2 \nu} \| f \|_{\infty}^2 \Big] \times \\
			&\times \exp \Big( \frac{\kappa_3(\kappa_4 + 1)}{\nu} \Big( 2 (1+\theta) \rho_0 + \frac{r}{\nu} \| f \|_{\infty}^2 \Big) \Big)
			:=\rho_2 ( \| f \|_{\infty}, r). 
		\end{align*}
	\normalcolor
\end{confidential}
		\begin{align*}
			A_n \leq \rho_2 ( \| f \|_{\infty}, r), 
		\end{align*}
		where $\rho_2 (\|f \|_{\infty}, r)$ in \eqref{eq:rho-2} is independent of initial conditions. 
		Then we apply Lemma \ref{lemma:Gron-uniform} with $n_1 = N_0 + N_r + 1$, $n_2 = N_r - 3$, and $n_{\ast} = N_0+2 N_r-1$,
		\begin{align}
			&A_{N_0 + 2N_r}   
			\label{eq:An-initial-later}\\
			\leq& \bigg\{ 
			\frac{16 \rho_0}{\nu (2 \beta_2 -1) r} 
			\Big( 1 + \frac{4 \rho_0^4(\beta_0^2 + \beta_1^2)^2 } { \nu^4 (2\beta_2-1)^3} \Big) 
			+ \frac{4 (4+3\theta) \Delta t}{ (2 \beta_2 -1) r} A_{N_0+N_r}
			\notag \\
			& 
			+ \frac{32}{ \nu^2 (2 \beta_2-1)^2 \lambda_1} 
			\Big(1 + \frac{\rho_0^2 (\beta_0^2 + \beta_1^2)^2}{\nu^4 (2 \beta_2-1)^2} \Big) \| f \|_{\infty}^2 \notag \\
			& +\frac{\kappa_3 C_{\Delta t}}{\nu^2} \| f \|^2_{\infty}
			\big( 2 \rho_0 + \frac{r}{\nu \lambda_1} \| f \|^2_{\infty} \big)
			+ \frac{ r \|f \|^2_{\infty}}{2 \nu} \bigg\} 
			\exp \Big( \frac{\kappa_3 (\kappa_4 + 1)}{\nu} \big( 2 \rho_0 
			+ \frac{r}{\nu \lambda_1} \| f \|^2_{\infty} \big) \Big), \notag
		\end{align}
		By \eqref{eq:dt-cond-limit-2} and \eqref{eq:A-N0-Nr-ineq-2}, 
		\begin{align*} 
			\frac{4 (4+3\theta) \Delta t}{ (2 \beta_2 -1) r} A_{N_0+N_r} 
			\leq \frac{4 (4+3\theta) \Delta t \rho_1}{ (2 \beta_2 -1) r} 
			\leq \exp \Big( \frac{\kappa_3 (\kappa_4 + 1)}{\nu} \big( 2 \rho_0 
			+ \frac{r}{\nu \lambda_1} \| f \|^2_{\infty} \big) \Big).
		\end{align*}
		Hence \eqref{eq:An-initial-later} becomes 
		\begin{align}
			\label{eq:A-N0-2Nr-bound}
			&A_{N_0 + 2N_r} \leq \rho_1. 
		\end{align}
		We use \eqref{eq:A-N0-2Nr-bound} and similar argument to the proof of Theorem \ref{thm:H1-bound-finite} with initial condition $A_{N_0 + 2N_r}$ to obtain for $n = N_0 +2N_r +1, \cdots, N_0+3N_r$, 
		\begin{align*}
			A_n \leq \rho_2. 
		\end{align*}
		Then we apply Lemma \ref{lemma:Gron-uniform} with $n_1 = N_0 + 2N_r + 1$, $n_2 = N_r - 3$, and $n_{\ast} = N_0+3 N_r-1$,
		Under the time step \eqref{eq:dt-cond-limit-2}, we have $A_{N_0 + 3 N_r} \leq \rho_1$.
		We iterate the above process and have \eqref{eq:AN-rho-2-bound}  
		for $N$ large enough such that $(N - 2) \Delta t > T_{\ast} + r$. 
	\end{proof}

	\section{Conclusion}  \label{sec:conclusion} 
		In this report, we employ the family of semi-discrete-in-time DLN formulations for the two-dimensional Navier-Stokes equations and prove the uniform-in-time stability in $L^2$- and $H^1$-norm via 
		a newly-derived $G$-stability identity for the family of DLN schemes under uniform time grids in Appendix \ref{sec:appendix}.
		The stability in $L^2$-norm holds under this auxiliary identity and the time step restriction in \eqref{eq:dt-limit-1}.
		For the proof of stability in $H^1$-norm, we first establish the boundedness of solutions for the finite time interval and then use the discrete uniform Gr\"onwall lemma to derive the uniform upper bound for later time under the time step constraint in \eqref{eq:dt-cond-limit-2} by iteration.  
		Essentially, the discrete global attractors are independent of initial conditions and true for all values of method parameter $\theta \in (0,1)$. 
		Since the family of DLN methods enjoys the non-linear stability and high accuracy under arbitrary time grids, we are encouraged to extend the results to the variable time step scenario in the future.

	\section{Acknowledgement} Isabel Barrio Sanchez and Catalin Trenchea are partially supported by the National Science Foundation under grant DMS-2208220.

	\appendix 

	\section{Proof of Lemma \ref{lem:H-stab}} \label{sec:appendix}
	\begin{proof}
		If there exists $\epsilon>0$, $a, b, c \in \mathbb{R}$, and the positive definite matrix $H(\theta)$ in \eqref{eq:Hmatrix} such that \eqref{eq:H-stab} holds, then we obtain the following system by matching the coefficients of terms $\| u_{n+1} \|^2$, $\| u_n\|^2$, $\| u_{n-1} \|^2$, $(u_{n+1}, u_{n})$, $(u_{n+1}, u_{n-1})$, $(u_{n}, u_{n-1})$
\begin{confidential}
	\color{darkblue}
	For the left hand side of \eqref{eq:H-stab}
    \begin{align*}
        &\begin{Vmatrix}
                u_{n+1} \\
                u_{n}
        \end{Vmatrix}_{G(\theta)}^2
        - 
        \begin{Vmatrix}
            u_{n} \\
            u_{n-1}
        \end{Vmatrix}_{G(\theta)}^2
        + \Big\| \sum_{\ell}^{2} a_{\ell} u_{n-1+\ell} \Big\|^{2}
        + \frac{\nu \Delta t \lambda_1}{2} \| u_{n,\beta} \|^2 \\
        =& \Big( \frac{1+\theta}{4} \| u_{n+1} \|^2 + \frac{1 - \theta}{4} \| u_{n} \|^2 \Big)
        - \Big( \frac{1+\theta}{4} \| u_{n} \|^2 + \frac{1 - \theta}{4} \| u_{n-1} \|^2 \Big) \\   
        +& \big( a_2 u_{n+1} + a_1 u_n + a_0 u_{n-1}, a_2 u_{n+1} + a_1 u_n + a_0 u_{n-1} \big) \\
        +& \frac{\nu \Delta t \lambda_1}{2}
        \big( \beta_2 u_{n+1} + \beta_1 u_n + \beta_0 u_{n-1}, \beta_2 u_{n+1} + \beta_1 u_n + \beta_0 u_{n-1} \big) \\
        =&\frac{1+\theta}{4} \| u_{n+1} \|^2 + \Big(\frac{1 - \theta}{4} -  \frac{1+\theta}{4} \Big) \| u_{n} \|^2
        - \frac{1 - \theta}{4} \| u_{n-1} \|^2 \\
        +& a_2^2 \| u_{n+1} \|^2 + a_1^2 \| u_{n} \|^2 + a_0^2 \| u_{n-1} \|^2 + 2 a_2 a_1 ( u_{n+1}, u_{n} ) + 2 a_2 a_0 ( u_{n+1}, u_{n-1} ) \\
        +& 2 a_1 a_0 ( u_{n}, u_{n-1} ) \\
        +& \frac{\nu \Delta t \lambda_1}{2} \Big\{ \beta_2^2 \| u_{n+1} \|^2 + \beta_1^2 \| u_{n} \|^2 + \beta_0^2 \| u_{n-1} \|^2 + 2 \beta_2 \beta_1 ( u_{n+1}, u_{n} ) + 2 \beta_2 \beta_0 ( u_{n+1}, u_{n-1} ) \\
        &\quad \quad \qquad+ 2 \beta_1 \beta_0 ( u_{n}, u_{n-1} ) \Big\} \\
        =& \frac{1+\theta}{4} \| u_{n+1} \|^2 + \Big( \cancel{\frac{1}{4}} - \frac{\theta}{4} \cancel{- \frac{1}{4}} - \frac{\theta}{4} \Big) \| u_{n} \|^2
        - \frac{1 - \theta}{4} \| u_{n-1} \|^2 \\
        +& \Big(a_2^2 + \frac{\nu \Delta t \lambda_1 \beta_2^2}{2} \Big)
        \| u_{n+1} \|^2 + \Big(a_1^2 + \frac{\nu \Delta t \lambda_1 \beta_1^2}{2} \Big) \| u_{n} \|^2
        + \Big(a_0^2 + \frac{\nu \Delta t \lambda_1 \beta_0^2}{2} \Big)
        \| u_{n-1} \|^2 \\
        +& \Big( 2 a_2 a_1 + \nu \Delta t \lambda_1 \beta_2 \beta_1 \Big) ( u_{n+1}, u_{n} )
        + \Big( 2 a_2 a_0 + \nu \Delta t \lambda_1 \beta_2 \beta_0 \Big) ( u_{n+1}, u_{n-1} ) \\
        +& \Big( 2 a_1 a_0 + \nu \Delta t \lambda_1 \beta_1 \beta_0 \Big) ( u_{n}, u_{n-1} ) \\
        =& \Big( \frac{1+\theta}{4} + \frac{a_1^2}{4} + \frac{\nu \Delta t \lambda_1 \beta_2^2}{2} \Big) \| u_{n+1} \|^2 
        + \Big( a_1^2 + \frac{\nu \Delta t \lambda_1 \beta_1^2}{2} 
        - \frac{\theta}{2} \Big) \| u_{n} \|^2 \\
        +& \Big( \frac{a_1^2}{4} + \frac{\nu \Delta t \lambda_1 \beta_0^2}{2} - \frac{1 - \theta}{4} \Big) \| u_{n-1} \|^2
        + \Big[ 2 \big( - \frac{a_1}{2} \big) a_1 + \nu \Delta t \lambda_1 \beta_2 \beta_1 \Big] ( u_{n+1}, u_{n} ) \\
        +& \Big[ 2 \big( - \frac{a_1}{2} \big) \big( - \frac{a_1}{2} \big) + \nu \Delta t \lambda_1 \beta_2 \beta_0 \Big] ( u_{n+1}, u_{n-1}) 
        + \Big[ 2 a_1 \big( - \frac{a_1}{2} \big) + \nu \Delta t \lambda_1 \beta_1 \beta_0 \Big] ( u_{n}, u_{n-1} )
        \\
        =& \Big[ \frac{1+\theta}{4} + \frac{1}{4} \frac{\theta (1 - \theta)(1 + \theta)}{2} + \frac{\nu \Delta t \lambda_1 \beta_2^2}{2} \Big] \| u_{n+1} \|^2 
        + \Big[ \frac{\theta (1 - \theta)(1 + \theta)}{2} + \frac{\nu \Delta t \lambda_1 \beta_1^2}{2} - \frac{\theta}{2} \Big] \| u_{n} \|^2 \\ 
        +& \Big[ \frac{1}{4} \frac{\theta (1 - \theta)(1 + \theta)}{2} + \frac{\nu \Delta t \lambda_1 \beta_0^2}{2} - \frac{1 - \theta}{4} \Big] \| u_{n-1} \|^2
        + \Big( \nu \Delta t \lambda_1 \beta_2 \beta_1 - a_1^2 \Big) ( u_{n+1}, u_{n} ) \\
        +& \Big( \nu \Delta t \lambda_1 \beta_2 \beta_0 + \frac{a_1^2}{2}\Big) ( u_{n+1}, u_{n-1}) 
        + \Big( \nu \Delta t \lambda_1 \beta_1 \beta_0 - a_1^2 \Big) ( u_{n}, u_{n-1} ) \\
        =& \Big[ \frac{1+\theta}{4} \Big( \frac{2 + \theta - \theta^2}{2} \Big) + \frac{\nu \Delta t \lambda_1 \beta_2^2}{2} \Big] \| u_{n+1} \|^2 
        + \Big[ \frac{\theta}{2} \big( \cancel{1} - \theta^2 \cancel{- 1} \big) + \frac{\nu \Delta t \lambda_1 \beta_1^2}{2} \Big] \| u_{n} \|^2 \\
        +& \Big[ \frac{1 - \theta}{4} \Big( \frac{\theta + \theta^2 - 2 }{2} \Big) + \frac{\nu \Delta t \lambda_1 \beta_0^2}{2} \Big] \| u_{n-1} \|^2 + \Big( \nu \Delta t \lambda_1 \beta_2 \beta_1 - a_1^2 \Big) ( u_{n+1}, u_{n} ) \\ 
        +& \Big( \nu \Delta t \lambda_1 \beta_2 \beta_0 + \frac{a_1^2}{2}\Big) ( u_{n+1}, u_{n-1}) 
        + \Big( \nu \Delta t \lambda_1 \beta_1 \beta_0 - a_1^2 \Big) ( u_{n}, u_{n-1} ) \\
        =& \Big[ \frac{(1+\theta) (2 + \theta - \theta^2)}{8} + \frac{\nu \Delta t \lambda_1 \beta_2^2}{2} \Big] \| u_{n+1} \|^2 
        + \Big( \frac{\nu \Delta t \lambda_1 \beta_1^2}{2} - \frac{\theta^3}{2} \Big) \| u_{n} \|^2 \\ 
        +& \Big[ \frac{(1 - \theta)(\theta^2 + \theta - 2)}{8} + \frac{\nu \Delta t \lambda_1 \beta_0^2}{2} \Big] \| u_{n-1} \|^2 
        + \Big( \nu \Delta t \lambda_1 \beta_2 \beta_1 - a_1^2 \Big) ( u_{n+1}, u_{n} ) \\ 
        +& \Big( \nu \Delta t \lambda_1 \beta_2 \beta_0 + \frac{a_1^2}{2}\Big) ( u_{n+1}, u_{n-1}) 
        + \Big( \nu \Delta t \lambda_1 \beta_1 \beta_0 - a_1^2 \Big) ( u_{n}, u_{n-1} ) \\
    \end{align*}
	The right hand side of \eqref{eq:H-stab} takes the following form 
    \begin{align*}
        &(1 + \epsilon) 
		[u_{n+1}^{\top} \  u_n^{\top}]
        \begin{bmatrix}
            h_{11} \mathbb{I}_d & 0 \\
            0 & h_{22} \mathbb{I}_d
        \end{bmatrix}
        \begin{bmatrix}
            u_{n+1} \\ u_{n}
        \end{bmatrix}
        - 
		[u_{n}^{\top} \ u_{n-1}^{\top}]
        \begin{bmatrix}
            h_{11} \mathbb{I}_d & 0 \\
            0 & h_{22} \mathbb{I}_d
        \end{bmatrix}
        \begin{bmatrix}
            u_{n} \\ u_{n-1}
        \end{bmatrix} \\
        &+ \big\| a u_{n+1} + b u_{n} + c u_{n-1} \big\|^2 \\
        &= (1 + \epsilon) \big( h_{11} \| u_{n+1} \|^2 + h_{22} \| u_{n} \|^2 \big) - \big( h_{11} \| u_{n} \|^2 + h_{22} \| u_{n-1} \|^2 \big) \\
		&+ a^2 \| u_{n+1} \|^2
        + b^2 \| u_{n} \|^2 + c^2 \| u_{n-1} \|^2 \\
        &+ 2 ab (u_{n+1}, u_{n}) + 2 ac (u_{n+1}, u_{n-1})
        + 2 bc (u_{n}, u_{n-1}) \\
        &= \big[ (1 + \epsilon) h_{11} + a^2 \big] \| u_{n+1} \|^2
        + \big[ (1 + \epsilon) h_{22} - h_{11} + b^2 \big] \| u_{n} \|^2
        + \big( c^2 - h_{22} \big) \| u_{n-1} \|^2 \\
        &+ 2 ab  (u_{n+1}, u_{n}) + 2 ac (u_{n+1}, u_{n-1}) 
        + 2 bc  (u_{n}, u_{n-1}). 
    \end{align*}
	\normalcolor
\end{confidential}
		\begin{align} \label{eq:H-stab-system}
			\begin{cases} \displaystyle 
				(1 + \epsilon) h_{11} + a^2
				= \frac{(1+\theta) (2 + \theta - \theta^2)}{8} + \frac{\nu \Delta t \lambda_1 \beta_2^2}{2}, 
				& (\| u_{n+1} \|^2) \\ \displaystyle 
				(1 + \epsilon) h_{22} - h_{11} + b^2
				= \frac{\nu \Delta t \lambda_1 \beta_1^2}{2} - \frac{\theta^3}{2}, & (\| u_{n} \|^2) \\ \displaystyle  
				c^2 - h_{22} = \frac{(1 - \theta)(\theta^2 + \theta - 2)}{8} + \frac{\nu \Delta t \lambda_1 \beta_0^2}{2}, & (\| u_{n-1} \|^2) \\ \displaystyle 
				2 ab =\nu \Delta t \lambda_1 \beta_2 \beta_1- a_1^2, 
				& (u_{n+1}, u_{n}) \\ \displaystyle 
				2 ac =\nu \Delta t \lambda_1 \beta_2 \beta_0 + \frac{a_1^2}{2},
				& (u_{n+1}, u_{n-1}) \\ \displaystyle 
				2 bc = \nu \Delta t \lambda_1 \beta_1 \beta_0 - a_1^2,
				& (u_{n}, u_{n-1}). 
			\end{cases}
		\end{align}
		We add the above six equations together, and use the fact $\beta_2 + \beta_1 + \beta_0 = 1$ to have 
\begin{confidential}
	\color{darkblue}
	the left hand side becomes 
    \begin{align*}
        & \cancel{h_{11}} + \epsilon h_{11} + a^2
        \bcancel{+ h_{22}} + \epsilon h_{22} \cancel{- h_{11}} + b^2
        + c^2 \bcancel{- h_{22}} + 2 ab + 2 ac 
        + 2 bc  \\
        &= \epsilon h_{11} + \epsilon h_{22} + (a + b + c)^2, 
    \end{align*}
	and the right hand side becomes 
    \begin{align*}
        &\frac{(1+\theta) (2 + \theta - \theta^2)}{8} + \frac{\nu \Delta t \lambda_1 \beta_2^2}{2} 
        + \frac{\nu \Delta t \lambda_1 \beta_1^2}{2} - \frac{\theta^3}{2}
        + \frac{(1 - \theta)(\theta^2 + \theta - 2)}{8} + \frac{\nu \Delta t \lambda_1 \beta_0^2}{2} \\
        &+ \nu \Delta t \lambda_1 \beta_2 \beta_1 - a_1^2 
        + \nu \Delta t \lambda_1 \beta_2 \beta_0 + \frac{a_1^2}{2}
        + \nu \Delta t \lambda_1 \beta_1 \beta_0 - a_1^2 \\
        &= \frac{(1+\theta) (2 + \theta - \theta^2)}{8} - \frac{\theta^3}{2}
        + \frac{(1 - \theta)(\theta^2 + \theta - 2)}{8} 
        - a_1^2 + \frac{a_1^2}{2} - a_1^2  \\
        &+ \frac{\nu \Delta t \lambda_1 \beta_2^2}{2} 
        + \frac{\nu \Delta t \lambda_1 \beta_1^2}{2}
        + \frac{\nu \Delta t \lambda_1 \beta_0^2}{2} 
        + \nu \Delta t \lambda_1 \beta_2 \beta_1
        + \nu \Delta t \lambda_1 \beta_2 \beta_0
        + \nu \Delta t \lambda_1 \beta_1 \beta_0 \\
        &= \frac{ \cancel{2} + \theta \bcancel{- \theta^2} + 2 \theta \bcancel{+ \theta^2} - \theta^3 - 4 \theta^3 \bcancel{+ \theta^2} + \theta \cancel{- 2} - \theta^3 \bcancel{- \theta^2} + 2 \theta }{8} - \frac{3 a_1^2}{2} \\
        &+ \frac{\nu \Delta t \lambda_1}{2} 
        \big( \beta_2^2 + \beta_1^2 + \beta_0^2 + 2 \beta_2 \beta_1
        + 2 \beta_2 \beta_0 + 2 \beta_1 \beta_0 \big) \\
        &= \frac{6 \theta - 6 \theta^3}{8} - \frac{3}{2} \big( \frac{\theta - \theta^3}{2} \big) + \frac{\nu \Delta t \lambda_1}{2} 
        (\underbrace{\beta_2 + \beta_1 + \beta_0}_{=1})^2 \\
        &= \frac{3 \theta - 3 \theta^3}{4} - \frac{3 \theta - 3 \theta^3}{4} + \frac{\nu \Delta t \lambda_1}{2} \cdot 1 
        = \frac{\nu \Delta t \lambda_1}{2}. 
    \end{align*}
    where we use the fact that $\beta_2 + \beta_1 + \beta_0 = 1$ for the second to the last equality. 
	\normalcolor
\end{confidential}
		\begin{align*}
			\epsilon (h_{11} + h_{22}) + (a + b + c)^2
			= \frac{\nu \Delta t \lambda_1}{2}. 
		\end{align*}
		Now we assume that $\epsilon (h_{11} + h_{22})  = \mu \frac{\nu \Delta t \lambda_1}{2}$ for some $\mu \in (0,1)$ and obtain
		\begin{align*}
			(a + b + c)^2 = (1 - \mu) \frac{\nu \Delta t \lambda_1}{2}, 
			\quad \Leftrightarrow
			\quad b + (a + c) = \pm \sqrt{\frac{(1 - \mu) \nu \Delta t \lambda_1}{2}} := \pm \sqrt{x}.
		\end{align*}
		We consider the $'+'$ case $b + (a + c) = \sqrt{x}$ and add the fourth and sixth equations of system \eqref{eq:H-stab-system} together
\begin{confidential}
	\color{darkblue}
	\begin{align*}
        &2 ab + 2 bc 
        = \nu \Delta t \lambda_1 \beta_1 (\beta_2 + \beta_0)- 2 a_1^2, \\
        &2 b (a+c) = \nu \Delta t \lambda_1 \beta_1 (\beta_2 + \beta_0) - 2 a_1^2, \\
        &b (a+c) = \frac{\nu \Delta t \lambda_1 \beta_1 (\beta_2 + \beta_0)}{2} - a_1^2. 
    \end{align*}
	\begin{align*}
		b (a+c) &= \frac{\nu \Delta t \lambda_1 \beta_1 (\beta_2 + \beta_0)}{2} - a_1^2 \\
		&= \frac{\nu \Delta t \lambda_1}{2} \big(\frac{\theta^2 }{2} \big) \Big[ \frac{1}{4} (2 + \theta - \theta^2)  + \frac{1}{4} (2 - \theta - \theta^2) \Big]
		- \frac{\theta (1 - \theta^2)}{2}  \\
		&= \frac{\nu \Delta t \lambda_1}{2} \big(\frac{\theta^2 }{2} \big)
		\big( \frac{1}{2} \cancel{+ \frac{\theta}{4}} - \frac{\theta^2}{4} 
		+ \frac{1}{2} \cancel{- \frac{\theta}{4}} - \frac{\theta^2}{4} \big) - \frac{\theta (1 - \theta^2)}{2} \\
		&= \frac{\nu \Delta t \lambda_1 \theta^2}{4} \big(1 - \frac{\theta^2}{2} \big)
		- \frac{\theta (1 - \theta^2)}{2}  \\
		&= \frac{\nu \Delta t \lambda_1 \theta^2 (2 - \theta^2)}{8}
		- \frac{\theta (1 - \theta^2)}{2}. 
	\end{align*}
	\normalcolor
\end{confidential}
		\begin{align*}
			b (a+c) = \frac{\nu \Delta t \lambda_1 \beta_1 (\beta_2 + \beta_0)}{2} - a_1^2
			= \frac{\nu \Delta t \lambda_1 \theta^2 (2 - \theta^2)}{8}
		- \frac{\theta (1 - \theta^2)}{2} :=E.
		\end{align*}
		Hence $(a + c)$ and $b$ are two roots of the following quadratic equation
		\begin{align*}
			X^2 - \sqrt{x} X + E= 0. 
		\end{align*}
		Hence we need $\mu \in (0,1)$ and $\Delta t > 0$ such that the discriminant of the above quadratic equation $\mathcal{D}$ satisfies $\mathcal{D} = x - 4E \geq 0$. 
\begin{confidential}
	\color{darkblue}
	\begin{align*}
        \mathcal{D} = x - 4 E
		&= \frac{(1 - \mu) \nu \Delta t \lambda_1}{2}
        - 4 \Big( \frac{\nu \Delta t \lambda_1 \theta^2 (2 - \theta^2)}{8}
        - \frac{\theta (1 - \theta^2)}{2} \Big) \\
        &= \frac{(1 - \mu) \nu \Delta t \lambda_1}{2}
        - \frac{\nu \Delta t \lambda_1 \theta^2 (2 - \theta^2)}{2}
        + 2 \theta (1 - \theta^2) \geq 0.
    \end{align*}
	\normalcolor
\end{confidential}
		Then we consider the case  
		\begin{align}
			b = \frac{1}{2} \Big( \sqrt{x} - \sqrt{\mathcal{D}} \Big), \quad 
			a + c = \frac{1}{2} \Big( \sqrt{x} + \sqrt{\mathcal{D}} \Big),
			\label{eq:X-roots}
		\end{align}
		Then we combine \eqref{eq:X-roots} and 
		the fifth equation of \eqref{eq:H-stab-system}, and obtain that $a$, $c$ are two roots of the following quadratic equation about $Y$ with the discriminant $\mathscr{D}$ 
		\begin{align}
        &Y^2 - \frac{1}{2} \big( \sqrt{x} + \sqrt{\mathcal{D}}) Y 
        + \Big( \frac{\nu \Delta t \lambda_1 \beta_2 \beta_0}{2} + \frac{a_1^2}{4} \Big) = 0, \notag \\  
        &\mathscr{D} = \frac{1}{4} \big( \sqrt{x} + \sqrt{\mathcal{D}})^2 
        - 4 \Big( \frac{\nu \Delta t \lambda_1 \beta_2 \beta_0}{2} + \frac{a_1^2}{4} \Big).  \label{eq:Y-quad-D}
    \end{align}
	We consider the following case 
    \begin{align*}
        c = \frac{1}{2} \Big[ \frac{1}{2} \big( \sqrt{x} + \sqrt{\mathcal{D}} \big) + \sqrt{\mathscr{D}} \Big], \quad
        a = \frac{1}{2} \Big[ \frac{1}{2} \big( \sqrt{x} + \sqrt{\mathcal{D}} \big) - \sqrt{\mathscr{D}} \Big], \quad
        b = \frac{1}{2} \big( \sqrt{x} - \sqrt{\mathcal{D}} \big).
    \end{align*}
\begin{confidential}
	\color{darkblue}
	\begin{align*}
        ac = \frac{\nu \Delta t \lambda_1 \beta_2 \beta_0}{2} + \frac{a_1^2}{4} 
        &= \frac{\nu \Delta t \lambda_1}{2} \Big[ \frac{1}{4} (2 + \theta - \theta^2) \Big] \Big[ \frac{1}{4} (2 - \theta - \theta^2) \Big]
        + \frac{1}{4} \frac{\theta (1 - \theta^2)}{2} \\
        &= \frac{\nu \Delta t \lambda_1}{32} (2 + \theta - \theta^2)
        (2 - \theta - \theta^2) + \frac{\theta (1 - \theta^2)}{8}.
    \end{align*}
	\normalcolor
\end{confidential}
		By algebraic calculation, 
		\begin{align}
			\sqrt{ \mathscr{D}}
			=& \sqrt{\frac{1}{2} x + \frac{1}{2} \sqrt{x} \sqrt{D} 
        	- \Big[ \underbrace{\frac{\nu \Delta t \lambda_1 \beta_1 (\beta_2 + \beta_0 )}{2} + 2 \nu \Delta t \lambda_1 \beta_2 \beta_0}_{= F} \Big] }. \label{eq:sqrt-D-scr}
		\end{align}
\begin{confidential}
	\color{darkblue}
	\begin{align*}
        \sqrt{ \mathscr{D}} 
        =& \sqrt{ \frac{1}{4} \big( \sqrt{x} + \sqrt{\mathcal{D}})^2 
        - 4 \Big( \frac{\nu \Delta t \lambda_1 \beta_2 \beta_0}{2} + \frac{a_1^2}{4} \Big) } \notag \\
        =& \sqrt{ \frac{1}{4} \big( x + \mathcal{D} + 2 \sqrt{x} \sqrt{D} \big) - 4 \Big( \frac{\nu \Delta t \lambda_1 \beta_2 \beta_0}{2} + \frac{a_1^2}{4} \Big) } \notag \\
        =& \sqrt{ \frac{1}{4} x + \frac{1}{4} (x - 4E)  + \frac{1}{2} \sqrt{x} \sqrt{D} - 4 \Big( \frac{\nu \Delta t \lambda_1 \beta_2 \beta_0}{2} + \frac{a_1^2}{4} \Big) } \notag \\
        =& \sqrt{ \frac{1}{2} x - E  + \frac{1}{2} \sqrt{x} \sqrt{D} - 4 \Big( \frac{\nu \Delta t \lambda_1 \beta_2 \beta_0}{2} + \frac{a_1^2}{4} \Big) } \notag \\
        =& \sqrt{\frac{1}{2} x + \frac{1}{2} \sqrt{x} \sqrt{D} 
        - \Big( \frac{\nu \Delta t \lambda_1 \beta_1 (\beta_2 + \beta_0 )}{2} - a_1^2 \Big) - 4 \Big( \frac{\nu \Delta t \lambda_1 \beta_2 \beta_0}{2} + \frac{a_1^2}{4} \Big) } \notag \\
        =& \sqrt{\frac{1}{2} x + \frac{1}{2} \sqrt{x} \sqrt{D} 
        - \frac{\nu \Delta t \lambda_1 \beta_1 (\beta_2 + \beta_0 )}{2} \cancel{+ a_1^2} - 2 \nu \Delta t \lambda_1 \beta_2 \beta_0 \cancel{- a_1^2} } \notag \\
        =& \sqrt{\frac{1}{2} x + \frac{1}{2} \sqrt{x} \sqrt{D} 
        - \Big[ \underbrace{\frac{\nu \Delta t \lambda_1 \beta_1 (\beta_2 + \beta_0 )}{2} + 2 \nu \Delta t \lambda_1 \beta_2 \beta_0}_{= F} \Big] }. 
    \end{align*}
	\normalcolor
\end{confidential}
		We use the fourth equation of \eqref{eq:H-stab-system} again and achieve
		\begin{align*}
			\sqrt{\frac{1}{2} x + \frac{1}{2} \sqrt{x} \sqrt{D} - F} \big( \sqrt{x} - \sqrt{\mathcal{D}} \big)
        	= \nu \Delta t \lambda_1 \beta_1 ( \beta_0 - \beta_2)
		\end{align*}
		We square the above equation and obtain a linear function of $x$ (terms of $x^2$ are cancelled)
		\begin{align*}
			&\bigg\{ 4 FE - \Big[\nu \Delta t \lambda_1 \beta_1 ( \beta_2 - \beta_0) \Big]^2 \bigg\}^2
        	= 4 (E - F) x \bigg\{ \Big[\nu \Delta t \lambda_1 \beta_1 ( \beta_2 - \beta_0) \Big]^2 - 4E^2 \bigg\}.
		\end{align*}
\begin{confidential}
	\color{darkblue}
	\begin{align*}
		&2 ab = \nu \Delta t \lambda_1 \beta_2 \beta_1 - a_1^2, \\
        &2 \frac{1}{2} \Big[ \frac{1}{2} \big( \sqrt{x} + \sqrt{\mathcal{D}} \big) - \sqrt{\mathscr{D}} \Big] 
        \frac{1}{2} \big( \sqrt{x} - \sqrt{\mathcal{D}} \big)
        = \nu \Delta t \lambda_1 \beta_2 \beta_1 - a_1^2 \\
        & \Big[ \frac{1}{2} \big( \sqrt{x} + \sqrt{\mathcal{D}} \big) - \sqrt{\mathscr{D}} \Big] \big( \sqrt{x} - \sqrt{\mathcal{D}} \big)
        = 2 \Big( \nu \Delta t \lambda_1 \beta_2 \beta_1 - a_1^2 \Big) \\
        & \frac{1}{2} \big( \sqrt{x} + \sqrt{\mathcal{D}} \big) 
        \big( \sqrt{x} - \sqrt{\mathcal{D}} \big) 
        - \sqrt{\mathscr{D}} \big( \sqrt{x} - \sqrt{\mathcal{D}} \big)
        = 2 \nu \Delta t \lambda_1 \beta_2 \beta_1 - 2 a_1^2 \\
        &\frac{1}{2} \big( x - \mathcal{D} \big) - \sqrt{\mathscr{D}} \big( \sqrt{x} - \sqrt{\mathcal{D}} \big) 
        = 2 \nu \Delta t \lambda_1 \beta_2 \beta_1 - 2 a_1^2  \\
        &\frac{1}{2} 4 E
        - \sqrt{\mathscr{D}} \big( \sqrt{x} - \sqrt{\mathcal{D}} \big)
        = 2 \nu \Delta t \lambda_1 \beta_2 \beta_1 - 2 a_1^2  \\
        &2 \Big( \frac{\nu \Delta t \lambda_1 \beta_1 (\beta_2 + \beta_0 )}{2} - a_1^2 \Big) - \sqrt{\mathscr{D}} \big( \sqrt{x} - \sqrt{\mathcal{D}} \big)
        = 2 \nu \Delta t \lambda_1 \beta_2 \beta_1 - 2 a_1^2 \\
        &\nu \Delta t \lambda_1 \beta_1 \beta_2 + 
        \nu \Delta t \lambda_1 \beta_1 \beta_0 \cancel{- 2 a_1^2} - \sqrt{\mathscr{D}} \big( \sqrt{x} - \sqrt{\mathcal{D}} \big)
        = 2 \nu \Delta t \lambda_1 \beta_2 \beta_1 \cancel{- 2 a_1^2}  \\ 
        & \sqrt{\mathscr{D}} \big( \sqrt{x} - \sqrt{\mathcal{D}} \big) 
        = \nu \Delta t \lambda_1 \beta_1 \beta_2 + 
        \nu \Delta t \lambda_1 \beta_1 \beta_0 - 2 \nu \Delta t \lambda_1 \beta_2 \beta_1 
        = \nu \Delta t \lambda_1 \beta_1 \beta_0 - \nu \Delta t \lambda_1 \beta_2 \beta_1 \\
        & \sqrt{\mathscr{D}} \big( \sqrt{x} - \sqrt{\mathcal{D}} \big)
        = \nu \Delta t \lambda_1 \beta_1 ( \beta_0 - \beta_2)
    \end{align*}
	\begin{align*}
        &\sqrt{\frac{1}{2} x + \frac{1}{2} \sqrt{x} \sqrt{\mathcal{D}} - F } \big( \sqrt{x} - \sqrt{\mathcal{D}} \big) = \nu \Delta t \lambda_1 \beta_1 ( \beta_0 - \beta_2) \\
        &\Big( \frac{1}{2} x + \frac{1}{2} \sqrt{x} \sqrt{\mathcal{D}} - F \Big)
        \big( x + \mathcal{D} - 2\sqrt{x} \sqrt{\mathcal{D}}\big) 
        = \Big[ \nu \Delta t \lambda_1 \beta_1 ( \beta_2 - \beta_0) \Big]^2 \\
        &\Big( \frac{1}{2} x + \frac{1}{2} \sqrt{x} \sqrt{\mathcal{D}} - F \Big)
        \big( \mathcal{D} + 4E + \mathcal{D} - 2\sqrt{x} \sqrt{\mathcal{D}} \big) 
        = \Big[\nu \Delta t \lambda_1 \beta_1 ( \beta_2 - \beta_0) \Big]^2 \\
        &\Big( x + \sqrt{x} \sqrt{\mathcal{D}} - 2 F \Big)
        \big( \mathcal{D} + 2E - \sqrt{x} \sqrt{\mathcal{D}} \big) 
        = \Big[\nu \Delta t \lambda_1 \beta_1 ( \beta_2 - \beta_0) \Big]^2 \\
    \end{align*}
    The LHS is 
    \begin{align*}
        &\Big( x + \sqrt{x} \sqrt{\mathcal{D}} - 2 F \Big)
        \big( \mathcal{D} + 2E - \sqrt{x} \sqrt{\mathcal{D}} \big) \\
        =& \cancel{x \mathcal{D}} + 2 xE - x \sqrt{x} \sqrt{\mathcal{D}} 
        + \sqrt{x} \mathcal{D} \sqrt{\mathcal{D}} + 2E \sqrt{x} \sqrt{\mathcal{D}} \cancel{- x \mathcal{D}} - 2F \mathcal{D} - 4 FE 
        + 2F \sqrt{x} \sqrt{\mathcal{D}} \\
        =& 2 xE - (x - \mathcal{D}) \sqrt{x} \sqrt{\mathcal{D}} 
        + 2E \sqrt{x} \sqrt{\mathcal{D}} - 2F \mathcal{D} + 2F \sqrt{x} \sqrt{\mathcal{D}} - 4 FE  \\
        =&2 xE - (4 E) \sqrt{x} \sqrt{\mathcal{D}} + 2E \sqrt{x} \sqrt{\mathcal{D}} - 2F \mathcal{D} + 2F \sqrt{x} \sqrt{\mathcal{D}} - 4 FE  \\
        =& 2 xE - 2 F \mathcal{D} - 2E \sqrt{x} \sqrt{\mathcal{D}}
        + 2F \sqrt{x} \sqrt{\mathcal{D}} - 4 FE \\
        =& 2 xE - 2 F(x - 4E) - 2 (E - F) \sqrt{x} \sqrt{\mathcal{D}} - 4FE \\
        =& 2 xE - 2 xF + 8 FE - 2 (E - F) \sqrt{x} \sqrt{\mathcal{D}} - 4FE \\ 
        =& 2 (E - F) x - 2 (E - F) \sqrt{x} \sqrt{\mathcal{D}} + 4 FE.
    \end{align*}
	\begin{align*}
        &2 (E - F) x - 2 (E - F) \sqrt{x} \sqrt{\mathcal{D}} + 4 FE
        = \Big[\nu \Delta t \lambda_1 \beta_1 ( \beta_2 - \beta_0) \Big]^2 \\
        &2 (E - F) x + 4 FE - \Big[\nu \Delta t \lambda_1 \beta_1 ( \beta_2 - \beta_0) \Big]^2 
        = 2 (E - F) \sqrt{x} \sqrt{\mathcal{D}} \\
        &\bigg\{ 2 (E - F) x + 4 FE - \Big[\nu \Delta t \lambda_1 \beta_1 ( \beta_2 - \beta_0) \Big]^2 \bigg\}^2
        = 4 (E - F)^2 x \mathcal{D} \\
        & \cancel{4 (E - F)^2 x^2} + 4 (E - F) x \bigg\{ 4 FE - \Big[\nu \Delta t \lambda_1 \beta_1 ( \beta_2 - \beta_0) \Big]^2 \bigg\} \\
        &+ \bigg\{ 4 FE - \Big[\nu \Delta t \lambda_1 \beta_1 ( \beta_2 - \beta_0) \Big]^2 \bigg\}^2 
        = 4 (E - F)^2 x \big( x - 4E \big)
        = \cancel{4 (E - F)^2 x^2} - 16 E (E - F)^2 x \\
        &\bigg\{ 4 FE - \Big[\nu \Delta t \lambda_1 \beta_1 ( \beta_2 - \beta_0) \Big]^2 \bigg\}^2
        = 4 (E - F) x \bigg\{ \Big[\nu \Delta t \lambda_1 \beta_1 ( \beta_2 - \beta_0) \Big]^2 - 4FE \bigg\}
        - 16 E (E - F)^2 x \\
        &\bigg\{ 4 FE - \Big[\nu \Delta t \lambda_1 \beta_1 ( \beta_2 - \beta_0) \Big]^2 \bigg\}^2
        = 4 (E - F) x \bigg\{ \Big[\nu \Delta t \lambda_1 \beta_1 ( \beta_2 - \beta_0) \Big]^2 - 4FE - 4E (E - F) \bigg\} \\
        &\bigg\{ 4 FE - \Big[\nu \Delta t \lambda_1 \beta_1 ( \beta_2 - \beta_0) \Big]^2 \bigg\}^2
        = 4 (E - F) x \bigg\{ \Big[\nu \Delta t \lambda_1 \beta_1 ( \beta_2 - \beta_0) \Big]^2 \cancel{- 4FE} - 4E^2 \cancel{+ 4 EF} \bigg\} \\
        &\bigg\{ 4 FE - \Big[\nu \Delta t \lambda_1 \beta_1 ( \beta_2 - \beta_0) \Big]^2 \bigg\}^2
        = 4 (E - F) x \Big[ \nu \Delta t \lambda_1 \beta_1 ( \beta_2 - \beta_0) - 2 E \Big]
        \Big[ \nu \Delta t \lambda_1 \beta_1 ( \beta_2 - \beta_0) + 2 E \Big] 
    \end{align*}
	\begin{align*}
        &E - F \\
        =&\frac{\nu \Delta t \lambda_1 \beta_1 ( \beta_2 + \beta_0)}{2}
        - a_1^2 - \Big[ \frac{\nu \Delta t \lambda_1 \beta_1 (\beta_2 + \beta_0 )}{2} + 2 \nu \Delta t \lambda_1 \beta_2 \beta_0 \Big] \\
        =& \cancel{\frac{\nu \Delta t \lambda_1 \beta_1 ( \beta_2 + \beta_0)}{2}}
        - a_1^2 \cancel{- \frac{\nu \Delta t \lambda_1 \beta_1 (\beta_2 + \beta_0 )}{2}} - 2 \nu \Delta t \lambda_1 \beta_2 \beta_0 \\
        =& - a_1^2 - 2 \nu \Delta t \lambda_1 \beta_2 \beta_0, \\
		\\
        &\nu \Delta t \lambda_1 \beta_1 ( \beta_2 - \beta_0)- 2 E \\
        =& \nu \Delta t \lambda_1 \beta_1 ( \beta_2 - \beta_0)
        - 2 \Big[ \frac{\nu \Delta t \lambda_1 \beta_1 ( \beta_2 + \beta_0)}{2} - a_1^2  \Big] \\
        =& \nu \Delta t \lambda_1 \beta_1 ( \beta_2 - \beta_0)
        - \nu \Delta t \lambda_1 \beta_1 ( \beta_2 + \beta_0) + 2 a_1^2 \\
        =& \nu \Delta t \lambda_1 \beta_1 ( \cancel{\beta_2} - \beta_0 \cancel{- \beta_2} - \beta_0) + 2 a_1^2 
        =- 2 \nu \Delta t \lambda_1 \beta_1 \beta_0 + 2 a_1^2 \\
		\\ 
        &\nu \Delta t \lambda_1 \beta_1 ( \beta_2 - \beta_0) + 2 E\\
        =&\nu \Delta t \lambda_1 \beta_1 ( \beta_2 - \beta_0) +
        2 \Big[ \frac{\nu \Delta t \lambda_1 \beta_1 ( \beta_2 + \beta_0)}{2} - a_1^2 \Big] \\
        =&\nu \Delta t \lambda_1 \beta_1 ( \beta_2 - \beta_0)
        + \nu \Delta t \lambda_1 \beta_1 ( \beta_2 + \beta_0) 
        - 2 a_1^2 \\
        =& \nu \Delta t \lambda_1 \beta_1 ( \beta_2 \cancel{- \beta_0} + \beta_2 \cancel{+ \beta_0}) - 2 a_1^2 
        = 2 \nu \Delta t \lambda_1 \beta_1 \beta_2 - 2 a_1^2.
    \end{align*}
	\begin{align*}
        &4 (E - F) \bigg\{ \Big[\frac{ \nu \Delta t \beta_1 ( \beta_2 - \beta_0)}{C_{p}^{2}} \Big]^2 - 4E^2 \bigg\} \\
        =& - 4 \Big( a_1^2 + 2 \nu \Delta t \lambda_1 \beta_2 \beta_0 \Big) \Big( 2 a_1^2 - 2 \nu \Delta t \lambda_1 \beta_1 \beta_0 \Big) \Big( 2 \nu \Delta t \lambda_1 \beta_1 \beta_2 - 2 a_1^2 \Big) \\
        =& 16 \Big( \underbrace{a_1^2 + 2 \nu \Delta t \lambda_1 \beta_2 \beta_0}_{ > 0} \Big) \Big( a_1^2 - \nu \Delta t \lambda_1 \beta_1 \beta_0 \Big) \Big( a_1^2 - \nu \Delta t \lambda_1 \beta_1 \beta_2 \Big)
    \end{align*}
	\normalcolor
\end{confidential}
		The requirement $x \geq 0$ yields 
		\begin{align*}
			&4 (E - F) \bigg\{ \Big[\nu \Delta t \lambda_1 \beta_1 ( \beta_2 - \beta_0) \Big]^2 - 4E^2 \bigg\} \\
			=& 16 \Big( \underbrace{a_1^2 + 2 \nu \Delta t \lambda_1 \beta_2 \beta_0}_{ > 0} \Big) \Big( a_1^2 - \nu \Delta t \lambda_1 \beta_1 \beta_0 \Big) \Big( a_1^2 - \nu \Delta t \lambda_1 \beta_1 \beta_2 \Big) > 0. 
		\end{align*}
\begin{confidential}
	\color{darkblue}
	\begin{align*}
        &\Big( a_1^2 - \nu \Delta t \lambda_1 \beta_1 \beta_0 \Big) \Big( a_1^2 - \nu \Delta t \lambda_1 \beta_1 \beta_2 \Big)
        >0 
    \end{align*}
	We consider the case $\theta \in (0,1)$, which implies $\beta_2 > \beta_0$ and 
    \begin{align*}
        0 < \nu \Delta t \lambda_1 \beta_1 \beta_0
        < \nu \Delta t \lambda_1 \beta_1 \beta_2.
    \end{align*}
	Thus 
	\begin{align*}
		&a_1^2 < \nu \Delta t \lambda_1 \beta_1 \beta_0, \qquad \text{or } \qquad
		a_1^2 > \nu \Delta t \lambda_1 \beta_1 \beta_2 
	\end{align*}
	Since we can set $\Delta t$ to be small enough for accuracy, we only works for the second inequality, 
    \begin{align*}
        &\frac{\theta (1 - \theta^2)}{2} > \nu \Delta t \lambda_1 \beta_1 \beta_2 = \nu \Delta t \lambda_1
        \frac{1}{2}\big( \theta^2 \big) \frac{1}{4}\big( 2 + \theta - \theta^2 \big) \\
        &\frac{\theta (1 - \theta)(1 + \theta)}{2} 
        > \frac{ \nu \Delta t \lambda_1}{8} \theta^2 (2 - \theta) (1 + \theta) \\
        & (1 - \theta) > \frac{ \nu \Delta t \lambda_1}{4} \theta (2 - \theta), \quad \Leftrightarrow \quad
        \Delta t < \frac{4 (1 - \theta)}{ \nu \lambda_1 \theta (2 - \theta)}.
    \end{align*}
	\normalcolor
\end{confidential}
		We need the time step restriction 
		\begin{align}
			\label{eq:dt-cond-1}
			\Delta t < \frac{4 (1 - \theta)}{ \nu \lambda_1 \theta (2 - \theta)},
		\end{align}
		and solve for $x$
		\begin{align}
			x = \frac{\bigg\{ 4 FE - \Big[\nu \Delta t \lambda_1 \beta_1 ( \beta_2 - \beta_0) \Big]^2 \bigg\}^2}{4 (E - F) \bigg\{\Big[ \nu \Delta t \lambda_1 \beta_1 ( \beta_2 - \beta_0) \Big]^2 - 4E^2 \bigg\}} \geq 0. 
			\label{eq:x-sol}
		\end{align} 
		From \eqref{eq:x-sol}, $\mathcal{D} = x - 4E >0$ as long as $E<0$, which is ensured by the time step restriction in \eqref{eq:dt-cond-1}. 
\begin{confidential}
	\color{darkblue}
	We check that the time step restriction in \eqref{eq:dt-cond-1} ensures $E < 0$. 
    \begin{align*}
        &E = \frac{\nu \Delta t \lambda_1 \beta_1 ( \beta_2 + \beta_0)}{2}
        - a_1^2 < 0, \\
        &\frac{\nu \Delta t \lambda_1}{2} 
        \big( \frac{1}{2} \theta^2 \big) \Big[ \frac{1}{4} \big( 2 + \theta - \theta^2 \big) + \frac{1}{4} \big(2 - \theta - \theta^2 \big) \Big]
        < \frac{\theta (1 - \theta^2)}{2}, \\
        &\frac{\nu \Delta t \lambda_1}{16} \theta^2 \big( 2 \cancel{+ \theta} - \theta^2 + 2 \cancel{- \theta} - \theta^2 \big) < \frac{\theta (1 - \theta)(1 + \theta)}{2} \\
        &\frac{\nu \Delta t \lambda_1}{8} \theta (2 - \theta^2) 
        < \frac{(1 - \theta)(1 + \theta)}{2} \\
        &\frac{\nu \Delta t \lambda_1}{4} \theta (2 - \theta^2) 
        < (1 - \theta)(1 + \theta) \\
        & \Delta t < \frac{4 (1 - \theta)(1 + \theta) }{\nu \lambda_1 \theta (2 - \theta^2) }.
    \end{align*}
	We will show: for $\theta \in (0,1)$
    \begin{align*}
        &\frac{4 (1 - \theta)(1 + \theta) }{\nu \lambda_1 \theta (2 - \theta^2) } > \frac{4 (1 - \theta)}{ \nu \lambda_1 \theta (2 - \theta) } \ \Leftrightarrow \ 
        \frac{1 + \theta}{2 - \theta^2} > \frac{1}{2 - \theta} 
        \ \Leftrightarrow \
        (1 + \theta) (2 - \theta) > 2 - \theta^2 \\
        &\bcancel{2} - \theta + 2 \theta \cancel{- \theta^2} > \bcancel{2} \cancel{- \theta^2}, \ \Leftrightarrow \ 
        \theta > 0. 
    \end{align*}
	Hence the second time step restriction is absorbed by the restriction in \eqref{eq:dt-cond-1}.
	By \eqref{eq:sqrt-D-scr}
    \begin{align*}
        \mathscr{D} = \frac{1}{2} x + \frac{1}{2} \sqrt{x} \sqrt{\mathcal{D}} - \Big[ \underbrace{\frac{\nu \Delta t \beta_1 (\beta_2 + \beta_0 )}{2 C_{p}^{2}} + \frac{2 \nu \Delta t \beta_2 \beta_0}{ C_{p}^{2}} }_{= F} \Big] = \frac{1}{2} x + \frac{1}{2} \sqrt{x} \sqrt{\mathcal{D}} - F > 0.
    \end{align*}
	\normalcolor
\end{confidential}
		By the facts that $E < 0$ and $F >0$, we can show that $\mathscr{D} >0$ in \eqref{eq:sqrt-D-scr}. 
\begin{confidential}
	\color{darkblue}
	By \eqref{eq:sqrt-D-scr}, to show 
    \begin{align*}
        \mathscr{D} = \frac{1}{2} x + \frac{1}{2} \sqrt{x} \sqrt{\mathcal{D}} - F > 0, 
    \end{align*}
	it suffices to show 
    \begin{align*}
        &\frac{1}{2} \sqrt{x} \sqrt{\mathcal{D}} > F - \frac{1}{2} x 
        \ \Leftrightarrow \ \sqrt{x} \sqrt{\mathcal{D}} > 2 F - x 
        \ \Leftrightarrow \ x \mathcal{D} > (2 F - x)^2 \\
        &x (x - 4E) > 4F^2 - 4Fx + x^2  \ \Leftrightarrow \ 
        \cancel{x^2} - 4xE > 4 F^2 - 4xF \cancel{+ x^2} \ \Leftrightarrow \ 
        \cancel{4}x (F - E) > \cancel{4} F^2. 
    \end{align*}
	By \eqref{eq:x-sol}, we have 
    \begin{align*}
        &\cancel{(F - E)} \frac{\bigg\{ 4 FE - \Big[ \nu \Delta t \lambda_1 \beta_1 ( \beta_2 - \beta_0) \Big]^2 \bigg\}^2}{4 \cancel{(F-E)} \bigg\{ \underbrace{4E^2 - \Big[\nu \Delta t \lambda_1 \beta_1 ( \beta_2 - \beta_0) \Big]^2}_{>0\ \text{by previous calculation and time step limits}} \bigg\}} > F^2, \\
        &\bigg\{ 4 FE - \Big[ \nu \Delta t \lambda_1 \beta_1 ( \beta_2 - \beta_0) \Big]^2 \bigg\}^2 
        > 4 F^2 \bigg\{ 4E^2 - \Big[\nu \Delta t \lambda_1 \beta_1 ( \beta_2 - \beta_0) \Big]^2 \bigg\} \\
        &\cancel{16E^2F^2} - 8 EF \Big[ \nu \Delta t \lambda_1 \beta_1 ( \beta_2 - \beta_0) \Big]^2 + \Big[\nu \Delta t \lambda_1 \beta_1 ( \beta_2 - \beta_0) \Big]^4 
        > \cancel{16 E^2 F^2} - 4 F^2 \Big[\nu \Delta t \lambda_1 \beta_1 ( \beta_2 - \beta_0) \Big]^2 \\
        & 4 F^2 \Big[\nu \Delta t \lambda_1 \beta_1 ( \beta_2 - \beta_0) \Big]^2 - 8 E F \Big[\nu \Delta t \lambda_1 \beta_1 ( \beta_2 - \beta_0) \Big]^2 + \Big[ \nu \Delta t \lambda_1 \beta_1 ( \beta_2 - \beta_0) \Big]^4 > 0,
    \end{align*}
	and 
	\begin{align*}
        F &= \frac{\nu \Delta t \lambda_1 \beta_1 (\beta_2 + \beta_0 )}{2} + 2 \nu \Delta t \lambda_1 \beta_2 \beta_0 \\
        &= \frac{\nu \Delta t \lambda_1}{2} \big( \frac{1}{2}\theta^2 \big) \Big[ \frac{1}{4} \big(2 + \theta - \theta^2 \big) + \frac{1}{4} \big( 2 - \theta - \theta^2 \big) \Big]
        + 2 \nu \Delta t \lambda_1 \cdot \frac{1}{4} \big(2 + \theta - \theta^2 \big) \cdot \frac{1}{4} \big( 2 - \theta - \theta^2 \big) \\
        &= \frac{\nu \Delta t \lambda_1 \theta^2 }{16} \big(2 \cancel{+ \theta} - \theta^2 + 2 \cancel{- \theta} - \theta^2 \big)
        + \frac{\nu \Delta t \lambda_1}{8} (2 - \theta)(1 + \theta) (2 + \theta)(1 - \theta) \\
        &= \frac{\nu \Delta t \lambda_1 \theta^2 (2 - \theta^2)}{8} 
        + \frac{\nu \Delta t \lambda_1}{8} (2 - \theta)(1 + \theta) (2 + \theta)(1 - \theta) > 0.
    \end{align*}
	\normalcolor
\end{confidential}
		We have solved $a$, $b$, $c$ in the numerical dissipation term in \eqref{eq:H-stab}. Now we employ the following new time step restriction  
		\begin{align}
			\Delta t < \frac{2 (1 - \theta)}{\nu \lambda_1 },  
			\label{eq:dt-cond-2} 
		\end{align}
		which is more restrictive than the limit in \eqref{eq:dt-cond-1}. 
\begin{confidential}
	\color{darkblue}
	\begin{align*}
        &\frac{2 \bcancel{(1 - \theta)}}{\bcancel{\nu} \cancel{\lambda_1}} < \frac{4 \bcancel{(1 - \theta)}}{ \bcancel{\nu} \cancel{\lambda_1} \theta (2 - \theta) } \ \Leftrightarrow \         
        1 < \frac{2}{\theta (2 - \theta)} \ \Leftrightarrow \    
        \theta (2 - \theta) < 2 \\
        &2\theta - \theta^2 < 2 \ \Leftrightarrow \        
        \theta^2 - 2 \theta + 2 > 0 \ \Leftrightarrow \   
        (\theta - 1)^2 + 1 > 0. 
    \end{align*}
	\normalcolor
\end{confidential}
		By the third equation of \eqref{eq:H-stab-system} and the updated time step restriction in \eqref{eq:dt-cond-2}
		\begin{align}
			h_{22} 
			=& c^2 + \frac{(1 - \theta)(2 - \theta - \theta^2)}{8}
			- \frac{\nu \Delta t \lambda_1 \beta_0^2}{2} \notag \\
        	>& c^2 + \frac{(1 - \theta)(2 - \theta - \theta^2)}{8}
        	- (1- \theta) \Big[ \frac{1}{4} \big( 2 - \theta - \theta^2 \big) \Big]^2 
			\label{eq:h22-sol} \\
			=& c^2 + \frac{ \theta (1 - \theta)^2 (1 + \theta)(2 + \theta)}{16} 
			\geq \frac{ \theta (1 - \theta)^2 (1 + \theta)(2 + \theta)}{16}
			>0. \notag 
		\end{align}
\begin{confidential}
	\color{darkblue}
	\begin{align*}
        h_{22} 
        =& c^2 + \frac{(1 - \theta)(2 - \theta - \theta^2)}{8} - \frac{\nu \Delta t \lambda_1 \beta_0^2}{2} \notag \\
        =& c^2 + \frac{(1 - \theta)(2 - \theta - \theta^2)}{8} - \frac{\nu \Delta t \lambda_1}{2} \Big[ \frac{1}{4} \big( 2 - \theta - \theta^2 \big) \Big]^2 \notag \\
        >& c^2 + \frac{(1 - \theta)(2 - \theta - \theta^2)}{8}
        - \frac{1}{\cancel{2}} \cdot \cancel{2} (1- \theta) \Big[ \frac{1}{4} \big( 2 - \theta - \theta^2 \big) \Big]^2 \notag \\
        =& c^2 + \frac{(1 - \theta)(2 - \theta - \theta^2)}{8}
        - \frac{1}{16} (1- \theta) \big( 2 - \theta - \theta^2 \big)^2 \notag \\
        =& c^2 + \frac{(1 - \theta)(2 - \theta - \theta^2)}{8} 
        \Big[ 1 - \frac{1}{2} \big( 2 - \theta - \theta^2 \big) \Big] \notag \\
        =& c^2 + \frac{(1 - \theta)(2 - \theta - \theta^2)}{8} 
        \big( \cancel{1} \cancel{- 1} + \frac{1}{2} \theta + \frac{1}{2} \theta^2 \big) \notag \\
        =& c^2 + \frac{(1 - \theta)(2 + \theta) (1 - \theta)}{8} \frac{1}{2} \theta (1 + \theta) \notag \\
        =& c^2 + \frac{ \theta (1 - \theta)^2 (1 + \theta)(2 + \theta)}{16}
        \geq \frac{ \theta (1 - \theta)^2 (1 + \theta)(2 + \theta)}{16}.
    \end{align*}
	\normalcolor
\end{confidential}
		To solve for $h_{11}$, we multiply first equation of \eqref{eq:H-stab-system} by $h_{22}$ and second equation of \eqref{eq:H-stab-system} by $h_{11}$
		\begin{align}
			&(1 + \epsilon) h_{11} h_{22} + a^2 h_{22}
			= \frac{(1+\theta) (2 + \theta - \theta^2)}{8} h_{22} + \frac{\nu \Delta t \lambda_1 \beta_2^2}{2} h_{22}, \label{eq:h11-eq1} \\
			&(1 + \epsilon) h_{22}h_{11} - h_{11}^2 + b^2 h_{11}
			= \frac{\nu \Delta t \lambda_1 \beta_1^2}{2}h_{11} - \frac{\theta^3}{2} h_{11}. \label{eq:h11-eq2}
		\end{align}
		We subtract \eqref{eq:h11-eq2} from \eqref{eq:h11-eq1} and obtain the following quadratic equation of $h_{11}$
		\begin{align}
			&h_{11}^2 + \Big( \frac{\nu \Delta t \lambda_1 \beta_1^2}{2} - \frac{\theta^3}{2} - b^2 \Big) h_{11} + \Big[ a^2 - \frac{(1+\theta) (2 + \theta - \theta^2)}{8} - \frac{\nu \Delta t \lambda_1 \beta_2^2}{2} \Big]h_{22} = 0.  \label{eq:quad-h11}
		\end{align}
\begin{confidential}
	\color{darkblue}
	\begin{align*}
			&a^2 h_{22} + h_{11}^2 - b^2 h_{11} = \frac{(1+\theta) (2 + \theta - \theta^2)}{8} h_{22} + \frac{\nu \Delta t \lambda_1 \beta_2^2}{2} h_{22} - \frac{\nu \Delta t \lambda_1 \beta_1^2}{2}h_{11} + \frac{\theta^3}{2} h_{11}. 
		\end{align*}
	\normalcolor
\end{confidential}
		We assume that the two roots of the quadratic equation \eqref{eq:quad-h11} are $h_{11}^+$ (the larger one) and $h_{11}^-$ 
		and have 
		\begin{align}
			&h_{11}^{+} + h_{11}^{-} = - \Big( \frac{\nu \Delta t \lambda_1 \beta_1^2}{2} - \frac{\theta^3}{2} - b^2 \Big), \label{eq:h11-sum} \\
			&h_{11}^{+} h_{11}^{-} = \Big[ a^2 - \frac{(1+\theta) (2 + \theta - \theta^2)}{8} - \frac{\nu \Delta t \lambda_1 \beta_2^2}{2} \Big]h_{22}. \notag 
		\end{align}
		By \eqref{eq:dt-cond-2} and \eqref{eq:h11-sum}
		\begin{align}
			h_{11}^{+} + h_{11}^{-} = b^2 + \frac{\theta^3}{2} \Big[ 1 - \frac{\nu \Delta t \lambda_1 \theta}{4} \Big] 
			> b^2 + \frac{\theta^3}{2} \big[1 - \frac{1}{2} \theta (1 - \theta) \big] > 0. 
			\label{eq:h11-sum-positive}
		\end{align}
\begin{confidential}
	\color{darkblue}
	\begin{align*}
        &h_{11}^{+} + h_{11}^{-} =
		b^2 + \frac{\theta^3}{2} - \frac{\nu \Delta t \lambda_1 \beta_1^2}{2} 
        =b^2 + \frac{\theta^3}{2} - \frac{\nu \Delta t \lambda_1}{2} \Big(\frac{1}{2} \theta^2 \Big)^2 \\
        &=b^2 + \frac{\theta^3}{2} - \frac{\nu \Delta t \lambda_1 \theta^4}{8} 
        =b^2 + \frac{\theta^3}{2} \Big[ 1 - \frac{\nu \Delta t \lambda_1 \theta}{4} \Big] 
		> b^2 + \frac{\theta^3}{2} \Big[ 1 - \frac{\bcancel{\nu} \cancel{\lambda_1} \theta}{4} 
        \frac{2 (1 - \theta)}{\bcancel{\nu} \cancel{\lambda_1}} \Big] \\
		&= b^2 + \frac{\theta^3}{2} \big[1 - \frac{1}{2} \theta (1 - \theta) \big] > 0.
    \end{align*}
	\normalcolor
\end{confidential}
		Since we have shown that $h_{22} >0$ and 
		\begin{align*}
			&a^2 - \frac{(1+\theta) (2 + \theta - \theta^2)}{8} - \frac{\nu \Delta t \lambda_1 \beta_2^2}{2} \\
			=& \frac{\sqrt{x} + \sqrt{\mathcal{D}}}{4}  
			\frac{\frac{1}{4}\big( \sqrt{x} + \sqrt{\mathcal{D}} \big)^2 - \mathscr{D}}{\underbrace{\frac{1}{2}\big( \sqrt{x} + \sqrt{\mathcal{D}} \big) + \sqrt{\mathscr{D}}}_{> \frac{1}{2}\big( \sqrt{x} + \sqrt{\mathcal{D}} \big)}} 
			- \Big( \frac{\nu \Delta t \lambda_1 \beta_2 \beta_0}{2} + \frac{a_1^2}{4} \Big) 
			- \frac{(1+\theta) \beta_2}{2} - \frac{\nu \Delta t \lambda_1 \beta_2^2}{2} \\
			<& \frac{1}{2} \big( \sqrt{x} + \sqrt{\mathcal{D}} \big)^2 - \mathscr{D} 
			- \Big( \frac{\nu \Delta t \lambda_1 \beta_2 \beta_0}{2} + \frac{a_1^2}{4} \Big) 
			- \frac{(1+\theta) (2 + \theta - \theta^2)}{8} - \frac{\nu \Delta t \lambda_1 \beta_2^2}{2} \\
			=& - \frac{\nu \Delta t \lambda_1 \beta_2 \theta}{4} 
        	- \frac{(1+\theta)}{4} < 0, 
		\end{align*}
		we have shown $h_{11}^{+} h_{11}^{-} <0$.
\begin{confidential}
	\color{darkblue}
	\begin{align*}
        &a^2 - \frac{(1+\theta) (2 + \theta - \theta^2)}{8} - \frac{\nu \Delta t \lambda_1 \beta_2^2}{2} \\
        =& \bigg\{ \frac{1}{2} \Big[ \frac{1}{2} \big( \sqrt{x} + \sqrt{\mathcal{D}} \big) - \sqrt{\mathscr{D}} \Big] \bigg\}^2
        - \frac{(1+\theta) (2 + \theta - \theta^2)}{8} - \frac{\nu \Delta t \lambda_1 \beta_2^2}{2} \\
        =& \frac{1}{4} \Big[ \frac{1}{2} \big( \sqrt{x} + \sqrt{\mathcal{D}} \big) - \sqrt{\mathscr{D}} \Big]^2 - \frac{(1+\theta) (2 + \theta - \theta^2)}{8} - \frac{\nu \Delta t \lambda_1 \beta_2^2}{2} \\
        =& \frac{1}{4} \Big[ \frac{1}{4} \big( \sqrt{x} + \sqrt{\mathcal{D}} \big)^2 - \sqrt{\mathscr{D}} \big( \sqrt{x} + \sqrt{\mathcal{D}} \big) + \mathscr{D} \Big]
        - \frac{(1+\theta) (2 + \theta - \theta^2)}{8} - \frac{\nu \Delta t \lambda_1  \beta_2^2}{2} \\ 
        =& \frac{1}{4} \Big[ \frac{1}{4} \big( \sqrt{x} + \sqrt{\mathcal{D}} \big)^2 - \sqrt{\mathscr{D}} \big( \sqrt{x} + \sqrt{\mathcal{D}} \big) + \frac{1}{4} \big( \sqrt{x} + \sqrt{\mathcal{D}})^2 
        - 4 \Big( \frac{\nu \Delta t \lambda_1 \beta_2 \beta_0}{2} + \frac{a_1^2}{4} \Big) \Big] \\
        &- \frac{(1+\theta) (2 + \theta - \theta^2)}{8} - \frac{\nu \Delta t \lambda_1 \beta_2^2}{2} \\ 
        =& \frac{1}{4} \Big[ \frac{1}{2} \big( \sqrt{x} + \sqrt{\mathcal{D}} \big)^2 - \sqrt{\mathscr{D}} \big( \sqrt{x} + \sqrt{\mathcal{D}} \big) - 4 \Big( \frac{\nu \Delta t \lambda_1 \beta_2 \beta_0}{2} + \frac{a_1^2}{4} \Big) \Big] \\
        &- \frac{(1+\theta) (2 + \theta - \theta^2)}{8} - \frac{\nu \Delta t \lambda_1 \beta_2^2}{2} \\
        =& \frac{1}{8} \big( \sqrt{x} + \sqrt{\mathcal{D}} \big)^2
        - \frac{1}{4} \sqrt{\mathscr{D}} \big( \sqrt{x} + \sqrt{\mathcal{D}} \big) - \Big( \frac{\nu \Delta t \lambda_1 \beta_2 \beta_0}{2} + \frac{a_1^2}{4} \Big) \\
        &- \frac{(1+\theta) (2 + \theta - \theta^2)}{8} - \frac{\nu \Delta t \lambda_1 \beta_2^2}{2} \\
        =& \frac{1}{4} \big( \sqrt{x} + \sqrt{\mathcal{D}} \big) 
        \Big[ \frac{1}{2}\big( \sqrt{x} + \sqrt{\mathcal{D}} \big) - \sqrt{\mathscr{D}} \Big] - \Big( \frac{\nu \Delta t \lambda_1 \beta_2 \beta_0}{2} + \frac{a_1^2}{4} \Big) \\
        &- \frac{(1+\theta) (2 + \theta - \theta^2)}{8} - \frac{\nu \Delta t \lambda_1 \beta_2^2}{2} \\
        =& \frac{1}{4} \big( \sqrt{x} + \sqrt{\mathcal{D}} \big) \frac{\Big[ \frac{1}{2}\big( \sqrt{x} + \sqrt{\mathcal{D}} \big) - \sqrt{\mathscr{D}} \Big] \Big[ \frac{1}{2}\big( \sqrt{x} + \sqrt{\mathcal{D}} \big) + \sqrt{\mathscr{D}} \Big]}{\frac{1}{2}\big( \sqrt{x} + \sqrt{\mathcal{D}} \big) + \sqrt{\mathscr{D}} }
        - \Big( \frac{\nu \Delta t \lambda_1 \beta_2 \beta_0}{2} + \frac{a_1^2}{4} \Big) \\
        &- \frac{(1+\theta) (2 + \theta - \theta^2)}{8} - \frac{\nu \Delta t \lambda_1 \beta_2^2}{2} \\
        =& \frac{1}{4} \big( \sqrt{x} + \sqrt{\mathcal{D}} \big) 
        \frac{\frac{1}{4}\big( \sqrt{x} + \sqrt{\mathcal{D}} \big)^2 - \mathscr{D}}{\underbrace{\frac{1}{2}\big( \sqrt{x} + \sqrt{\mathcal{D}} \big) + \sqrt{\mathscr{D}}}_{> \frac{1}{2}\big( \sqrt{x} + \sqrt{\mathcal{D}} \big)}} 
        - \Big( \frac{\nu \Delta t \lambda_1 \beta_2 \beta_0}{2} + \frac{a_1^2}{4} \Big) \\
        &- \frac{(1+\theta) (2 + \theta - \theta^2)}{8} - \frac{\nu \Delta t \lambda_1 \beta_2^2}{2} \\
        <& \frac{1}{4} \bcancel{\big( \sqrt{x} + \sqrt{\mathcal{D}} \big)}  
        \frac{\frac{1}{4} \big( \sqrt{x} + \sqrt{\mathcal{D}} \big)^2 - \mathscr{D}}{\frac{1}{2} \bcancel{\big( \sqrt{x} + \sqrt{\mathcal{D}} \big)}}
        - \Big( \frac{\nu \Delta t \lambda_1 \beta_2 \beta_0}{2} + \frac{a_1^2}{4} \Big) \\
        &- \frac{(1+\theta) (2 + \theta - \theta^2)}{8} - \frac{\nu \Delta t \lambda_1 \beta_2^2}{2} \\
        =& \frac{1}{2} \Big[ \frac{1}{4} \big( \sqrt{x} + \sqrt{\mathcal{D}} \big)^2 - \mathscr{D} \Big] - \Big( \frac{\nu \Delta t \lambda_1 \beta_2 \beta_0}{2} + \frac{a_1^2}{4} \Big) 
        - \frac{(1+\theta) (2 + \theta - \theta^2)}{8} - \frac{\nu \Delta t \lambda_1 \beta_2^2}{2} \\
        =& \frac{1}{2} \bigg\{ \frac{1}{4} \big( \sqrt{x} + \sqrt{\mathcal{D}} \big)^2 - \Big[ \frac{1}{4} \big( \sqrt{x} + \sqrt{\mathcal{D}})^2 - 4 \Big( \frac{\nu \Delta t \lambda_1 \beta_2 \beta_0}{2} + \frac{a_1^2}{4} \Big)  \Big] \bigg\} \\
        &- \Big( \frac{\nu \Delta t \lambda_1 \beta_2 \beta_0}{2} + \frac{a_1^2}{4} \Big) 
        - \frac{(1+\theta) (2 + \theta - \theta^2)}{8} - \frac{\nu \Delta t \lambda_1 \beta_2^2}{2} \\
        =& \frac{1}{2} \Big[ \cancel{\frac{1}{4} \big( \sqrt{x} + \sqrt{\mathcal{D}} \big)^2} \cancel{- \frac{1}{4} \big( \sqrt{x} + \sqrt{\mathcal{D}})^2}  + 4 \Big( \frac{\nu \Delta t \lambda_1 \beta_2 \beta_0}{2} + \frac{a_1^2}{4} \Big) \Big] \\
        &- \Big( \frac{\nu \Delta t \lambda_1 \beta_2 \beta_0}{2} + \frac{a_1^2}{4} \Big) 
        - \frac{(1+\theta) (2 + \theta - \theta^2)}{8} - \frac{\nu \Delta t \lambda_1 \beta_2^2}{2} \\
        =& 2 \Big( \frac{\nu \Delta t \lambda_1 \beta_2 \beta_0}{2} + \frac{a_1^2}{4} \Big) - \Big( \frac{\nu \Delta t \lambda_1 \beta_2 \beta_0}{2} + \frac{a_1^2}{4} \Big)
        - \frac{(1+\theta) (2 + \theta - \theta^2)}{8} - \frac{\nu \Delta t \lambda_1 \beta_2^2}{2} \\
        =& \Big( \frac{\nu \Delta t \lambda_1 \beta_2 \beta_0}{2} + \frac{a_1^2}{4} \Big) - \frac{(1+\theta) (2 + \theta - \theta^2)}{8} - \frac{\nu \Delta t \lambda_1 \beta_2^2}{2} \\
        =& \frac{\nu \Delta t \lambda_1 \beta_2 \beta_0}{2} - \frac{\nu \Delta t \lambda_1 \beta_2^2}{2} + \frac{a_1^2}{4} 
        - \frac{(1+\theta) (2 + \theta - \theta^2)}{8} \\
        =& \frac{\nu \Delta t \lambda_1 \beta_2}{2} \big( \beta_0 - \beta_2 \big) + \frac{1}{4} \frac{\theta (1 - \theta^2)}{2}
        - \frac{(1+\theta) (2 + \theta - \theta^2)}{8} \\
        =& \frac{\nu \Delta t \lambda_1 \beta_2}{2} \Big[ \frac{1}{4} \big(2 - \theta - \theta^2 \big) - \frac{1}{4} \big(2 + \theta - \theta^2 \big) \Big] + \frac{1}{8} \theta (1 - \theta)(1 + \theta)
        - \frac{(1+\theta) (2 + \theta - \theta^2)}{8} \\
        =& \frac{\nu \Delta t \lambda_1 \beta_2}{2} \frac{1}{4} \big(\cancel{2} - \theta \bcancel{- \theta^2} \cancel{- 2} - \theta \bcancel{+ \theta^2} \big) + \frac{(1+\theta)}{8} \Big[ \theta (1 - \theta) - \big(2 + \theta - \theta^2 \big) \Big] \\
        =& \frac{\nu \Delta t \lambda_1 \beta_2}{2} \frac{1}{4} (-2\theta) 
        + \frac{(1+\theta)}{8} \big(\cancel{\theta} \bcancel{- \theta^2} -2 \cancel{- \theta} \bcancel{+ \theta^2} \big) \\
        =& - \frac{\nu \Delta t \lambda_1 \beta_2 \theta}{4} 
        - \frac{(1+\theta)}{4} < 0. 
    \end{align*}
	\normalcolor
\end{confidential}
		Thus we can solve $h_{11}^+$ via the quadratic equation in \eqref{eq:quad-h11} 
		\begin{align}
			h_{11} = h_{11}^{+} 
			=& \frac{1}{2} \Bigg\{ \underbrace{- \Big( \frac{\nu \Delta t \lambda_1 \beta_1^2}{2} - \frac{\theta^3}{2} - b^2 \Big)}_{>0} \label{eq:h11-sol} \\
        	&\ \ \ +\!\! \sqrt{\Big( \frac{\nu \Delta t \lambda_1 \beta_1^2}{2} \!-\! \frac{\theta^3}{2} \!-\! b^2 \Big)^2 \!+\! \underbrace{4 \Big[ \frac{\nu \Delta t \lambda_1 \beta_2^2}{2} \!+\! \frac{(1+\theta) (2 + \theta - \theta^2)}{8} \!-\! a^2 \Big]h_{22}}_{>0} } \Bigg\} \notag \\
			=& - \Big( \frac{\nu \Delta t \lambda_1 \beta_1^2}{2} - \frac{\theta^3}{2} - b^2 \Big) 
        	\geq \frac{\theta^3}{2} \Big[ 1 - \frac{1}{2} \theta (1 - \theta) \Big] > 0,  \notag
		\end{align}
		where the lower bound is by \eqref{eq:h11-sum-positive}.
        We combine \eqref{eq:h22-sol} and \eqref{eq:h11-sol} to have \eqref{eq:h-lower-bound}. 
\begin{confidential}
	\color{darkblue}
	We solve for $h_{11}^{+}$ by quadratic formula 
    \begin{align*}
        &h_{11} = h_{11}^{+} \\
        =& \frac{1}{2} \bigg\{ - \Big( \frac{\nu \Delta t \lambda_1 \beta_1^2}{2} - \frac{\theta^3}{2} - b^2 \Big) 
        + \sqrt{\Big( \frac{\nu \Delta t \lambda_1 \beta_1^2}{2} - \frac{\theta^3}{2} - b^2 \Big)^2 - 4 \Big[ a^2 - \frac{(1+\theta) (2 + \theta - \theta^2)}{8} - \frac{\nu \Delta t \lambda_1 \beta_2^2}{2} \Big]h_{22} } \bigg\} \\
        =& \frac{1}{2} \bigg\{ \underbrace{- \Big( \frac{\nu \Delta t \lambda_1 \beta_1^2}{2} - \frac{\theta^3}{2} - b^2 \Big)}_{>0} 
        + \sqrt{\Big( \frac{\nu \Delta t \lambda_1 \beta_1^2}{2} - \frac{\theta^3}{2} - b^2 \Big)^2 + \underbrace{4 \Big[ \frac{\nu \Delta t \lambda_1 \beta_2^2}{2} + \frac{(1+\theta) (2 + \theta - \theta^2)}{8} - a^2 \Big]h_{22}}_{>0} } \bigg\} \notag \\
        \geq& \frac{1}{2} \bigg\{ - \Big( \frac{\nu \Delta t \lambda_1 \beta_1^2}{2} - \frac{\theta^3}{2} - b^2 \Big) + \sqrt{\Big( \frac{\nu \Delta t \lambda_1 \beta_1^2}{2} - \frac{\theta^3}{2} - b^2 \Big)^2 } \bigg\} \notag \\
        =& \frac{1}{2} \bigg\{ - \Big( \frac{\nu \Delta t \lambda_1 \beta_1^2}{2} - \frac{\theta^3}{2} - b^2 \Big) 
        + \Big| \frac{\nu \Delta t \lambda_1 \beta_1^2}{2} - \frac{\theta^3}{2} - b^2 \Big| \bigg\} \notag \\
        =& - \Big( \frac{\nu \Delta t \lambda_1 \beta_1^2}{2} - \frac{\theta^3}{2} - b^2 \Big) 
        \geq \frac{\theta^3}{2} \Big[ 1 - \frac{1}{2} \theta (1 - \theta) \Big].
    \end{align*}
	\normalcolor
\end{confidential}
		Finally we solve for $\epsilon$ and derive a lower bound for $1/\epsilon$. For convenienece, we denote
		\begin{align}
			B = \frac{\nu \Delta t \lambda_1 \beta_1^2}{2} - \frac{\theta^3}{2} - b^2 \ (<0), \quad 
			C = \frac{\nu \Delta t \lambda_1 \beta_2^2}{2} + \frac{(1+\theta) (2 + \theta - \theta^2)}{8} - a^2 \ (>0),
			\label{eq:note-B-C} 
		\end{align}
		and obtain via the second equation of \eqref{eq:H-stab-system}
		\begin{align}
			\epsilon = \frac{\sqrt{B^2 + 4 C h_{22}} - (2 h_{22} - B) }{2 h_{22}} 
			\label{eq:sol-epsi} 
		\end{align}
\begin{confidential}
	\color{darkblue}
	\begin{align*}
        &(1 + \epsilon) h_{22} \\
        =& \Big( \frac{\nu \Delta t \lambda_1 \beta_1^2}{2} - \frac{\theta^3}{2} - b^2 \Big) + h_{11} \\
        =& \Big( \frac{\nu \Delta t \lambda_1 \beta_1^2}{2} - \frac{\theta^3}{2} - b^2 \Big) - \frac{1}{2} \Big( \frac{\nu \Delta t \lambda_1 \beta_1^2}{2} - \frac{\theta^3}{2} - b^2 \Big) \\
        &+ \frac{1}{2} \sqrt{ \Big( \frac{\nu \Delta t \lambda_1 \beta_1^2}{2} - \frac{\theta^3}{2} - b^2 \Big)^2 
        + 4 \Big[  \underbrace{\frac{\nu \Delta t \lambda_1 \beta_2^2}{2} + \frac{(1+\theta) (2 + \theta - \theta^2)}{8} - a^2}_{> 0} \Big] h_{22} } \\
        =& \frac{1}{2} \Big( \frac{\nu \Delta t \lambda_1 \beta_1^2}{2} - \frac{\theta^3}{2} - b^2 \Big) 
        + \frac{1}{2} \sqrt{ \Big( \frac{\nu \Delta t \lambda_1 \beta_1^2}{2} - \frac{\theta^3}{2} - b^2 \Big)^2 
        + 4 \Big[ \frac{\nu \Delta t \lambda_1 \beta_2^2}{2} + \frac{(1+\theta) (2 + \theta - \theta^2)}{8} - a^2 \Big] h_{22} }.
    \end{align*}
	\normalcolor
\end{confidential}
\begin{confidential}
	\color{darkblue}
	\begin{align*}
		&(1 + \epsilon) h_{22} = \frac{1}{2} B + \frac{1}{2} \sqrt{B^2 + 4 C h_{22}} \quad \Leftrightarrow \quad
        1 + \epsilon = \frac{\frac{1}{2} B + \frac{1}{2} \sqrt{B^2 + 4 C h_{22}} }{h_{22}} \\
		&\epsilon = \frac{\frac{1}{2} B + \frac{1}{2} \sqrt{B^2 + 4 C h_{22}} }{h_{22}} - 1 \quad \Leftrightarrow \quad
		\epsilon = \frac{\frac{1}{2} B + \frac{1}{2} \sqrt{B^2 + 4 C h_{22}} - h_{22} }{h_{22}} \\
		&\epsilon = \frac{B + \sqrt{B^2 + 4 C h_{22}} - 2 h_{22} }{2 h_{22}}
		\quad \Leftrightarrow \quad 
		\epsilon = \frac{\sqrt{B^2 + 4 C h_{22}} - (2 h_{22} - B) }{2 h_{22}}
    \end{align*}
	\normalcolor
\end{confidential}
		To show $\epsilon > 0$, it suffice to show that $B^2 + 4 C h_{22} > (2 h_{22} - B)^2$ due to the facts that $h_{22}>0$ and $2 h_{22} - B >0$.
		By algebraic calculation, 
		\begin{align}
			B^2 + 4 C h_{22} - (2 h_{22} - B)^2
			= 4 h_{22} (B + C - h_{22}) = 4 h_{22} \Big( \frac{\nu \Delta t \lambda_1}{2} - x \Big). 
			\label{eq:B-C-h22}
		\end{align}
\begin{confidential}
	\color{darkblue}
	\begin{align*}
		&B^2 + 4 C h_{22} - (2 h_{22} - B)^2 \\
		=& B^2 + 4 C h_{22} - (4 h_{22}^2 - 4 B h_{22} + B^2) \\
		=& \cancel{B^2} + 4 C h_{22} - 4 h_{22}^2 + 4 B h_{22} \cancel{- B^2} \\
		=& 4 h_{22} ( B + C - h_{22}).
	\end{align*}
	\begin{align*}
		&B + C - h_{22} \\
		=& \frac{\nu \Delta t \lambda_1 \beta_1^2}{2} - \frac{\theta^3}{2} - b^2 + \frac{\nu \Delta t \lambda_1 \beta_2^2}{2} + \frac{(1+\theta) (2 + \theta - \theta^2)}{8} - a^2 
		- \Big( c^2 - \frac{(1 - \theta)(\theta^2 + \theta - 2)}{8} - \frac{\nu \Delta t \lambda_1 \beta_0^2}{2} \Big) \\
		=& \frac{\nu \Delta t \lambda_1}{2} ( \beta_2^2 + \beta_1^2 + \beta_0^2 ) 
		+ \frac{(1+\theta) (2 + \theta - \theta^2)}{8} + \frac{(1 - \theta)(\theta^2 + \theta - 2)}{8}
        - \frac{\theta^3}{2} - (a^2 + c^2 + b^2) 
	\end{align*}
	\begin{align*}
        &a^2 + c^2 + b^2 \\
        =& \frac{1}{4} \Big[ \frac{1}{2} \Big( \sqrt{x} + \sqrt{\mathcal{D}} \Big) -
        \sqrt{\mathscr{D}} \Big]^2 
        + \frac{1}{4} \Big[ \frac{1}{2} \Big( \sqrt{x} + \sqrt{\mathcal{D}} \Big) + 
        \sqrt{\mathscr{D}} \Big]^2 
        + \frac{1}{4} \Big( \sqrt{x} - \sqrt{\mathcal{D}} \Big)^2 \\
        =& \frac{1}{4} \Big[ \frac{1}{4} \Big( \sqrt{x} + \sqrt{\mathcal{D}} \Big)^2 + \mathscr{D}
        - \sqrt{\mathscr{D}} \Big( \sqrt{x} + \sqrt{\mathcal{D}} \Big) \Big] \\
        &+ \frac{1}{4} \Big[ \frac{1}{4} \Big( \sqrt{x} + \sqrt{\mathcal{D}} \Big)^2 + \mathscr{D}
        + \sqrt{\mathscr{D}} \Big( \sqrt{x} + \sqrt{\mathcal{D}} \Big) \Big] + \frac{1}{4} \Big( \sqrt{x} - \sqrt{\mathcal{D}} \Big)^2 \\
        =& \frac{1}{16} \Big( \sqrt{x} + \sqrt{\mathcal{D}} \Big)^2 + \frac{1}{4} \mathscr{D} 
        \cancel{- \frac{1}{4} \sqrt{\mathscr{D}} \Big( \sqrt{x} + \sqrt{\mathcal{D}} \Big)} \\
        &+ \frac{1}{16} \Big( \sqrt{x} + \sqrt{\mathcal{D}} \Big)^2 + \frac{1}{4} \mathscr{D} 
        \cancel{+ \frac{1}{4} \sqrt{\mathscr{D}} \Big( \sqrt{x} + \sqrt{\mathcal{D}} \Big)} 
        + \frac{1}{4} \Big( \sqrt{x} - \sqrt{\mathcal{D}} \Big)^2 \\
        =& \frac{1}{8} \Big( \sqrt{x} + \sqrt{\mathcal{D}} \Big)^2 + \frac{1}{2} \mathscr{D}
        + \frac{1}{4} \Big( \sqrt{x} - \sqrt{\mathcal{D}} \Big)^2 \\
        =& \frac{1}{8} \Big( \sqrt{x} + \sqrt{\mathcal{D}} \Big)^2
        + \frac{1}{4} \Big( \sqrt{x} - \sqrt{\mathcal{D}} \Big)^2 
        + \frac{1}{2} \Big[ \frac{1}{4} \big( \sqrt{x} + \sqrt{\mathcal{D}})^2 - 4 \Big( \frac{\nu \Delta t \lambda_1 \beta_2 \beta_0}{2} + \frac{a_1^2}{4} \Big) \Big] \\
        =& \frac{1}{8} \Big( \sqrt{x} + \sqrt{\mathcal{D}} \Big)^2
        + \frac{1}{4} \Big( \sqrt{x} - \sqrt{\mathcal{D}} \Big)^2 
        + \frac{1}{8} \Big( \sqrt{x} + \sqrt{\mathcal{D}} \Big)^2
        - 2 \Big( \frac{\nu \Delta t \lambda_1 \beta_2 \beta_0}{2} + \frac{a_1^2}{4} \Big) \\
        =& \frac{1}{4} \Big( \sqrt{x} + \sqrt{\mathcal{D}} \Big)^2
        + \frac{1}{4} \Big( \sqrt{x} - \sqrt{\mathcal{D}} \Big)^2 
        - \Big( \nu \Delta t \lambda_1 \beta_2 \beta_0 + \frac{a_1^2}{2} \Big)\\
        =& \frac{1}{4} \Big\{ x
        + \mathcal{D} + 2 \sqrt{\mathcal{D}} \sqrt{x} \Big\} 
        + \frac{1}{4} \Big\{ x + \mathcal{D} - 2 \sqrt{\mathcal{D}} \sqrt{x} \Big\} 
        - \nu \Delta t \lambda_1 \beta_2 \beta_0 - \frac{1}{2} \frac{\theta (1 - \theta^2)}{2} \\
        =& \frac{1}{4}x 
        + \frac{1}{4} \mathcal{D} \cancel{+ \frac{1}{2} \sqrt{\mathcal{D}} \sqrt{x}} + \frac{1}{4}x
        + \frac{1}{4} \mathcal{D} \cancel{- \frac{1}{2} \sqrt{\mathcal{D}} \sqrt{x}} - \nu \Delta t \lambda_1 \beta_2 \beta_0 - \frac{\theta (1 - \theta^2)}{4} \\
        =& \frac{1}{2} x + \frac{1}{2} \mathcal{D} 
        - \nu \Delta t \lambda_1 \beta_2 \beta_0 - \frac{\theta (1 - \theta^2)}{4} \\
        =& \frac{1}{2} x + \frac{1}{2} \big(x - 4E \big)
        - \nu \Delta t \lambda_1 \beta_2 \beta_0 - \frac{\theta (1 - \theta^2)}{4} \\
        =& \frac{1}{2} x + \frac{1}{2} x - 2E 
        - \nu \Delta t \lambda_1 \beta_2 \beta_0 - \frac{\theta (1 - \theta^2)}{4} \\
        =& x - 2 \Big( \frac{\nu \Delta t \lambda_1 \beta_1 (\beta_2 + \beta_0 )}{2} - a_1^2 \Big) 
        - \nu \Delta t \lambda_1 \beta_2 \beta_0 - \frac{\theta (1 - \theta^2)}{4} \\
        =& x - \nu \Delta t \lambda_1 \beta_1 (\beta_2 + \beta_0 ) + 2 a_1^2 - \nu \Delta t \lambda_1 \beta_2 \beta_0 - \frac{\theta (1 - \theta^2)}{4} \\
        =& x - \nu \Delta t \lambda_1 \big( \beta_2 \beta_1 + \beta_1 \beta_0 + \beta_2 \beta_0 \big) + \cancel{2} \frac{\theta (1 - \theta^2)}{\cancel{2}} - \frac{\theta (1 - \theta^2)}{4} \\
        =& x - \nu \Delta t \lambda_1 \big( \beta_2 \beta_1 + \beta_1 \beta_0 + \beta_2 \beta_0 \big) + \theta (1 - \theta^2)
        - \frac{\theta (1 - \theta^2)}{4} \\
        =& x - \nu \Delta t \lambda_1 \big( \beta_2 \beta_1 + \beta_1 \beta_0 + \beta_2 \beta_0 \big) + \big( 1 - \frac{1}{4} \big) \theta (1 - \theta^2) \\
        =& x - \nu \Delta t \lambda_1 \big( \beta_2 \beta_1 + \beta_1 \beta_0 + \beta_2 \beta_0 \big) + \frac{3}{4} \theta (1 - \theta^2).
    \end{align*}
	\begin{align*}
        &\frac{\nu \Delta t \lambda_1}{2} \big(\beta_2^2 + \beta_1^2 + \beta_0^2 \big) +  \frac{(1+\theta) (2 + \theta - \theta^2)}{8} + \frac{(1 - \theta)(\theta^2 + \theta - 2)}{8}
        - \frac{\theta^3}{2} - ( a^2 + c^2 + b^2 ) \\
        =& \frac{\nu \Delta t \lambda_1}{2} \big(\beta_2^2 + \beta_1^2 + \beta_0^2 \big) +  \frac{(1+\theta) (2 + \theta - \theta^2)}{8} + \frac{(1 - \theta)(\theta^2 + \theta - 2)}{8} - \frac{\theta^3}{2}
        \\
        &- \Big[ x - \nu \Delta t \lambda_1 \big( \beta_2 \beta_1 + \beta_1 \beta_0 + \beta_2 \beta_0 \big) + \frac{3}{4} \theta (1 - \theta^2) \Big] \\
        =& \frac{\nu \Delta t \lambda_1}{2} \big(\beta_2^2 + \beta_1^2 + \beta_0^2 \big) +  \frac{(1+\theta) (2 + \theta - \theta^2)}{8} + \frac{(1 - \theta)(\theta^2 + \theta - 2)}{8} - \frac{\theta^3}{2} \\
        &- x + \nu \Delta t \lambda_1 \big( \beta_2 \beta_1 + \beta_1 \beta_0 + \beta_2 \beta_0 \big) - \frac{3}{4} \theta (1 - \theta^2) \\
        =& \frac{\nu \Delta t \lambda_1}{2} \big(\beta_2^2 + \beta_1^2 + \beta_0^2 \big) + \nu \Delta t \lambda_1 \big( \beta_2 \beta_1 + \beta_1 \beta_0 + \beta_2 \beta_0 \big) - x \\
        &+ \frac{(1+\theta) (2 + \theta - \theta^2)}{8} + \frac{(1 - \theta)(\theta^2 + \theta - 2)}{8} - \frac{\theta^3}{2} - \frac{3}{4} \theta (1 - \theta^2) \\
        =& \frac{\nu \Delta t \lambda_1}{2} \big(\beta_2^2 + \beta_1^2 + \beta_0^2 + 2 \beta_2 \beta_1 + 2 \beta_1 \beta_0 + 2 \beta_2 \beta_0 \big) - x \\
        &+ \frac{(1+\theta) (2 + \theta - \theta^2)}{8} + \frac{(1 - \theta)(\theta^2 + \theta - 2)}{8} - \frac{4 \theta^3}{8} - \frac{6 \theta (1 - \theta^2)}{8}  \\
        =& \frac{\nu \Delta t \lambda_1}{2} \big( \underbrace{\beta_2 + \beta_1 + \beta_0}_{=1} \big)^2 - x \\
        &+ \frac{1}{8} \Big[(1+\theta) (2 + \theta - \theta^2) + (1 - \theta)(\theta^2 + \theta - 2) - 4 \theta^3 - 6 \theta (1 - \theta^2) \Big] \\
        =& \frac{\nu \Delta t \lambda_1}{2} - x 
        + \frac{1}{8} \big( \bcancel{2} + \theta - \theta^2 + 2 \theta + \theta^2 \cancel{- \theta^3} + \theta^2 + \theta \bcancel{- 2} \cancel{- \theta^3} - \theta^2 + 2 \theta \cancel{-  4 \theta^3} - 6 \theta \cancel{+ 6 \theta^3} \big) \\
        =& \frac{\nu \Delta t \lambda_1}{2} - x + \frac{1}{8} \big( \cancel{\theta} \bcancel{- \theta^2} \cancel{+ 2 \theta} \bcancel{+ \theta^2} \bcancel{+ \theta^2} \cancel{+ \theta} \bcancel{- \theta^2} \cancel{+ 2 \theta} \cancel{- 6 \theta} \big) \\
        =& \frac{\nu \Delta t \lambda_1}{2} - x,
    \end{align*}
	By \eqref{eq:x-sol}
    \begin{align*}
        x = \frac{\bigg\{ 4 FE - \Big[\nu \Delta t \lambda_1 \beta_1 ( \beta_2 - \beta_0) \Big]^2 \bigg\}^2}{4 (E - F) \bigg\{ \Big[ \nu \Delta t \lambda_1 \beta_1 ( \beta_2 - \beta_0) \Big]^2 - 4E^2 \bigg\}} < \frac{\nu \Delta t \lambda_1}{2}, \\
    \end{align*}
	For the numerator
    \begin{align*}
        &4 FE - \Big[ \nu \Delta t \lambda_1 \beta_1 ( \beta_2 - \beta_0) \Big]^2 \\
        =& 4 \Big[ \frac{\nu \Delta t \lambda_1 \beta_1 (\beta_2 + \beta_0 )}{2} + 2 \nu \Delta t \lambda_1 \beta_2 \beta_0 \Big] 
        \Big[ \frac{\nu \Delta t \lambda_1 \beta_1 (\beta_2 + \beta_0 )}{2} 
        - a_1^2 \Big] - \Big[\nu \Delta t \lambda_1 \beta_1 ( \beta_2 - \beta_0) \Big]^2 \\
        =&4 \bigg\{ \Big[ \frac{\nu \Delta t \lambda_1 \beta_1 (\beta_2 + \beta_0 )}{2} \Big]^2 + \Big[ 2 \nu \Delta t \lambda_1 \beta_2 \beta_0 - a_1^2 \Big]\Big[ \frac{\nu \Delta t \lambda_1 \beta_1 (\beta_2 + \beta_0 )}{2} \Big] - a_1^2 2 \nu \Delta t \lambda_1 \beta_2 \beta_0 \bigg\} \\
        &- \Big[\nu \Delta t \lambda_1 \beta_1 ( \beta_2 - \beta_0) \Big]^2 \\
        =&\cancel{4} \frac{1}{\cancel{4}}\Big[ \nu \Delta t \lambda_1 \beta_1 (\beta_2 + \beta_0 ) \Big]^2
        + 2 \Big[ 2 \nu \Delta t \lambda_1 \beta_2 \beta_0 - a_1^2 \Big]\Big[ \nu \Delta t \lambda_1 \beta_1 (\beta_2 + \beta_0 ) \Big] - 8 a_1^2  \nu \Delta t \lambda_1 \beta_2 \beta_0 \\
        &- \Big[\nu \Delta t \lambda_1 \beta_1 ( \beta_2 - \beta_0) \Big]^2 \\
        =& \Big( \nu \Delta t \lambda_1 \Big)^2 \beta_1^2 (\beta_2 + \beta_0 )^2 + 4 \Big( \nu \Delta t \lambda_1 \Big) \beta_2 \beta_0 \Big( \nu \Delta t \lambda_1 \Big) \beta_1 (\beta_2 + \beta_0 ) 
        - 2 a_1^2 \beta_1 (\beta_2 + \beta_0 ) \Big( \nu \Delta t \lambda_1 \Big) \\
        &- 8 a_1^2 \beta_2 \beta_0 \Big( \nu \Delta t \lambda_1 \Big) 
        - \beta_1^2 ( \beta_2 - \beta_0)^2 \Big( \nu \Delta t \lambda_1 \Big)^2 \\
        =& \Big( \nu \Delta t \lambda_1 \Big)^2 \Big[ \beta_1^2 (\beta_2 + \beta_0 )^2 + 4 \beta_2 \beta_1 \beta_0 (\beta_2 + \beta_0) - \beta_1^2 ( \beta_2 - \beta_0)^2 \Big] \\
        &- \Big[ 2 a_1^2 \beta_1 (\beta_2 + \beta_0 ) + 8 a_1^2 \beta_2 \beta_0 \Big] \Big( \nu \Delta t \lambda_1 \Big) \\
        =& \Big( \nu \Delta t \lambda_1 \Big)^2 \bigg\{ 
        \beta_1^2 \Big[ (\beta_2 + \beta_0 )^2 - ( \beta_2 - \beta_0)^2 \Big] + 4 \beta_2 \beta_1 \beta_0 (\beta_2 + \beta_0) \bigg\} \\
        &- 2 a_1^2 \Big[ \beta_1 (\beta_2 + \beta_0 ) + 4 \beta_2 \beta_0 \Big] \Big( \nu \Delta t \lambda_1 \Big) \\
        =& \Big( \nu \Delta t \lambda_1 \Big)^2 \bigg\{ 
        \beta_1^2 \Big[ \big( \beta_2 + \beta_0 + \beta_2 - \beta_0 \big) \big( \beta_2 + \beta_0 - ( \beta_2 - \beta_0) \big) \Big] + 4 \beta_2 \beta_1 \beta_0 (\beta_2 + \beta_0) \bigg\} \\
        &- 2 a_1^2 \Big[ \beta_1 (\beta_2 + \beta_0 ) + 4 \beta_2 \beta_0 \Big] \Big( \nu \Delta t \lambda_1 \Big) \\ 
        =& \Big( \nu \Delta t \lambda_1 \Big)^2 \bigg\{ 
        \beta_1^2 \Big[ \big( \beta_2 \cancel{+ \beta_0} + \beta_2 \cancel{- \beta_0} \big) \big( \bcancel{\beta_2} + \beta_0 \bcancel{- \beta_2} + \beta_0 \big) \Big] + 4 \beta_2 \beta_1 \beta_0 (\beta_2 + \beta_0) \bigg\} \\
        &- 2 a_1^2 \Big[ \beta_1 (\beta_2 + \beta_0 ) + 4 \beta_2 \beta_0 \Big] \Big( \nu \Delta t \lambda_1 \Big) \\ 
        =& \Big( \nu \Delta t \lambda_1 \Big)^2 \Big[ \beta_1^2 (2 \beta_2) (2 \beta_0) + 4 \beta_2 \beta_1 \beta_0 (\beta_2 + \beta_0) \Big] - 2 a_1^2 \Big[ \beta_1 (\beta_2 + \beta_0 ) + 4 \beta_2 \beta_0 \Big] \Big( \nu \Delta t \lambda_1 \Big) \\ 
        =& \Big( \nu \Delta t \lambda_1 \Big)^2 \Big[ 4 \beta_2 \beta_1 \beta_0 \beta_1 + 4 \beta_2 \beta_1 \beta_0 (\beta_2 + \beta_0) \Big] - 2 a_1^2 \Big[ \beta_1 (\beta_2 + \beta_0 ) + 4 \beta_2 \beta_0 \Big] \Big( \nu \Delta t \lambda_1 \Big) \\
        =& \Big( \nu \Delta t \lambda_1 \Big)^2 4 \beta_2 \beta_1 \beta_0 \big( \underbrace{\beta_1 + \beta_2 + \beta_0}_{=1} \big)
        - 2 a_1^2 \Big[ \beta_1 (\beta_2 + \beta_0 ) + 4 \beta_2 \beta_0 \Big] \Big( \nu \Delta t \lambda_1 \Big) \\
        =& \Big( \nu \Delta t \lambda_1 \Big)^2 4 \beta_2 \beta_1 \beta_0 - 2 a_1^2 \Big[ \beta_1 (\beta_2 + \beta_0 ) + 4 \beta_2 \beta_0 \Big] \Big( \nu \Delta t \lambda_1 \Big) \\
        =& \Big( \nu \Delta t \lambda_1 \Big) \bigg\{ 4 \beta_2 \beta_1 \beta_0 \Big( \nu \Delta t \lambda_1 \Big) - 2 a_1^2 \Big[ \beta_1 (\beta_2 + \beta_0 ) + 4 \beta_2 \beta_0 \Big] \bigg\} \\
        =& 2 \Big( \nu \Delta t \lambda_1  \Big) \bigg\{ 2 \beta_2 \beta_1 \beta_0 \Big( \nu \Delta t \lambda_1 \Big) - a_1^2 \Big[ \beta_1 (\beta_2 + \beta_0 ) + 4 \beta_2 \beta_0 \Big] \bigg\}
    \end{align*}
	For the denominator, we have 
    \begin{align*}
        &4 (E - F) \bigg\{ \Big[\nu \Delta t \lambda_1 \beta_1 ( \beta_2 - \beta_0) \Big]^2 - 4E^2 \bigg\} \\
        =& 16 \Big( a_1^2 + 2 \nu \Delta t \lambda_1 \beta_2 \beta_0 \Big) \Big( a_1^2 - \nu \Delta t \lambda_1 \beta_1 \beta_0 \Big) \Big( a_1^2 - \nu \Delta t \lambda_1 \beta_1 \beta_2 \Big).
    \end{align*}
	Hence we need 
    \begin{align*}
        &\frac{ 4 \Big( \nu \Delta t \lambda_1 \Big)^{\cancel{2}} \bigg\{ 2 \beta_2 \beta_1 \beta_0 \Big( \nu \Delta t \lambda_1 \Big) - a_1^2 \Big[ \beta_1 (\beta_2 + \beta_0 ) + 4 \beta_2 \beta_0 \Big] \bigg\}^2 }{ 16 \Big( a_1^2 + 2 \nu \Delta t \lambda_1 \beta_2 \beta_0 \Big) \Big( a_1^2 - \nu \Delta t \lambda_1 \beta_1 \beta_0 \Big) \Big( a_1^2 - \nu \Delta t \lambda_1 \beta_1 \beta_2 \Big) } 
        < \frac{1}{2} \cancel{\nu \Delta t \lambda_1} \\
        &4 \Big( \nu \Delta t \lambda_1 \Big) \bigg\{ 2 \beta_2 \beta_1 \beta_0 \Big( \nu \Delta t \lambda_1 \Big) - a_1^2 \Big[ \beta_1 (\beta_2 + \beta_0 ) + 4 \beta_2 \beta_0 \Big] \bigg\}^2 \\
        &< 8 \Big( 2 \nu \Delta t \lambda_1 \beta_2 \beta_0 + a_1^2 \Big) \Big( a_1^2 - \nu \Delta t \lambda_1 \beta_1 \beta_0 \Big) \Big( a_1^2 - \nu \Delta t \lambda_1 \beta_1 \beta_2 \Big) \\
        &\Big( \nu \Delta t \lambda_1 \Big) \bigg\{ 2 \beta_2 \beta_1 \beta_0 \Big( \nu \Delta t \lambda_1 \Big) - a_1^2 \Big[ \beta_1 (\beta_2 + \beta_0 ) + 4 \beta_2 \beta_0 \Big] \bigg\}^2 \\
        &< 2 \Big( 2 \nu \Delta t \lambda_1 \beta_2 \beta_0 + a_1^2 \Big) \Big( a_1^2 - \nu \Delta t \lambda_1 \beta_1 \beta_0 \Big) \Big( a_1^2 - \nu \Delta t \lambda_1 \beta_1 \beta_2 \Big) \\
        & \Big( \nu \Delta t \lambda_1 \Big) \bigg\{ 4 \beta_2^2 \beta_1^2 \beta_0^2 \Big( \nu \Delta t \lambda_1 \Big)^2
        + a_1^4 \Big[ \beta_1 (\beta_2 + \beta_0 ) + 4 \beta_2 \beta_0 \Big]^2 - 4 a_1^2 \beta_2 \beta_1 \beta_0 \Big[ \beta_1 (\beta_2 + \beta_0 ) + 4 \beta_2 \beta_0 \Big] \Big(\nu \Delta t \lambda_1 \Big) \bigg\} \\
        &< 2 \Big( 2 \nu \Delta t \lambda_1 \beta_2 \beta_0 + a_1^2 \Big) \Big[ a_1^4 - \Big( \nu \Delta t \lambda_1 \beta_1 \beta_0 + \nu \Delta t \lambda_1 \beta_1 \beta_2 \Big) a_1^2 + \Big( \nu \Delta t \lambda_1 \Big)^2 \beta_2 \beta_1^2 \beta_0 \Big] \\
		\\
        &4 \beta_2^2 \beta_1^2 \beta_0^2 \Big( \nu \Delta t \lambda_1 \Big)^3 + a_1^4 \Big[ \beta_1 (\beta_2 + \beta_0 ) + 4 \beta_2 \beta_0 \Big]^2 \Big( \nu \Delta t \lambda_1 \Big)
        - 4 a_1^2 \beta_2 \beta_1 \beta_0 \Big[ \beta_1 (\beta_2 + \beta_0 ) + 4 \beta_2 \beta_0 \Big] \Big( \nu \Delta t \lambda_1 \Big)^2 \\
        &< \Big[ 4 \beta_2 \beta_0 \Big( \nu \Delta t \lambda_1 \Big) + 2 a_1^2 \Big] \Big[ a_1^4 - \Big( \nu \Delta t \lambda_1 \beta_1 \beta_0 + \nu \Delta t \lambda_1 \beta_1 \beta_2 \Big) a_1^2 + \Big( \nu \Delta t \lambda_1 \Big)^2 \beta_2 \beta_1^2 \beta_0 \Big] \\
		\\
        &\cancel{4 \beta_2^2 \beta_1^2 \beta_0^2 \Big( \nu \Delta t \lambda_1 \Big)^3} + a_1^4 \Big[ \beta_1 (\beta_2 + \beta_0 ) + 4 \beta_2 \beta_0 \Big]^2 \Big( \nu \Delta t \lambda_1 \Big)
        - 4 a_1^2 \beta_2 \beta_1 \beta_0 \Big[ \beta_1 (\beta_2 + \beta_0 ) + 4 \beta_2 \beta_0 \Big] \Big( \nu \Delta t \lambda_1 \Big)^2 \\
        &< 4 a_1^4 \beta_2 \beta_0 \Big( \nu \Delta t \lambda_1 \Big) - 4 a_1^2 \beta_2 \beta_0 \Big( \nu \Delta t \lambda_1 \Big) \beta_1 (\beta_2 + \beta_0) \Big( \nu \Delta t \lambda_1 \Big) \cancel{+ 4 \beta_2^2 \beta_1^2 \beta_0^2 \Big(\nu \Delta t \lambda_1 \Big)^3} \\
        &+ 2 a_1^6 - 2 a_1^4 \beta_1 (\beta_2 + \beta_0) \Big( \nu \Delta t \lambda_1 \Big) + 2 a_1^2 \beta_2 \beta_1^2 \beta_0 
        \Big( \nu \Delta t \lambda_1 \Big)^2 \\
		\\
        &a_1^4 \Big[ \beta_1 (\beta_2 + \beta_0 ) + 4 \beta_2 \beta_0 \Big]^2 \Big( \nu \Delta t \lambda_1 \Big)
        - 4 a_1^2 \beta_2 \beta_1 \beta_0 \Big[ \beta_1 (\beta_2 + \beta_0 ) + 4 \beta_2 \beta_0 \Big] \Big( \nu \Delta t \lambda_1 \Big)^2 \\
        &<4 a_1^4 \beta_2 \beta_0 \Big( \nu \Delta t \lambda_1 \Big) - 2 a_1^4 \beta_1 (\beta_2 + \beta_0) \Big( \nu \Delta t \lambda_1 \Big) 
        - 4 a_1^2 \beta_2 \beta_1 \beta_0 (\beta_2 + \beta_0) \Big( \nu \Delta t \lambda_1 \Big)^2 
        + 2 a_1^2 \beta_2 \beta_1^2 \beta_0 
        \Big( \nu \Delta t \lambda_1 \Big)^2 + 2 a_1^6. 
    \end{align*}
	Hence 
    \begin{align}
        &4 a_1^2 \beta_2 \beta_1 \beta_0  \Big( \nu \Delta t \lambda_1 \Big)^2 \Big[ \beta_1 (\beta_2 + \beta_0 ) + 4 \beta_2 \beta_0 - (\beta_2 + \beta_0) \Big] 
        + 2 a_1^2 \beta_2 \beta_1^2 \beta_0 
        \Big( \nu \Delta t \lambda_1 \Big)^2 \notag \\
        &+ a^4 \Big( \nu \Delta t \lambda_1 \Big)
        \bigg\{ 4 \beta_2 \beta_0 - 2 \beta_1 (\beta_2 + \beta_0) - \Big[ \beta_1 (\beta_2 + \beta_0 ) + 4 \beta_2 \beta_0 \Big]^2 \bigg\}
        + 2 a_1^6 > 0.  \label{eq:ineq-dt}
    \end{align}
	\begin{align*}
        &\beta_1 (\beta_2 + \beta_0 ) + 4 \beta_2 \beta_0 - (\beta_2 + \beta_0) \\
        =& \big( \frac{1}{2} \theta^2 \big) \Big[ \frac{1}{4} \big( 2 + \theta - \theta^2 \big) + \frac{1}{4} \big( 2 - \theta - \theta^2 \big) \Big] + 4 \cdot \frac{1}{4} \big( 2 + \theta - \theta^2 \big) 
        \cdot \frac{1}{4} \big( 2 - \theta - \theta^2 \big) \\
        &- \Big[ \frac{1}{4} \big( 2 + \theta - \theta^2 \big) + \frac{1}{4} \big( 2 - \theta - \theta^2 \big) \Big] \\
        =& \big( \frac{1}{8} \theta^2 \big) \Big[ 2 \cancel{+ \theta} - \theta^2 + 2 \cancel{- \theta} - \theta^2 \Big]
        + \frac{1}{4} \big( 2 - \theta^2 + \theta \big) \big( 2 - \theta^2 - \theta \big) \\
        &- \frac{1}{4} \Big[ 2 \cancel{+ \theta} - \theta^2 + 2 \cancel{- \theta} - \theta^2 \Big] \\
        =& \big( \frac{1}{4} \theta^2 \big) (2 - \theta^2) 
        + \frac{1}{4} \Big[ (2 - \theta^2)^2 - \theta^2 \Big]
        - \frac{1}{4} 2 (2 - \theta^2) \\
        =& \frac{1}{4} \theta^2 (2 - \theta^2) + \frac{1}{4} (2 - \theta^2)^2 - \frac{1}{4} \theta^2 - \frac{1}{2} (2 - \theta^2) \\
        =& \frac{1}{4} (2 - \theta^2) \big( \cancel{\theta^2} + 2 \cancel{- \theta^2} \big) - \frac{1}{4} \theta^2 - \frac{1}{2} (2 - \theta^2) \\
        =& \cancel{\frac{1}{2} (2 - \theta^2)} - \frac{1}{4} \theta^2 
        \cancel{- \frac{1}{2} (2 - \theta^2)} 
        = - \frac{1}{4} \theta^2 = - \frac{1}{2} \big( \frac{1}{2} \theta^2 \big) = - \frac{1}{2} \beta_1.
    \end{align*}
	Thus \eqref{eq:ineq-dt} become
    \begin{align}
        &4 a_1^2 \beta_2 \beta_1 \beta_0  \Big( \nu \Delta t \lambda_1 \Big)^2 \big( - \frac{1}{2} \beta_1 \big)
        + 2 a_1^2 \beta_2 \beta_1^2 \beta_0 
        \Big( \nu \Delta t \lambda_1 \Big)^2 \notag \\
        &+ a^4 \Big( \nu \Delta t \lambda_1 \Big)
        \bigg\{ 4 \beta_2 \beta_0 - 2 \beta_1 (\beta_2 + \beta_0) - \Big[ \beta_1 (\beta_2 + \beta_0 ) + 4 \beta_2 \beta_0 \Big]^2 \bigg\}
        + 2 a_1^6 > 0. \notag \\
        &\cancel{- 2 a_1^2 \beta_2 \beta_1^2 \beta_0  \Big( \nu \Delta t \lambda_1 \Big)^2} 
        \cancel{+ 2 a_1^2 \beta_2 \beta_1^2 \beta_0 
        \Big( \nu \Delta t \lambda_1 \Big)^2} \notag \\
        &+ a^4 \Big( \nu \Delta t \lambda_1 \Big)
        \bigg\{ 4 \beta_2 \beta_0 - 2 \beta_1 (\beta_2 + \beta_0) - \Big[ \beta_1 (\beta_2 + \beta_0 ) + 4 \beta_2 \beta_0 \Big]^2 \bigg\}
        + 2 a_1^6 > 0. \label{eq:ineq-dt-eq2}
    \end{align}
	\begin{align*}
        &4 \beta_2 \beta_0 - 2 \beta_1 (\beta_2 + \beta_0) - \Big[ \beta_1 (\beta_2 + \beta_0 ) + 4 \beta_2 \beta_0 \Big]^2 \\
        =& 4 \cdot \frac{1}{4} \big( 2 + \theta - \theta^2 \big) 
        \cdot \frac{1}{4} \big( 2 - \theta - \theta^2 \big)
        - 2 \big( \frac{1}{2} \theta^2 \big) \Big[ \frac{1}{4} \big( 2 + \theta - \theta^2 \big) + \frac{1}{4} \big( 2 - \theta - \theta^2 \big) \Big] \\
        &- \bigg\{ \big( \frac{1}{2} \theta^2 \big) \Big[ \frac{1}{4} \big( 2 + \theta - \theta^2 \big) + \frac{1}{4} \big( 2 - \theta - \theta^2 \big) \Big] + 4 \cdot \frac{1}{4} \big( 2 + \theta - \theta^2 \big) \cdot \frac{1}{4} \big( 2 - \theta - \theta^2 \big) \bigg\}^2 \\
        =& \frac{1}{4} \big( 2 - \theta^2 + \theta \big) \big( 2 - \theta^2 - \theta \big) - \theta^2 \frac{1}{4} \big( 2 \cancel{+ \theta} - \theta^2 + 2 \cancel{- \theta} - \theta^2 \big) \\
        &- \bigg\{ \frac{1}{8} \theta^2 \big( 2 \cancel{+ \theta} - \theta^2 + 2 \cancel{- \theta} - \theta^2 \big) + \frac{1}{4} \big( 2 - \theta^2 + \theta \big) \big( 2 - \theta^2 - \theta \big) \bigg\}^2 \\
        =& \frac{1}{4} \Big[ (2 - \theta^2)^2 - \theta^2 \Big] 
        - \frac{1}{4} \theta^2 \cdot 2 (2 - \theta^2) 
        - \bigg\{ \frac{1}{8} \theta^2 \cdot 2 \big( 2 - \theta^2 \big) + \frac{1}{4} \Big[ (2 - \theta^2)^2 - \theta^2 \Big] \bigg\}^2 \\
        =& \frac{1}{4} (2 - \theta^2)^2 - \frac{1}{4} \theta^2 
        - \frac{1}{2} \theta^2 (2 - \theta^2)
        - \bigg\{ \frac{1}{4} \theta^2 \big( 2 - \theta^2 \big) + \frac{1}{4} (2 - \theta^2)^2 - \frac{1}{4} \theta^2 \bigg\}^2 \\
        =& \frac{1}{4} \big( 4 - 4 \theta^2 + \theta^4 \big) - \frac{1}{4} \theta^2 - \theta^2 + \frac{1}{2} \theta^4  
        - \bigg\{ \frac{1}{4} \big( 2 - \theta^2 \big) ( \cancel{\theta^2} + 2 \cancel{- \theta^2}) - \frac{1}{4} \theta^2 \bigg\}^2 \\
        =& 1 - \theta^2 + \frac{1}{4} \theta^4 - \frac{5}{4} \theta^2
        + \frac{1}{2} \theta^4 - \Big[ \frac{1}{2} (2 - \theta^2) - \frac{1}{4} \theta^2 \Big]^2 \\
        =& 1 - (1 + \frac{5}{4}) \theta^2 + \big( \frac{1}{4} + \frac{1}{2} \big) \theta^4 - \Big( 1 - \frac{1}{2} \theta^2 - \frac{1}{4} \theta^2 \Big)^2 \\
        =& 1 - \frac{9}{4} \theta^2 + \frac{3}{4} \theta^4 
        - \big( 1 - \frac{3}{4} \theta^2 \big)^2 
        = 1 - \frac{9}{4} \theta^2 + \frac{3}{4} \theta^4 
        - \big( 1 - \frac{3}{2} \theta^2 + \frac{9}{16} \theta^4 \big) \\
        =& \cancel{1} - \frac{9}{4} \theta^2 + \frac{3}{4} \theta^4 
        \cancel{- 1} + \frac{3}{2} \theta^2 - \frac{9}{16} \theta^4  
        = \big( - \frac{9}{4} + \frac{3}{2} \big) \theta^2  
        + \big( \frac{3}{4} - \frac{9}{16} \big) \theta^4 \\ 
        =& \big( - \frac{9}{4} + \frac{6}{4} \big) \theta^2
        + \big( \frac{12}{16} - \frac{9}{16} \big) \theta^4 
        = - \frac{3}{4} \theta^2 + \frac{3}{16} \theta^4 
        = - \frac{3}{4} \theta^2 (1 - \frac{1}{4} \theta^2)
    \end{align*}
	Hence \eqref{eq:ineq-dt-eq2} becomes
    \begin{align*}
        &a^4 \Big( \nu \Delta t \lambda_1 \Big)  \Big[- \frac{3}{4} \theta^2 (1 - \frac{1}{4} \theta^2) \Big] + 2 a_1^6 > 0 
        \ \Leftrightarrow \ 
        a^4 \Big( \nu \Delta t \lambda_1 \Big)  \Big[\frac{3}{4} \theta^2 (1 - \frac{1}{4} \theta^2) \Big] < 2 a_1^6 \\
        & \Big( \nu \Delta t \lambda_1 \Big)  \Big[\frac{3}{4} \theta^2 (1 - \frac{1}{4} \theta^2) \Big] < 2 a_1^2 
        \ \Leftrightarrow \ 
        \Big( \nu \Delta t \lambda_1 \Big)  \Big[\frac{3}{4} \theta^{\bcancel{2}} (1 - \frac{1}{4} \theta^2) \Big] < \cancel{2} \frac{ \bcancel{\theta} (1 - \theta^2)}{\cancel{2}} \\
        &\Big( \nu \Delta t \lambda_1 \Big)  \Big[\frac{3}{4} \theta (1 - \frac{1}{4} \theta^2) \Big] < (1 - \theta^2)
        \ \Leftrightarrow \ 
        \nu \Delta t \lambda_1 < \frac{4 (1 - \theta^2)}{3 \theta (1 - \frac{1}{4} \theta^2)} \ \Leftrightarrow \ 
        \Delta t < \frac{4 (1 - \theta^2)}{3 \nu \lambda_1 \theta (1 - \frac{1}{4} \theta^2)}.
    \end{align*}
	\normalcolor
\end{confidential}
		We need $\frac{\nu \Delta t \lambda_1}{2} - x > 0$ and obtain 
		\begin{align*}
			\Delta t < \frac{4 (1 - \theta^2)}{3 \nu \lambda_1 \theta (1 - \frac{1}{4} \theta^2)}, 
		\end{align*}
		which is less restrictive than \eqref{eq:dt-cond-2}. Hence we keep the time step limit in \eqref{eq:dt-cond-2}. 
\begin{confidential}
	\color{darkblue}
	Then we compare the time step restriction above with \eqref{eq:dt-cond-2}. We first show:  for all $\theta \in (0,1)$
    \begin{align*}
        &\frac{2 (1 - \theta)}{\bcancel{\nu} \cancel{\lambda_1}} < \frac{4 (1 - \theta^2)}{3 \bcancel{\nu} \cancel{\lambda_1}\theta (1 - \frac{1}{4} \theta^2)} \ \Leftrightarrow \          
        2 \cancel{(1 - \theta)} < \frac{4 \cancel{(1 - \theta)}(1 + \theta)}{3 \theta (1 - \frac{1}{4} \theta^2)} \ \Leftrightarrow \       
        1 < \frac{2 (1 + \theta)}{3 \theta (1 - \frac{1}{4} \theta^2)} \\
        &3 \theta (1 - \frac{1}{4} \theta^2) < 2 (1 + \theta) 
        \ \Leftrightarrow \ 
        3 \theta \big( 1 - \frac{1}{2} \theta \big) \big( 1 + \frac{1}{2} \theta \big) < 2 (1 + \theta).
    \end{align*}
	For lefthand side, we have 
    \begin{align*}
        3 \theta \big( 1 - \frac{1}{2} \theta \big) \big( 1 + \frac{1}{2} \theta \big) 
        \leq 3 \Big( \frac{\theta + 1 \cancel{- \frac{1}{2}\theta} + 1 \cancel{+ \frac{1}{2} \theta} }{3} \Big)^3 
        \leq 3 \Big( \frac{2 + \theta}{3} \Big) \Big( \underbrace{\frac{2 + \theta}{3}}_{ \in (\frac{2}{3},1)} \Big)^2 
        < (2 + \theta) < 2 (1 + \theta).
    \end{align*}
	\normalcolor
\end{confidential}
		By \eqref{eq:sol-epsi} and \eqref{eq:B-C-h22}, we have 
		\begin{align}
			\frac{1}{\epsilon}
			= \frac{\sqrt{B^2 + 4 C h_{22}} + \big(2 h_{22} - B \big)}{2 \big( B + C - h_{22} \big)} 
			= \frac{\sqrt{B^2 + 4 C h_{22}} + \big(2 h_{22} - B \big)}{2 \big( \frac{\nu \Delta t \lambda_1}{2} - x \big)}
			\label{eq:sol-1-epsi}
		\end{align}	
\begin{confidential}
	\color{darkblue}
	\begin{align*}
        \frac{1}{\epsilon}
        =& \frac{2 h_{22}}{\sqrt{B^2 + 4 C h_{22}} - \big(2 h_{22} - B \big)} \\
        =& \frac{2 h_{22} \Big[ \sqrt{B^2 + 4 C h_{22}} + \big(2 h_{22} - B \big) \Big]}{\Big[\sqrt{B^2 + 4 C h_{22}} - \big(2 h_{22} - B \big)\Big] \Big[ \sqrt{B^2 + 4 C h_{22}} + \big(2 h_{22} - B \big) \Big]} \\
        =& \frac{2 h_{22} \Big[ \sqrt{B^2 + 4 C h_{22}} + \big(2 h_{22} - B \big) \Big]}{ B^2 + 4 C h_{22} - \big(2 h_{22} - B \big)^2} \\
        =& \frac{2 h_{22} \Big[ \sqrt{B^2 + 4 C h_{22}} + \big(2 h_{22} - B \big) \Big]}{ B^2 + 4 C h_{22} - \big(4 h_{22}^2 + B^2 - 4 h_{22} B \big)} \\
        =& \frac{2 h_{22} \Big[ \sqrt{B^2 + 4 C h_{22}} + \big(2 h_{22} - B \big) \Big]}{ \cancel{B^2} + 4 C h_{22} - 4 h_{22}^2 \cancel{- B^2} + 4 h_{22} B } \\
        =& \frac{2 \cancel{h_{22}} \Big[ \sqrt{B^2 + 4 C h_{22}} + \big(2 h_{22} - B \big) \Big]}{ 4 \cancel{h_{22}} \big( C - h_{22} + B \big)} 
        = \frac{\sqrt{B^2 + 4 C h_{22}} + \big(2 h_{22} - B \big)}{2 \big( B + C - h_{22} \big)}.
    \end{align*}
	\normalcolor
\end{confidential}
		We claim that under the time step restriction in \eqref{eq:dt-limit-1}, $x < \frac{\nu \Delta t \lambda_1}{4}$. 
\begin{confidential}
	\color{darkblue}
	\begin{align*}
        &\frac{ 4 \Big( \nu \Delta t \lambda_1 \Big)^{\cancel{2}} \bigg\{ 2 \beta_2 \beta_1 \beta_0 \Big( \nu \Delta t \lambda_1 \Big) - a_1^2 \Big[ \beta_1 (\beta_2 + \beta_0 ) + 4 \beta_2 \beta_0 \Big] \bigg\}^2 }{ 16 \Big( a_1^2 + 2 \nu \Delta t \lambda_1 \beta_2 \beta_0 \Big) \Big( a_1^2 - \nu \Delta t \lambda_1 \beta_1 \beta_0 \Big) \Big( a_1^2 - \nu \Delta t \lambda_1 \beta_1 \beta_2 \Big) } 
        < \frac{1}{4} \cancel{\nu \Delta t \lambda_1} \\
		\\
        &4 \Big( \nu \Delta t \lambda_1 \Big) \bigg\{ 2 \beta_2 \beta_1 \beta_0 \Big( \nu \Delta t \lambda_1 \Big) - a_1^2 \Big[ \beta_1 (\beta_2 + \beta_0 ) + 4 \beta_2 \beta_0 \Big] \bigg\}^2 \\
        &< 4 \Big( 2 \nu \Delta t \lambda_1 \beta_2 \beta_0 + a_1^2 \Big) \Big( a_1^2 - \nu \Delta t \lambda_1 \beta_1 \beta_0 \Big) \Big( a_1^2 - \nu \Delta t \lambda_1 \beta_1 \beta_2 \Big) \\
		\\
        &\Big( \nu \Delta t \lambda_1 \Big) \bigg\{ 2 \beta_2 \beta_1 \beta_0 \Big( \nu \Delta t \lambda_1 \Big) - a_1^2 \Big[ \beta_1 (\beta_2 + \beta_0 ) + 4 \beta_2 \beta_0 \Big] \bigg\}^2 \\
        &< \Big( 2 \nu \Delta t \lambda_1 \beta_2 \beta_0 + a_1^2 \Big) \Big( a_1^2 - \nu \Delta t \lambda_1 \beta_1 \beta_0 \Big) \Big( a_1^2 - \nu \Delta t \lambda_1 \beta_1 \beta_2 \Big) \\
		\\
        & \Big( \nu \Delta t \lambda_1 \Big) \bigg\{ 4 \beta_2^2 \beta_1^2 \beta_0^2 \Big( \nu \Delta t \lambda_1 \Big)^2
        + a_1^4 \Big[ \beta_1 (\beta_2 + \beta_0 ) + 4 \beta_2 \beta_0 \Big]^2 - 4 a_1^2 \beta_2 \beta_1 \beta_0 \Big[ \beta_1 (\beta_2 + \beta_0 ) + 4 \beta_2 \beta_0 \Big] \Big( \nu \Delta t \lambda_1 \Big) \bigg\} \\
        &< \Big( 2 \nu \Delta t \lambda_1 \beta_2 \beta_0 + a_1^2 \Big) \Big[ a_1^4 - \Big( \nu \Delta t \lambda_1 \beta_1 \beta_0 + \nu \Delta t \lambda_1 \beta_1 \beta_2 \Big) a_1^2 + \Big( \nu \Delta t \lambda_1 \Big)^2 \beta_2 \beta_1^2 \beta_0 \Big] \\
		\\
        &4 \beta_2^2 \beta_1^2 \beta_0^2 \Big( \nu \Delta t \lambda_1 \Big)^3 + a_1^4 \Big[ \beta_1 (\beta_2 + \beta_0 ) + 4 \beta_2 \beta_0 \Big]^2 \Big( \nu \Delta t \lambda_1 \Big)
        - 4 a_1^2 \beta_2 \beta_1 \beta_0 \Big[ \beta_1 (\beta_2 + \beta_0 ) + 4 \beta_2 \beta_0 \Big] \Big( \nu \Delta t \lambda_1 \Big)^2 \\
        &< \Big[ 2 \beta_2 \beta_0 \Big( \nu \Delta t \lambda_1 \Big) + a_1^2 \Big] \Big[ a_1^4 - \Big( \nu \Delta t \lambda_1 \beta_1 \beta_0 + \nu \Delta t \lambda_1 \beta_1 \beta_2 \Big) a_1^2 + \Big( \nu \Delta t \lambda_1 \Big)^2 \beta_2 \beta_1^2 \beta_0 \Big] \\
		\\
        &\underbrace{4 \beta_2^2 \beta_1^2 \beta_0^2 \Big( \nu \Delta t \lambda_1 \Big)^3} + a_1^4 \Big[ \beta_1 (\beta_2 + \beta_0 ) + 4 \beta_2 \beta_0 \Big]^2 \Big( \nu \Delta t \lambda_1 \Big)
        - 4 a_1^2 \beta_2 \beta_1 \beta_0 \Big[ \beta_1 (\beta_2 + \beta_0 ) + 4 \beta_2 \beta_0 \Big] \Big( \nu \Delta t \lambda_1 \Big)^2 \\
        &< 2 a_1^4 \beta_2 \beta_0 \Big( \nu \Delta t \lambda_1 \Big) - 2 a_1^2 \beta_2 \beta_0 \Big( \nu \Delta t \lambda_1 \Big) \beta_1 (\beta_2 + \beta_0) \Big( \nu \Delta t \lambda_1 \Big) + \underbrace{2 \beta_2^2 \beta_1^2 \beta_0^2 \Big( \nu \Delta t \lambda_1 \Big)^3} \\
        &+ a_1^6 - a_1^4 \beta_1 (\beta_2 + \beta_0) \Big( \nu \Delta t \lambda_1 \Big) + a_1^2 \beta_2 \beta_1^2 \beta_0 \Big( \nu \Delta t \lambda_1 \Big)^2 \\
		\\ 
        &2 \beta_2^2 \beta_1^2 \beta_0^2 \Big( \nu \Delta t \lambda_1 \Big)^3
        + a_1^4 \Big[ \beta_1 (\beta_2 + \beta_0 ) + 4 \beta_2 \beta_0 \Big]^2 \Big( \nu \Delta t \lambda_1 \Big)
        - 4 a_1^2 \beta_2 \beta_1 \beta_0 \Big[ \beta_1 (\beta_2 + \beta_0 ) + 4 \beta_2 \beta_0 \Big] \Big( \nu \Delta t \lambda_1 \Big)^2 \\
        &<2 a_1^4 \beta_2 \beta_0 \Big( \nu \Delta t \lambda_1 \Big) - a_1^4 \beta_1 (\beta_2 + \beta_0) \Big( \nu \Delta t \lambda_1 \Big) 
        - 2 a_1^2 \beta_2 \beta_1 \beta_0 (\beta_2 + \beta_0) \Big( \nu \Delta t \lambda_1 \Big)^2 
        + a_1^2 \beta_2 \beta_1^2 \beta_0 
        \Big( \nu \Delta t \lambda_1 \Big)^2 + a_1^6. 
    \end{align*}
	Hence 
    \begin{align}
        &- 2 \beta_2^2 \beta_1^2 \beta_0^2 \Big( \nu \Delta t \lambda_1 \Big)^3 \notag \\
        &+ 4 a_1^2 \beta_2 \beta_1 \beta_0  \Big( \nu \Delta t \lambda_1 \Big)^2 \Big[ \beta_1 (\beta_2 + \beta_0 ) + 4 \beta_2 \beta_0 - \frac{1}{2} (\beta_2 + \beta_0) \Big] 
        + a_1^2 \beta_2 \beta_1^2 \beta_0 
        \Big( \nu \Delta t \lambda_1 \Big)^2 \notag \\
        &+ a^4 \Big( \nu \Delta t \lambda_1 \Big)
        \bigg\{ 2 \beta_2 \beta_0 - \beta_1 (\beta_2 + \beta_0) - \Big[ \beta_1 (\beta_2 + \beta_0 ) + 4 \beta_2 \beta_0 \Big]^2 \bigg\} + a_1^6 > 0  \notag \\
		\notag \\
		&- 2 \beta_2^2 \beta_1^2 \beta_0^2 \Big( \nu \Delta t \lambda_1 \Big)^3 \notag \\
        &+ 4 a_1^2 \beta_2 \beta_1 \beta_0  \Big( \nu \Delta t \lambda_1 \Big)^2 \Big\{ \Big[ \beta_1 (\beta_2 + \beta_0 ) + 4 \beta_2 \beta_0 - \frac{1}{2} (\beta_2 + \beta_0) \Big] + \frac{1}{4} \beta_1 \Big\} \notag \\
        &+ a^4 \Big( \nu \Delta t \lambda_1 \Big)
        \bigg\{ 2 \beta_2 \beta_0 - \beta_1 (\beta_2 + \beta_0) - \Big[ \beta_1 (\beta_2 + \beta_0 ) + 4 \beta_2 \beta_0 \Big]^2 \bigg\} +  a_1^6 > 0, 
		\label{eq:ineq2-dt-gen}
    \end{align}
	\begin{align*}
        &\beta_1 (\beta_2 + \beta_0 ) + 4 \beta_2 \beta_0 - \frac{1}{2} (\beta_2 + \beta_0) \\
        =& \big( \frac{1}{2} \theta^2 \big) \Big[ \frac{1}{4} \big( 2 + \theta - \theta^2 \big) + \frac{1}{4} \big( 2 - \theta - \theta^2 \big) \Big] + 4 \cdot \frac{1}{4} \big( 2 + \theta - \theta^2 \big) 
        \cdot \frac{1}{4} \big( 2 - \theta - \theta^2 \big) \\
        &- \frac{1}{2} \Big[ \frac{1}{4} \big( 2 + \theta - \theta^2 \big) + \frac{1}{4} \big( 2 - \theta - \theta^2 \big) \Big] \\
        =& \big( \frac{1}{8} \theta^2 \big) \Big[ 2 \cancel{+ \theta} - \theta^2 + 2 \cancel{- \theta} - \theta^2 \Big]
        + \frac{1}{4} \big( 2 - \theta^2 + \theta \big) \big( 2 - \theta^2 - \theta \big) \\
        &- \frac{1}{8} \Big[ 2 \cancel{+ \theta} - \theta^2 + 2 \cancel{- \theta} - \theta^2 \Big] \\
        =& \big( \frac{1}{4} \theta^2 \big) (2 - \theta^2) 
        + \frac{1}{4} \Big[ (2 - \theta^2)^2 - \theta^2 \Big]
        - \frac{1}{8} 2 (2 - \theta^2) \\
        =& \frac{1}{4} \theta^2 (2 - \theta^2) + \frac{1}{4} (2 - \theta^2)^2 - \frac{1}{4} \theta^2 - \frac{1}{4} (2 - \theta^2) \\
        =& \frac{1}{4} (2 - \theta^2) \big( \cancel{\theta^2} + 2 \cancel{- \theta^2} \big) - \frac{1}{4} \theta^2 - \frac{1}{4} (2 - \theta^2) \\
        =& \frac{1}{2} (2 - \theta^2) - \frac{1}{4} (2 - \theta^2)
        - \frac{1}{4} \theta^2 
        = \frac{1}{4} (2 - \theta^2) - \frac{1}{4} \theta^2  \\
        =& \frac{1}{2} - \frac{1}{4} \theta^2 - \frac{1}{4} \theta^2 
        = \frac{1}{2} - \Big( \frac{1}{4} + \frac{1}{4} \Big) \theta^2
        = \frac{1}{2} - \frac{1}{2} \theta^2.
    \end{align*}
	\begin{align*}
        &\beta_1 (\beta_2 + \beta_0 ) + 4 \beta_2 \beta_0 - \frac{1}{2} (\beta_2 + \beta_0) + \frac{1}{4} \beta_1 \\
        =& \frac{1}{2} - \frac{1}{2} \theta^2
        + \frac{1}{4} \cdot \frac{1}{2}\theta^2 
        = \frac{1}{2} - \Big( \frac{1}{2} - \frac{1}{8} \Big) \theta^2 
        = \frac{1}{2} - \Big( \frac{4}{8} - \frac{1}{8} \Big) \theta^2 
        = \frac{1}{2} - \frac{3}{8} \theta^2 
        = \frac{1}{2} \big( 1 - \frac{3}{4} \theta^2 \big).
    \end{align*}
	\begin{align*}
        &2 \beta_2 \beta_0 - \beta_1 (\beta_2 + \beta_0) - \Big[ \beta_1 (\beta_2 + \beta_0 ) + 4 \beta_2 \beta_0 \Big]^2 \\
        =& 2 \cdot \frac{1}{4} \big( 2 + \theta - \theta^2 \big) 
        \cdot \frac{1}{4} \big( 2 - \theta - \theta^2 \big)
        - \big( \frac{1}{2} \theta^2 \big) \Big[ \frac{1}{4} \big( 2 + \theta - \theta^2 \big) + \frac{1}{4} \big( 2 - \theta - \theta^2 \big) \Big] \\
        &- \bigg\{ \big( \frac{1}{2} \theta^2 \big) \Big[ \frac{1}{4} \big( 2 + \theta - \theta^2 \big) + \frac{1}{4} \big( 2 - \theta - \theta^2 \big) \Big] + 4 \cdot \frac{1}{4} \big( 2 + \theta - \theta^2 \big) \cdot \frac{1}{4} \big( 2 - \theta - \theta^2 \big) \bigg\}^2 \\
        =& \frac{1}{8} \big( 2 - \theta^2 + \theta \big) \big( 2 - \theta^2 - \theta \big) - \theta^2 \frac{1}{8} \big( 2 \cancel{+ \theta} - \theta^2 + 2 \cancel{- \theta} - \theta^2 \big) \\
        &- \bigg\{ \frac{1}{8} \theta^2 \big( 2 \cancel{+ \theta} - \theta^2 + 2 \cancel{- \theta} - \theta^2 \big) + \frac{1}{4} \big( 2 - \theta^2 + \theta \big) \big( 2 - \theta^2 - \theta \big) \bigg\}^2 \\
        =& \frac{1}{8} \Big[ (2 - \theta^2)^2 - \theta^2 \Big] 
        - \frac{1}{8} \theta^2 \cdot 2 (2 - \theta^2) 
        - \bigg\{ \frac{1}{8} \theta^2 \cdot 2 \big( 2 - \theta^2 \big) + \frac{1}{4} \Big[ (2 - \theta^2)^2 - \theta^2 \Big] \bigg\}^2 \\
        =& \frac{1}{8} (2 - \theta^2)^2 - \frac{1}{8} \theta^2 
        - \frac{1}{4} \theta^2 (2 - \theta^2)
        - \bigg\{ \frac{1}{4} \theta^2 \big( 2 - \theta^2 \big) + \frac{1}{4} (2 - \theta^2)^2 - \frac{1}{4} \theta^2 \bigg\}^2 \\
        =& \frac{1}{8} \big( 4 - 4 \theta^2 + \theta^4 \big) - \frac{1}{8} \theta^2 - \frac{1}{2} \theta^2 + \frac{1}{4} \theta^4  
        - \bigg\{ \frac{1}{4} \big( 2 - \theta^2 \big) ( \cancel{\theta^2} + 2 \cancel{- \theta^2}) - \frac{1}{4} \theta^2 \bigg\}^2 \\
        =& \frac{1}{2} - \frac{1}{2} \theta^2 + \frac{1}{8} \theta^4 - \frac{5}{8} \theta^2
        + \frac{1}{4} \theta^4 - \Big[ \frac{1}{2} (2 - \theta^2) - \frac{1}{4} \theta^2 \Big]^2 \\
        =& \frac{1}{2} - (\frac{1}{2} + \frac{5}{8}) \theta^2 + \big( \frac{1}{8} + \frac{1}{4} \big) \theta^4 - \Big( 1 - \frac{1}{2} \theta^2 - \frac{1}{4} \theta^2 \Big)^2 \\
        =& \frac{1}{2} - \frac{9}{8} \theta^2 + \frac{3}{8} \theta^4 
        - \big( 1 - \frac{3}{4} \theta^2 \big)^2 \\
        =& \frac{1}{2} \Big(1 - \frac{9}{4} \theta^2 + \frac{3}{4} \theta^4 \Big)
        - \big( 1 - \frac{3}{4} \theta^2 \big)^2  
        \\ 
        =& \frac{1}{2} \Big[ 1 - \frac{6}{4} \theta^2 + \frac{9}{16} \theta^4 - \frac{3}{4} \theta^2 + \frac{3}{16} \theta^4 \Big]
        - \big( 1 - \frac{3}{4} \theta^2 \big)^2 \\
        =& \frac{1}{2} \Big[ \Big( 1 -  \frac{3}{4} \theta^2 \Big)^2 
        - \frac{3}{4} \theta^2 \Big( 1 - \frac{1}{4} \theta^2 \Big) \Big]
        - \big( 1 - \frac{3}{4} \theta^2 \big)^2 \\
        =& \frac{1}{2} \Big( 1 -  \frac{3}{4} \theta^2 \Big)^2 
        - \frac{1}{2} \frac{3}{4} \theta^2 \Big( 1 - \frac{1}{4} \theta^2 \Big)
        - \big( 1 - \frac{3}{4} \theta^2 \big)^2 \\ 
        =& - \frac{1}{2} \big( 1 - \frac{3}{4} \theta^2 \big)^2
        - \frac{3}{8} \theta^2 \Big( 1 - \frac{1}{4} \theta^2 \Big).
    \end{align*}
	\normalcolor
\end{confidential}
		The requirement $x < \frac{\nu \Delta t \lambda_1}{4}$ is equivalent to the following inequality 
		\begin{align}
        &- 2 \beta_2^2 \beta_1^2 \beta_0^2 \Big( \nu \Delta t \lambda_1 \Big)^3 
        + 2 \Big( 1 - \frac{3}{4} \theta^2 \Big)  
        a_1^2 \beta_2 \beta_1 \beta_0  \Big( \nu \Delta t \lambda_1 \Big)^2 \notag \\
        &- a^4 \Big( \nu \Delta t \lambda_1 \Big)
        \bigg\{ \frac{1}{2} \big( 1 - \frac{3}{4} \theta^2 \big)^2
        + \frac{3}{8} \theta^2 \Big( 1 - \frac{1}{4} \theta^2 \Big) \bigg\} +  a_1^6 > 0.
        \label{eq:x-dt-cond-2}
    \end{align}
	To ensure that \eqref{eq:x-dt-cond-2} holds, it suffices to have the following two inequalities 
	\begin{align*}
        &- 2 \beta_2^2 \beta_1^2 \beta_0^2 \Big( \nu \Delta t \lambda_1 \Big)^3 
        + 2 \Big( 1 - \frac{3}{4} \theta^2 \Big)  
        a_1^2 \beta_2 \beta_1 \beta_0  \Big( \nu \Delta t \lambda_1\Big)^2 > 0, \\
        &- a^4 \Big( \nu \Delta t \lambda_1 \Big)
        \bigg\{ \frac{1}{2} \big( 1 - \frac{3}{4} \theta^2 \big)^2
        + \frac{3}{8} \theta^2 \Big( 1 - \frac{1}{4} \theta^2 \Big) \bigg\} + a_1^6 > 0. 
    \end{align*}
	For the first equation, 
\begin{confidential}
	\color{darkblue}
    \begin{align*}
        &\cancel{2} \beta_2^2 \beta_1^2 \beta_0^2 \Big( \nu \Delta t \lambda_1 \Big)^{\bcancel{3}} 
        < \cancel{2} \Big( 1 - \frac{3}{4} \theta^2 \Big)  
        a_1^2 \beta_2 \beta_1 \beta_0  \bcancel{\Big( \nu \Delta t \lambda_1 \Big)^2} \\
        &\beta_2 \beta_1 \beta_0 \Big( \nu \Delta t \lambda_1 \Big)
        < \Big( 1 - \frac{3}{4} \theta^2 \Big) \frac{\theta (1 - \theta^2)}{2} \\
        &\frac{1}{4} \big( 2 + \theta - \theta^2 \big) \cancel{\frac{1}{2}} \theta^{\bcancel{2}} 
        \frac{1}{4} \big( 2 - \theta - \theta^2 \big) \Big( \nu \Delta t \lambda_1 \Big) 
        < \Big( 1 - \frac{3}{4} \theta^2 \Big) \frac{\bcancel{\theta} (1 - \theta^2)}{\cancel{2}} \\ 
        &\frac{1}{16} \theta (2 - \theta)\cancel{(1 + \theta)} (2 + \theta)\bcancel{(1 - \theta)} \Big( \nu \Delta t \lambda_1 \Big) 
        < \Big( 1 - \frac{3}{4} \theta^2 \Big) \bcancel{(1 - \theta)} \cancel{(1+\theta)}  \\
        & \Big( \nu \Delta t \lambda_1 \Big)
        < \frac{16 \Big( 1 - \frac{3}{4} \theta^2 \Big)}{ \theta (2 - \theta) (2 + \theta) }
        = \frac{ 4 (4 - 3 \theta^2)}{ \theta (2 - \theta) (2 + \theta) }. 
    \end{align*}
	\begin{align*}
        &\frac{ 4 (4 - 3 \theta^2)}{ \theta (2 - \theta) (2 + \theta) }
        > \frac{4 (4 - 4 \theta^2) }{\theta (2 - \theta) (2 + \theta)}
        = \frac{16 (1 - \theta) (1 + \theta) }{\theta (2 - \theta) (2 + \theta)}, \\
        & \frac{16 (1 - \theta) (1 + \theta) }{\theta (2 - \theta) (2 + \theta)} > 2 (1 - \theta) \ \Leftrightarrow \ 
        \frac{8 (1 + \theta) }{\theta (2 - \theta) (2 + \theta)} > 1
        \ \Leftrightarrow \ 
        8 (1 + \theta) > \theta (2 - \theta) (2 + \theta) \\
        & \theta (2 - \theta) (2 + \theta)
        = 4 ( \frac{1}{2} \theta) (2 - \theta) (1 + \frac{1}{2} \theta) 
        \leq 4 \Big( \frac{ \cancel{\frac{1}{2} \theta} + 2 \cancel{- \theta} + 1 \cancel{+ \frac{1}{2} \theta} }{3} \Big)^3 
        = 4 < 8 (1 + \theta).
    \end{align*}
	\normalcolor
\end{confidential}
		\begin{align*}
			\Delta t 
			< \frac{4 (4 - 3 \theta^2)}{ \nu \lambda_1 \theta (2 - \theta) (2 + \theta)}, 
		\end{align*}
		which is less restrictive the the restriction in \eqref{eq:dt-cond-2}.
		For the second equation, we obtain the time step restriction 
		\begin{align*}
			\Delta t < \frac{8 \theta (1 - \theta^2)}{ \nu \lambda_1 ( 8 - 6 \theta^2 + 3 \theta^4)}.  
		\end{align*}
\begin{confidential}
	\color{darkblue}
    \begin{align*}
        &a^4 \Big( \nu \Delta t \lambda_1 \Big)
        \bigg\{ \frac{1}{2} \big( 1 - \frac{3}{4} \theta^2 \big)^2
        + \frac{3}{8} \theta^2 \Big( 1 - \frac{1}{4} \theta^2 \Big) \bigg\} < a_1^6 \\
        &\Big( \nu \Delta t \lambda_1 \Big)
        \bigg\{ \frac{1}{2} \big( 1 - \frac{3}{4} \theta^2 \big)^2
        + \frac{3}{8} \theta^2 \Big( 1 - \frac{1}{4} \theta^2 \Big) \bigg\} < a_1^2 = \frac{\theta (1 - \theta^2)}{2} \\
		&\nu \Delta t \lambda_1
        < \frac{\frac{1}{2} \theta (1 - \theta^2)}{\frac{1}{2} \big( 1 - \frac{3}{4} \theta^2 \big)^2 + \frac{3}{8} \theta^2 \Big( 1 - \frac{1}{4} \theta^2 \Big)}
		= \frac{\theta (1 - \theta^2)}{ \big( 1 - \frac{3}{4} \theta^2 \big)^2 + \frac{3}{4} \theta^2 \Big( 1 - \frac{1}{4} \theta^2 \Big)}
		= \frac{16 \theta (1 - \theta^2)}{ \big( 4 - 3 \theta^2 \big)^2 + 3 \theta^2 \Big( 4 - \theta^2 \Big)}
    \end{align*}
	\begin{align*}
		&\big( 4 - 3 \theta^2 \big)^2 + 3 \theta^2 \Big( 4 - \theta^2 \Big) 
		= 16 - 24 \theta^2 + 9 \theta^4 + 12 \theta^2 - 3 \theta^4
		= 16 - 12 \theta^2 + 6 \theta^4 
		= 2 ( 8 - 6 \theta^2 + 3 \theta^4). 
	\end{align*}
	\normalcolor
\end{confidential}
		Therefore under the time step restriction in \eqref{eq:dt-limit-1}, we have $x < \frac{\nu \Delta t \lambda_1}{4}$ and by \eqref{eq:sol-1-epsi}
		\begin{align}
			\frac{1}{\epsilon} 
			< \frac{2}{\nu \lambda_1 \Delta t} \Big[ \sqrt{B^2 + 4 C h_{22}} + \big(2 h_{22} - B \big) \Big]. 
			\label{eq:1-epsi-bound-1}
		\end{align}
		By \eqref{eq:note-B-C} and \eqref{eq:dt-limit-1} and the fact that $0 < \big(2 h_{22} - B \big) < \sqrt{B^2 + 4 C h_{22}}$
		\begin{align}
			&\sqrt{B^2 + 4 C h_{22}} + \big(2 h_{22} - B \big) 
			\label{eq:1-epsi-bound-numerator} \\
        	=& 2 \sqrt{\Big( \underbrace{b^2 + \frac{\theta^3}{2} - \nu \Delta t \lambda_1 \beta_1^2}_{ > 0} \Big)^2 
        	+ 4 \Big[ \underbrace{\nu \Delta t \lambda_1 \beta_2^2 + \frac{(1+\theta) (2 + \theta - \theta^2)}{8} - a^2}_{>0} \Big] h_{22}} \notag \\
			\leq& 2 \!\!\sqrt{ \!\Big( b^2 \!+\! \frac{\theta^3}{2} \Big)^2 
        	\!+\! 4 \Big[ \frac{\nu \Delta t \lambda_1 \beta_2^2}{2} \!+\! \frac{(1+\theta) (2 + \theta - \theta^2)}{8} \Big] \!
        	\Big[ \underbrace{c^2 \!+\! \frac{(1 - \theta)(2 \!-\! \theta \!-\! \theta^2)}{8} \!-\! \frac{\nu \Delta t \lambda_1 \beta_0^2}{2}}_{ > 0} \Big] } \notag \\
			\leq& 2 \sqrt{ \Big( b^2 + \frac{\theta^3}{2} \Big)^2 
        	+ 4 \Big[ (1 - \theta) \beta_2^2 + \frac{(1+\theta) (2 + \theta - \theta^2)}{8} \Big] 
        	\Big[ c^2 + \frac{(1 - \theta)(2 - \theta - \theta^2)}{8} \Big]}. \notag 
		\end{align}
\begin{confidential}
	\color{darkblue}
	\begin{align*}
        &\sqrt{B^2 + 4 C h_{22}} + \big(2 h_{22} - B \big) \\
        \leq& 2 \sqrt{B^2 + 4 C h_{22}} \\
        =& 2 \sqrt{\Big( \frac{\nu \Delta t \beta_1^2}{2 C_{p}^{2}} - \frac{\theta^3}{2} - b^2 \Big)^2 
        + 4 \Big[ \frac{\nu \Delta t \beta_2^2}{2 C_{p}^{2}} + \frac{(1+\theta) (2 + \theta - \theta^2)}{8} - a^2 \Big] h_{22}} \\
        =& 2 \sqrt{\Big( \underbrace{b^2 + \frac{\theta^3}{2} - \frac{\nu \Delta t \beta_1^2}{2 C_{p}^{2}}}_{ > 0} \Big)^2 
        + 4 \Big[ \underbrace{\frac{\nu \Delta t \beta_2^2}{2 C_{p}^{2}} + \frac{(1+\theta) (2 + \theta - \theta^2)}{8} - a^2}_{>0} \Big] h_{22}} \\
        \leq& 2 \sqrt{ \big( b^2 + \frac{\theta^3}{2} \big)^2 
        + 4 \Big[ \frac{\nu \Delta t \beta_2^2}{2 C_{p}^{2}} + \frac{(1+\theta) (2 + \theta - \theta^2)}{8} \Big]h_{22} } \\
        =& 2 \sqrt{ \big( b^2 + \frac{\theta^3}{2} \big)^2 
        + 4 \Big[ \frac{\nu \Delta t \beta_2^2}{2 C_{p}^{2}} + \frac{(1+\theta) (2 + \theta - \theta^2)}{8} \Big] 
        \Big[ \underbrace{c^2 + \frac{(1 - \theta)(2 - \theta - \theta^2)}{8} - \frac{\nu \Delta t \beta_0^2}{2 C_{p}^{2}}}_{ > 0} \Big] } \\
        =& 2 \sqrt{ \big( b^2 + \frac{\theta^3}{2} \big)^2 
        + 4 \Big[ \frac{\nu \Delta t \beta_2^2}{2 C_{p}^{2}} + \frac{(1+\theta) (2 + \theta - \theta^2)}{8} \Big] 
        \Big[ c^2 + \frac{(1 - \theta)(2 - \theta - \theta^2)}{8} \Big] } \\
        \leq& 2 \sqrt{ \big( b^2 + \frac{\theta^3}{2} \big)^2 
        + 4 \Big[ (1 - \theta) \beta_2^2 + \frac{(1+\theta) (2 + \theta - \theta^2)}{8} \Big] 
        \Big[ c^2 + \frac{(1 - \theta)(2 - \theta - \theta^2)}{8} \Big] } \\
        \leq& C(\theta).
    \end{align*}
	\normalcolor
\end{confidential}
		By algebraic calculation and the time step restriction in \eqref{eq:dt-limit-1}
		\begin{align}
			a^2 + b^2 + c^2 
			=& x - \nu \Delta t \lambda_1 \big( \beta_2 \beta_1 + \beta_1 \beta_0 + \beta_2 \beta_0 \big) + \frac{3}{4} \theta (1 - \theta^2) \label{eq:a-b-c-sq}\\
			\leq& \frac{\nu \Delta t \lambda_1}{4} - \nu \Delta t \lambda_1 \big( \beta_2 \beta_1 + \beta_1 \beta_0 + \beta_2 \beta_0 \big) + \frac{3}{4} \theta (1 - \theta^2) \notag \\
			\leq& \frac{\nu \Delta t \lambda_1}{4} + \frac{\nu \Delta t \lambda_1}{3} (\underbrace{\beta_2 + \beta_1 + \beta_0}_{=1})^2 
			+ \frac{3}{4} \theta (1 - \theta^2)  \notag \\
            \leq& \frac{(1-\theta)}{12} (9 \theta^2 + 9 \theta + 14). \notag  
		\end{align}
\begin{confidential}
    \color{darkblue}
    \begin{align*}
        a^2 + b^2 + c^2 \leq& \frac{7 \nu \Delta t \lambda_1}{12} 
        + \frac{3}{4} \theta (1 - \theta^2) 
        \leq \frac{7 \bcancel{\nu \lambda_1}}{12} \frac{2(1 - \theta)}{\bcancel{\nu \lambda_1}} + \frac{3}{4} \theta (1 - \theta^2) \\  
        =& \frac{7 (1-\theta)}{6} + \frac{3}{4} \theta (1 - \theta^2)
        = \frac{(1-\theta)}{12} \big[ 14 + 9 \theta (1+\theta) \big] 
    \end{align*}
    \normalcolor
\end{confidential}
		We combine \eqref{eq:1-epsi-bound-1}-\eqref{eq:a-b-c-sq} to have 
		$\frac{1}{\epsilon} < \frac{C_{\epsilon}(\theta)}{\nu \lambda_1 \Delta t}$ for some positive constant $C_{\epsilon}(\theta)$ only depending on $\theta$. 
		Then we claim $\epsilon$ has a upper bound: by \eqref{eq:sol-epsi},\eqref{eq:h22-sol}, \eqref{eq:B-C-h22} and \eqref{eq:dt-limit-1}
		\begin{align*}
			\epsilon = \frac{2 \Big( \frac{\nu \Delta t \lambda_1}{2} - x \Big)}{\sqrt{ B^2 + 4 Ch_{22}} + (2h_{22} - B)}
			\leq \frac{\nu \Delta t \lambda_1}{ 2(2h_{22} - B)}
			\leq \frac{\nu \Delta t \lambda_1}{ 4 h_{22}}
			< \frac{\frac{8 \theta (1 - \theta^2)}{( 8 - 6 \theta^2 + 3 \theta^4)}}{ 4 \frac{\theta (1 - \theta)^2 (1 + \theta)(2 + \theta)}{16}} < 4.
		\end{align*}
		Hence we have proved $\frac{1}{\epsilon} > \frac{1}{4}$. 
\begin{confidential}
	\color{darkblue}
	\begin{align*}
		\epsilon =& \frac{\sqrt{B^2 + 4 C h_{22}} - (2 h_{22} - B) }{2 h_{22}} 
		= \frac{B^2 + 4 C h_{22} - (2 h_{22} - B)^2}{2 h_{22} \big[ \sqrt{B^2 + 4 C h_{22}} + (2 h_{22} - B) \big]} \\
		=& \frac{B^2 + 4 C h_{22} - (4 h_{22}^2 - 4Bh_{22} + B^2)}{2 h_{22} \big[ \sqrt{B^2 + 4 C h_{22}} + (2 h_{22} - B) \big]}
		= \frac{\cancel{B^2} + 4 C h_{22} - 4 h_{22}^2 + 4Bh_{22} \cancel{- B^2}}{2 h_{22} \big[ \sqrt{B^2 + 4 C h_{22}} + (2 h_{22} - B) \big]} \\
		=& \frac{4 \cancel{h_{22}}(C + B - h_{22})}{2 \cancel{h_{22}} \big[ \sqrt{B^2 + 4 C h_{22}} + (2 h_{22} - B) \big]}
		= \frac{2(C + B - h_{22})}{\big[ \sqrt{B^2 + 4 C h_{22}} + (2 h_{22} - B) \big]} \\
		\leq& \frac{ \cancel{2}(C + B - h_{22})}{ \cancel{2} (2 h_{22} \underbrace{- B}_{>0})} 
		\leq \frac{C + B - h_{22}}{2 h_{22}}
		= \frac{\frac{\nu \Delta t \lambda_1}{2} - x}{2 h_{22}}
		\leq \frac{\nu \Delta t \lambda_1}{4 h_{22}} \\
		\leq& \frac{\nu \Delta t \lambda_1}{4 \frac{ \theta (1 - \theta)^2 (1 + \theta)(2 + \theta)}{16}}
		= \frac{\frac{8 \theta (1 - \theta^2)}{( 8 - 6 \theta^2 + 3 \theta^4)}}{ \frac{\theta (1 - \theta)^2 (1 + \theta)(2 + \theta)}{4}} 
		= \frac{8 \cancel{\theta (1 - \theta^2)}}{( 8 - 6 \theta^2 + 3 \theta^4)} 
		\frac{4}{ \cancel{\theta (1 - \theta)^2} (\underbrace{1 + \theta}_{>1})(\underbrace{2 + \theta}_{>2})} \\
		=& \frac{32}{ (5 + 3 - 6 \theta^2 + 3 \theta^4) \cdot 1 \cdot 2} 
		= \frac{16}{ 5 + 3 (1 - 2 \theta^2 + \theta^4)}
		= \frac{16}{ \underbrace{5 + 3 (1 - \theta^2)^2}_{>5}} 
		< \frac{16}{5} < 4. 
	\end{align*}
	\normalcolor
\end{confidential}
        Finally by \eqref{eq:h22-sol}, \eqref{eq:h11-sol}, and \eqref{eq:a-b-c-sq}, we have a positive upper bound (only depending on $\theta$) for $h_{11}$ and $h_{22}$, which completes the proof of  \eqref{eq:1-epsi-cond}.
\begin{confidential}
    \color{darkblue}
    \begin{align}
        h_{22} 
        =& c^2 + \frac{(1 - \theta)(2 - \theta - \theta^2)}{8} - \frac{\nu \Delta t \lambda_1 \beta_0^2}{2} 
        \label{eq:g22-ubound} \\
        \leq& \frac{(1-\theta)}{12} \big[ 14 + 9 \theta (1+\theta) \big] 
        + \frac{(1 - \theta)(2 - \theta - \theta^2)}{8} = C(\theta) (>0). \notag  
    \end{align}
    \begin{align*}
        h_{11} =& h_{11}^{+} \\
        =& \frac{1}{2} \bigg\{ \underbrace{- \Big( \frac{\nu \Delta t \lambda_1 \beta_1^2}{2} - \frac{\theta^3}{2} - b^2 \Big)}_{>0} 
        + \sqrt{\Big( \frac{\nu \Delta t \lambda_1 \beta_1^2}{2} - \frac{\theta^3}{2} - b^2 \Big)^2 + \underbrace{4 \Big[ \frac{\nu \Delta t \lambda_1 \beta_2^2}{2} + \frac{(1+\theta) (2 + \theta - \theta^2)}{8} - a^2 \Big]h_{22}}_{>0} } \bigg\} \\
        \leq& \frac{1}{2} \bigg\{ b^2 + \frac{\theta^3}{2} 
        + \sqrt{ \Big( b^2 + \frac{\theta^3}{2} \Big)^2 
        + 4 \Big[ \frac{\nu \Delta t \lambda_1 \beta_2^2}{2} + \frac{(1+\theta) (2 + \theta - \theta^2)}{8} \Big] h_{22} } \bigg\} \\
        \leq& \frac{1}{2} \bigg\{ b^2 + \frac{\theta^3}{2} 
        + \sqrt{ \Big( b^2 + \frac{\theta^3}{2} \Big)^2 
        + 4 \Big[ (1 - \theta) \beta_2^2 + \frac{(1+\theta) (2 + \theta - \theta^2)}{8} \Big]h_{22} } \bigg\} < C(\theta) (>0).
    \end{align*}
    \normalcolor
\end{confidential}
	\end{proof}

	\bibliographystyle{abbrv}
	\bibliography{DLN_Longtime_H1_stability_2026}

\end{document}